\documentclass[12pt]{amsart}

\usepackage{amsfonts, amstext, amsmath, amsthm, amscd, amssymb, enumerate, url}
\usepackage{tikz-cd}
\usepackage{xcolor}
\usepackage{svg}

\usepackage[parfill]{parskip}

\usepackage[colorlinks,citecolor=cyan,linkcolor=magenta]{hyperref}
\usepackage[nameinlink]{cleveref}

\newtheorem{thm}{Theorem}[section]
\newtheorem{cor}[thm]{Corollary}
\newtheorem{lemma}[thm]{Lemma}
\newtheorem{prop}[thm]{Proposition}

\newtheorem*{thm:boundveertet}{\Cref{thm:boundveertet}}
\newtheorem*{thm:vtpAflow}{\Cref{thm:vtpAflow}}

\theoremstyle{definition}
\newtheorem{defn}[thm]{Definition}
\newtheorem{rmk}[thm]{Remark}

\numberwithin{equation}{section}

\renewcommand{\epsilon}{\varepsilon}

\newcommand{\red}{\mathrm{red}}

\begin{document}

\title[Dynamics of veering triangulations]{Dynamics of veering triangulations: infinitesimal components of their flow graphs and applications}

\author{Ian Agol}
\address{University of California, Berkeley \\
    970 Evans Hall \#3840 \\
    Berkeley, CA 94720-3840}
\email{ianagol@math.berkeley.edu}

\author{Chi Cheuk Tsang}
\address{University of California, Berkeley \\
    970 Evans Hall \#3840 \\
    Berkeley, CA 94720-3840}
\email{chicheuk@math.berkeley.edu}

\maketitle

\begin{abstract}
We study the strongly connected components of the flow graph associated to a veering triangulation, and show that the infinitesimal components must be of a certain form, which have to do with subsets of the triangulation which we call `walls'. We show two applications of this knowledge: (1) a fix of a proof in the original paper by the first author which introduced veering triangulations; and (2) an alternate proof that veering triangulations induce pseudo-Anosov flows without perfect fits, which was initially proved by Schleimer and Segerman. 
\end{abstract}


\section{Introduction} \label{sec:intro}
Veering triangulations were introduced by the first author in \cite{Ago11} to study mapping tori of pseudo-Anosov homeomorphisms. In that setting, these are ideal triangulations of the punctured mapping tori that encode a folding sequence of train tracks associated to the pseudo-Anosov homeomorphism. In \cite{Ago11}, veering triangulations were used to prove a quantitative refinement of a result obtained initially by Farb, Leininger, and Margalit (\cite{FLM11}), which states that given a bound on the normalized dilatation of the pseudo-Anosov monodromy, there are only finitely many homeomorphism types of such punctured mapping tori. However, the proof presented in \cite{Ago11} contained a gap: an estimate by Ham and Song (\cite{HS07}) was applied to a matrix which records how weights on the branches of the train tracks distribute under the folding moves, but the estimate only applies to irreducible matrices, and there exists examples for which this fails to be the case. 

Meanwhile, since \cite{Ago11} appeared, the study of veering triangulations has developed in many other directions, see for example \cite{Gue16}, \cite{HIS16}, \cite{HRST11}, \cite{Lan18}. In particular, Gu\'eritaud reconstructed the veering triangulations for punctured mapping tori in \cite{Gue16} in terms of the pseudo-Anosov suspension flow. Among other things, this shows that the veering triangulation is canonically associated to each fibered face of the Thurston unit ball in $H_2$, instead of individual fiberings in the interior of the face. Generalizing this, the first author and Gu\'{e}ritaud showed that if a 3-manifold has a pseudo-Anosov flow without perfect fits, then the manifold obtained by drilling out the singular orbits of the flow admits a veering triangulation (see \cite{LMT21}). 

Recently, Schleimer and Segerman proved the converse of this: if a 3-manifold admits a veering triangulation, then appropriate Dehn fillings of it carry pseudo-Anosov flows without perfect fits. Their construction involves isotoping the stable and unstable branched surfaces in order to form a dynamic pair in the sense of Mosher (\cite{Mos96}). The isotopy is in turn constructed by analyzing how the branched surfaces behave with respect to a canonical decomposition of the 3-manifold into veering solid tori. Details of this construction will appear in \cite{SS22}. We remark that Schleimer and Segerman's construction is in fact part of a big program showing that veering triangulations and pseudo-Anosov flows without perfect fits are in correspondence with each other. In particular, their construction is inverse to that of Agol and Gu\'eritaud in both directions, in a suitable sense. The first three parts of their program are \cite{SS18}, \cite{SS21}, and \cite{SS19}. Also see the introduction of \cite{SS19} for an outline of the whole program. 

This paper was born out of an attempt to reprove Schleimer and Segerman's construction using more direct means. There is a standard way of constructing pseudo-Anosov flows on a 3-manifold starting from an embedded graph with desirable properties. See the concept of templates, introduced by Birman and Williams in \cite{BW83a} and \cite{BW83b} (under the name of `knot-holders'), and the concept of dynamic pairs, introduced by Mosher in \cite{Mos96}. The general strategy is to thicken up the graph using flow boxes, then collapse along the complement to obtain a flow on the whole 3-manifold. Given a veering triangulation, there is a natural candidate to apply this strategy to: the flow graph. (See \Cref{sec:veertri} for definitions of objects associated to veering triangulations.) However, the fact that the flow graph might not be strongly connected presents difficulties in both constructing the flow and analyzing it.

This turns out to be exactly the same issue underlying the gap in \cite{Ago11}. In this paper, we explain a way of addressing this. We show that the problematic infinitesimal components must arise in a certain form, which have to do with subsets of the veering triangulation which we call `walls'. By throwing out these components, we obtain the reduced flow graph, which share many of the same features as the flow graph.

By analyzing this reduced flow graph, we are able to tackle the proof in \cite{Ago11}. The explicit quantitative bound we get is the following:
\begin{thm:boundveertet}

If $M$ is the punctured mapping torus of a pseudo-Anosov homeomorphism $\phi: S_{g,n} \to S_{g,n}$, where the normalized dilatation $\lambda(\phi)^{2g-2+\frac{2}{3}n} \leq P$, then $M$ has a veering triangulation with at most $\frac{P^9-1}{2} (\frac{2 \log P^9}{\log(2 P^{-9}+1)}-1)$ tetrahedra.
\end{thm:boundveertet}

Note that the bound we obtain is worse than that stated in \cite{Ago11} by an exponent of $2+\epsilon$, but there might be room for improvement.

To reprove Schleimer and Segerman's construction, we apply the general strategy outlined above to the reduced flow graph: we thicken up the graph by flow boxes, and collapse along its complement, with the help of the unstable branched surface. When the filling slopes satisfy a necessary condition, it is not difficult to see that the resulting flow is pseudo-Anosov. Furthermore, we can leverage the structure of the unstable branched surface to show that the pseudo-Anosov flow obtained has no perfect fits relative to the `filled' orbits. The precise result is as below, see \Cref{sec:veertri,sec:pAflow} for relevant definitions.

\begin{thm:vtpAflow}
Suppose $M$ admits a veering triangulation. Let $l=(l_i)$ denote the collection of ladderpole curves on each boundary component. Then $M(s)$ admits a transitive pseudo-Anosov flow $\phi$ if $|\langle s,l \rangle| \geq 2$. 

Furthermore, there are closed orbits $c_i$ isotopic to the cores of the filling solid tori, such that each $c_i$ is $|\langle s_i,l_i \rangle|$-pronged, and $\phi$ is without perfect fits relative to $\{c_i\}$.
\end{thm:vtpAflow}

We remark that there are 2 notions of pseudo-Anosov flows common in the 3-manifold topology literature: topological pseudo-Anosov flows and smooth pseudo-Anosov flows. \Cref{thm:vtpAflow} holds for both notions.

We also remark that, in contrast to the work by Schleimer and Segerman, it is not clear whether our construction provides a correspondence between veering triangulations and pseudo-Anosov flows. We pose this as a question more carefully in \Cref{sec:questions}.

The results we discuss in this paper run parallel to those in \cite{LMT21}. The walls in this paper are closely related to what are called AB regions in \cite{LMT21}. In the setting of \cite{LMT21}, one starts with a pseudo-Anosov flow without perfect fits and builds a veering triangulation transverse to the flow. Then it turns out that the flow graph encodes the orbits of the flow, with AB cycles governing the extent of over- and under-counting. In this paper we go in the opposite direction, starting with a veering triangulation and constructing a pseudo-Anosov flow without perfect fits. We end up with a similar conclusion: the flow graph encodes the orbits of the flow, with infinitesimal cycles of walls governing the extent of over- and under-counting. We will explain these connections between this paper and \cite{LMT21} where relevant.

Here is an outline of this paper. In \Cref{sec:veertri}, we review the notion of veering triangulations and related constructions. In \Cref{sec:infcomp}, we study the infinitesimal components of the flow graph and show that they must arise from walls. We also show how to define the reduced flow graph by throwing away these infinitesimal components. In \Cref{sec:finiteness}, we analyze the reduced flow graph of layered veering triangulations to fix the gap in \cite{Ago11} and prove \Cref{thm:boundveertet}. In \Cref{sec:pAflow}, we reprove Schleimer and Segerman's result (\Cref{thm:vtpAflow}) using the reduced flow graph. Finally in \Cref{sec:questions}, we discuss some questions coming out this paper. We remark that \Cref{sec:finiteness,sec:pAflow} are independent of each other, and the reader can jump to either of them after reading \Cref{sec:infcomp}.

{\bf Acknowledgments.} The first author discovered the gap in \cite{Ago11} with the help of the veering triangulation census. We thank Andreas Giannopolous, Saul Schleimer, and Henry Segerman for making the data in the census available. We also thank Saul Schleimer and Henry Segerman for sharing with us an early version of \cite{SS22}. We thank Michael Landry for the support and encouragement throughout this project. We thank Rafael Potrie on MathOverflow for pointing us in the direction of Mario Shannon's thesis, and Mario Shannon for guiding us through the material in \Cref{subsec:smooth}. We thank Anna Parlak and Samuel Taylor for comments on an early version of this paper. We thank the anonymous referee for their suggestions in improving the exposition and clarity of the paper.

The first author was supported by MSRI and the Simonyi Professorship at IAS during part of this project. Both authors are partially supported by a grant from the Simons Foundation \#376200.

{\bf Notational conventions.} Throughout this paper, 
\begin{itemize}
    \item $\overline{M}$ will be an oriented compact 3-manifold with torus boundary components, while $M$ will be the interior of such a manifold. We will sometimes conflate a torus end of $M$ with the corresponding boundary component of $\overline{M}$.
    \item $X \backslash \backslash Y$ will denote the metric completion of $X \backslash Y$ with respect to the induced path metric from $X$. In addition, we will call the components of $X \backslash \backslash Y$ the complementary regions of $Y$ in $X$.
    \item $\widetilde{X}$ will mean a universal cover of $X$, unless otherwise stated.
\end{itemize}

\section{Background: veering triangulations} \label{sec:veertri}

We recall the definition of a veering triangulation.

\begin{defn} \label{defn:tautstructure}
An \textit{ideal tetrahedron} is a tetrahedon with its 4 vertices removed. The removed vertices are called the \textit{ideal vertices}. An \textit{ideal triangulation} of $M$ is a decomposition of $M$ into ideal tetrahedra glued along pairs of faces.

A \textit{taut structure} on an ideal triangulation is a labelling of the dihedral angles by $0$ or $\pi$, such that 
\begin{enumerate}
    \item Each tetrahedron has exactly two dihedral angles labelled $\pi$, and they are opposite to each other.
    \item The angle sum around each edge in the triangulation is $2\pi$.
\end{enumerate}

Intuitively this means that there is a degenerate geometric structure on the triangulation where every tetrahedron is flat.

A \textit{transverse taut structure} is a taut structure along with a coorientation on each face, such that for any edge labelled $0$ in a tetrahedron, exactly one of the faces adjacent to it is cooriented out of the tetrahedron.

A \textit{transverse taut ideal triangulation} is an ideal triangulation with a transverse taut structure.
\end{defn}

We will always take the convention that the face coorientations are pointing upwards in our figures and descriptions.

\begin{defn} \label{defn:veertri}
A \textit{veering structure} on a transverse taut ideal triangulation is a coloring of the edges by red or blue, so that looking at each flat tetrahedron from above, when we go through the 4 outer edges counter-clockwise, starting from an endpoint of the inner edge in front, the edges are colored red, blue, red, blue in that order. We call such a tetrahedron a \textit{veering tetrahedron}.

A \textit{veering triangulation} is a transverse taut ideal triangulation with a veering structure.
\end{defn}

\Cref{fig:veertet} shows a veering tetrahedron in a veering triangulation.

\begin{figure} 
    \centering
    \fontsize{14pt}{14pt}\selectfont
    \resizebox{!}{4cm}{
\begingroup%
  \makeatletter%
  \providecommand\color[2][]{%
    \errmessage{(Inkscape) Color is used for the text in Inkscape, but the package 'color.sty' is not loaded}%
    \renewcommand\color[2][]{}%
  }%
  \providecommand\transparent[1]{%
    \errmessage{(Inkscape) Transparency is used (non-zero) for the text in Inkscape, but the package 'transparent.sty' is not loaded}%
    \renewcommand\transparent[1]{}%
  }%
  \providecommand\rotatebox[2]{#2}%
  \newcommand*\fsize{\dimexpr\f@size pt\relax}%
  \newcommand*\lineheight[1]{\fontsize{\fsize}{#1\fsize}\selectfont}%
  \ifx\svgwidth\undefined%
    \setlength{\unitlength}{325.82965819bp}%
    \ifx\svgscale\undefined%
      \relax%
    \else%
      \setlength{\unitlength}{\unitlength * \real{\svgscale}}%
    \fi%
  \else%
    \setlength{\unitlength}{\svgwidth}%
  \fi%
  \global\let\svgwidth\undefined%
  \global\let\svgscale\undefined%
  \makeatother%
  \begin{picture}(1,0.40856788)%
    \lineheight{1}%
    \setlength\tabcolsep{0pt}%
    \put(0,0){\includegraphics[width=\unitlength,page=1]{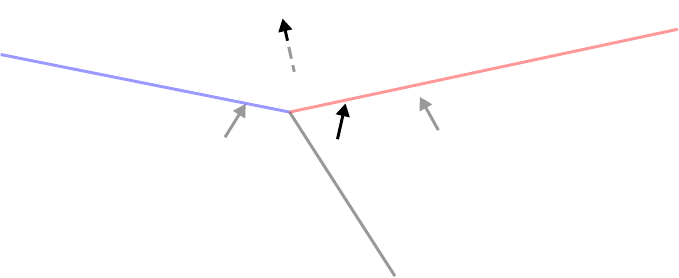}}%
    \put(0.45702311,0.36605622){\color[rgb]{0,0,0}\makebox(0,0)[lt]{\lineheight{1.25}\smash{\begin{tabular}[t]{l}$\pi$\end{tabular}}}}%
    \put(0.50761937,0.13701584){\color[rgb]{0,0,0}\transparent{0.40000001}\makebox(0,0)[lt]{\lineheight{1.25}\smash{\begin{tabular}[t]{l}$\pi$\end{tabular}}}}%
    \put(0.76168485,0.11769163){\color[rgb]{0,0,1}\makebox(0,0)[lt]{\lineheight{1.25}\smash{\begin{tabular}[t]{l}$0$\end{tabular}}}}%
    \put(0.31075696,0.09433859){\color[rgb]{1,0,0}\makebox(0,0)[lt]{\lineheight{1.25}\smash{\begin{tabular}[t]{l}$0$\end{tabular}}}}%
    \put(0.25420744,0.29133508){\color[rgb]{0,0,1}\transparent{0.40000001}\makebox(0,0)[lt]{\lineheight{1.25}\smash{\begin{tabular}[t]{l}$0$\end{tabular}}}}%
    \put(0.62556736,0.30455855){\color[rgb]{1,0,0}\transparent{0.40000001}\makebox(0,0)[lt]{\lineheight{1.25}\smash{\begin{tabular}[t]{l}$0$\end{tabular}}}}%
    \put(0,0){\includegraphics[width=\unitlength,page=2]{veertet.pdf}}%
  \end{picture}%
\endgroup%
}
    \caption{A tetrahedron in a transverse veering triangulation. There are no restrictions on the colors of the top and bottom edges.} 
    \label{fig:veertet}
\end{figure}

From now on $\Delta$ will denote a veering triangulation on $M$.

\begin{rmk} \label{rmk:transverse}
Some authors call the notion we have just defined a transverse veering triangulation, and call the version of this notion without face coorientations veering triangulations instead. These two versions do not differ by much, since one can always take a double cover to make the faces coorientable. However, we will just be studying the version with face coorientations in this paper. Without the face coorientations, our main object of study, the flow graph (\Cref{defn:flowgraph}), cannot be defined.
\end{rmk}

We recall some combinatorial facts and constructions for veering triangulations.

\begin{defn} \label{defn:fantet}
A veering tetrahedron in $\Delta$ is called a \textit{toggle tetrahedron} if the colors on its top and bottom edges differ. It is called a \textit{red/blue fan tetrahedron} if both its top and bottom edges are red/blue respectively.

Note that some authors call toggle and fan tetrahedra hinge and non-hinge respectively.
\end{defn}

\begin{prop}[{\cite[Observation 2.6]{FG13}}] \label{prop:fantet}
Every edge $e$ in $\Delta$ has one tetrahedron above it, one tetrahedron below it, and two stacks of tetrahedra, in between the tetrahedra above and below, on either of its sides.

Each stack must be nonempty. Suppose $e$ is red (blue, respectively). If there is exactly one tetrahedron in one stack, then that tetrahedron is a red (blue, respectively) fan tetrahedron. If there are $n>1$ tetrahedron in one stack, then going from top to bottom in that stack, the tetrahedra are: one toggle tetrahedron, $n-2$ blue (red, respectively) fan tetrahedra, and one toggle tetrahedron. 
\end{prop}

\begin{defn} \label{defn:branchsurf}
For each tetrahedron of $\Delta$, define a branched surface inside it by placing a quadrilateral with vertices on the top and bottom edges and the two side edges of the same color as the top edge, then adding a triangular sector for each side edge of the opposite color to the top edge, with a vertex on that side edge and attached to the quadrilateral along an arc going between the two faces adjacent to that side edge. We also require that the arcs of attachment for the two triangular sectors intersect only once on the quadrilateral. See \Cref{fig:branchsurf} top left. These branched surfaces in each tetrahedron can arranged to match up across faces, thus glue up to a branched surface in $M$, which we call the \textit{unstable branched surface} $B$.

The intersection of the unstable branched surface with the faces of $\Delta$ is called the \textit{unstable train track}. Notice that as one goes from the top faces to the bottom faces of each veering tetrahedron, the unstable train track undergoes a \textit{folding move}. See \Cref{fig:branchsurf} bottom. 

The branch locus of $B$ is a union of circles, smoothly carried by the branch locus and intersecting transversely at \textit{double points} of the branch locus. We call these circles the \textit{components} of the branch locus. Orient the circles so that they intersect the 2-skeleton of the veering triangulation negatively, i.e. they disagree with the coorientation of the faces whenever they meet. This orientation has the special property that at double points of the branch locus, it always points from the side with more sectors to the side with less sectors.
\end{defn}

Note that the unstable branched surface, when considered as a cell complex by taking the 0-cells to be the double points of its branch locus, the 1-skeleton to be the branch locus, and the 2-cells to be the sectors, is dual to the veering triangulation. As such, it makes sense for us to talk about, for example, the sector of $B$ dual to a given edge of $\Delta$.

Also, notice that each sector of $B$ is of the form of a diamond. Each sector has two upper edges and at least two lower edges, at least one on either side; one top vertex, two side vertices, some vertices between the lower edges on either side, and one bottom vertex. We refer to \cite[Section 6.13]{SS19} for details.

\begin{figure} 
    \centering
    \resizebox{!}{8cm}{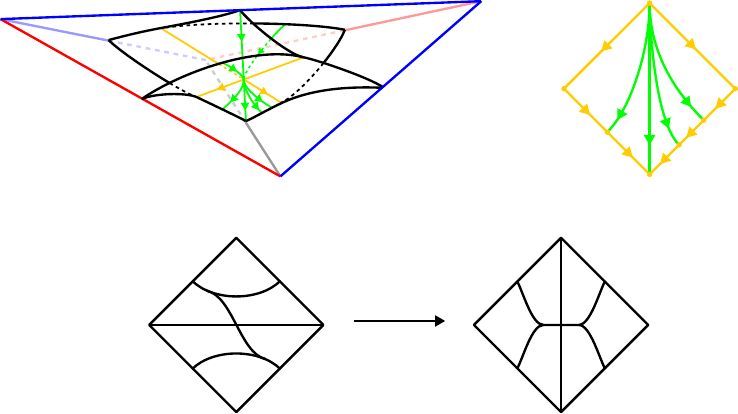}
    \caption{The unstable branched surface meets the faces of a veering tetrahedron in the unstable train track, which undergoes a folding move as one goes from the top faces to the bottom faces. The flow graph can be embedded in the unstable branched surface.}
    \label{fig:branchsurf}
\end{figure}

\begin{defn} \label{defn:flowgraph}
Let $\Delta$ be a veering triangulation. The \textit{flow graph} $\Phi$ is a directed graph with vertex set $V(\Phi)$ equals to the set of edges of $\Delta$, and edge set $E(\Phi)$ defined by adding 3 edges per tetrahedron, going from the top edge and the two side edges of opposite color to the top edge, into the bottom edge.
\end{defn}

We note that the flow graph was first defined in \cite{LMT20}. However the reader is cautioned that there the flow graph is oriented in the opposite direction compared to here.

The flow graph $\Phi$ can be embedded in the unstable branched surface $B$ in the following way. Place each vertex of $\Phi$ at the top vertex of the sector dual to the corresponding edge of $\Delta$, and place the edges of $\Phi$ exiting that vertex in the interior of the sector, as shown in \Cref{fig:branchsurf} top right. See also \cite[Section 4.3]{LMT20}. The flow graph inherits the structure of a (non-generic) oriented train track from this embedding, by making the edges meeting a vertex tangent to a vertical line. We will consider the flow graph to be a subset of $B$ in this way from now on.

The ideal triangulation $\Delta$ on $M$ induces a triangulation $\partial \Delta$ on the torus boundary components of $\overline{M}$, by considering them as links of the ideal vertices. The branch locus of the unstable branched surface inside each tetrahedron opens up towards two opposite ideal vertices, hence each side determines an oriented interval between two edges of a face in $\partial \Delta$. Join these intervals end-to-end. Since each face in $\partial \Delta$ contains at most one such interval, the paths must close up to form oriented parallel loops. Also note that since every edge in $\Delta$ is the top edge of some tetrahedron, each boundary component of $\overline{M}$ will receive at least one loop.

\begin{defn}[{\cite{FG13}}] \label{defn:ladderpole}
The homotopy class of the union of these oriented parallel loops on each boundary component of $\overline{M}$ is called the \textit{ladderpole curve} on the boundary component. The slopes they determine is called the \textit{ladderpole slope} on the boundary component.
\end{defn}

We end this section by analyzing the complementary regions of $B$ in $M$ and the complementary regions of $\Phi$ in $B$.

\begin{prop} \label{prop:complbranchsurf}
The components of $M \backslash \backslash B$ are (once-punctured cusped polygons)$\times S^1$, where each cusp$\times S^1$ represents the ladderpole slope on the corresponding boundary component of $\overline{M}$, and the sum over all cusp circles in each component represents the ladderpole curve on the corresponding boundary component of $\overline{M}$.
\end{prop}
\begin{proof}
Let $T$ be such a component. $T$ is a neighborhood of a torus end of $M$. Inside each tetrahedron $t$, $T \cap t$ is homeomorphic to the product of $T \cap t \cap B$ with $[0,\infty)$. Hence the product structures glue up to give a homeomorphism $T \cong T^2 \times [0,\infty)$. $\partial T$ inherits the branch locus of $B$ as cusp circles, with these representing the ladderpole classes and slopes as described in the statement by definition. An identification of $\partial T$ with $T^2$ sending these cusp circles to longitudes extend to an identification of $T$ with (once-punctured cusped $n$-gon)$\times S^1$, where $n$ is the number of cusp circles. 
\end{proof}

\begin{lemma} \label{lemma:flowgraphcompl}
The components of $B \backslash \backslash \Phi$ are annuli or M\"obius bands with tongues, i.e. they can be obtained by attaching triangular sectors (`tongues') along arcs to a smooth annulus or M\"obius band and smoothing all the arcs of attachment in the same direction with respect to the core of the annulus/M\"obius band. (This terminology is borrowed from \cite{Mos96}.)

Moreover, the arcs of attachment zig-zag along the annulus/M\"obius band. More precisely, the arcs on the annulus/M\"obius band lift to $y=\pm x + 2i, i \in \mathbb{Z}$ in the universal cover $[0,1] \times \mathbb{R}$. These arcs are subintervals of the branch locus of $B$ and are oriented downwards (i.e. decreasing $y$ in the above model). See \Cref{fig:flowgraphcompl} bottom.
\end{lemma}

\begin{figure}
    \centering
    \resizebox{!}{12cm}{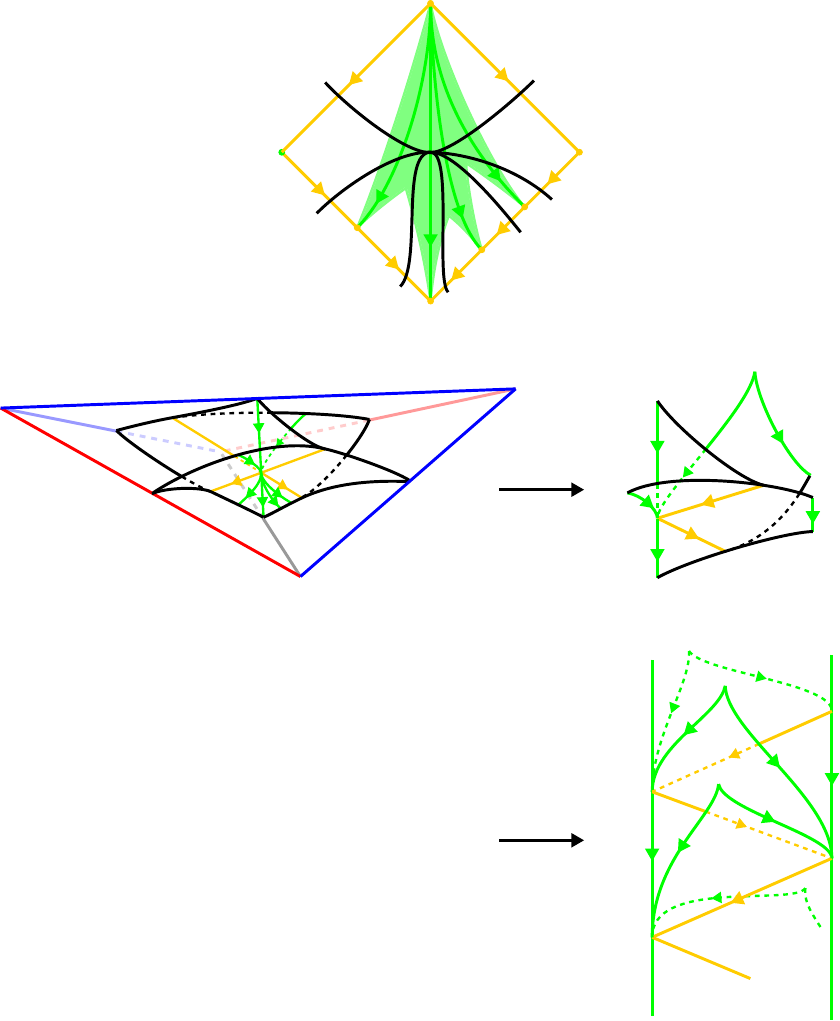}
    \caption{The complementary regions of $\Phi$ in $B$ are annuli/M\"obius bands with tongues.}
    \label{fig:flowgraphcompl}
\end{figure}

\begin{proof}
The topology of the complementary regions is unchanged if we thicken up $\Phi$ to be a regular neighborhood of itself in $B$ within each sector. Hence we can take a neighborhood of $\Phi$ in each sector of $B$ as in \Cref{fig:flowgraphcompl} top, where the black lines in the figure are where the faces of $\Delta$ meet the sector.

With this adjustment in place, it is straightforward to analyze the portion of each complementary region within a veering tetrahedron. There are two of these for each tetrahedron, corresponding to the two side edges of the same color as the top edge. We show the form of these components in \Cref{fig:flowgraphcompl} middle. In particular they intersect the faces of the tetrahedron in train track switches. 

Components of $B \backslash \backslash \Phi$ can be obtained by gluing these pieces along the train track switches on the top and bottom, and the result of the gluing must be an annulus or M\"obius band with tongues with the arcs of attachments of the tongues as described.
\end{proof}

\section{Infinitesimal components of the flow graph} \label{sec:infcomp}

\begin{defn} \label{defn:PF}
A directed graph $G$ is said to be \textit{strongly connected} if for every ordered pair of vertices $(v,w)$ there a directed edge path going from $v$ to $w$.

The \textit{adjacency matrix} of $G$ is defined to be the matrix $A \in Hom(\mathbb{R}^{V(G)}, \mathbb{R}^{V(G)})$ with entries $A_{wv}$=number of edges going from $v$ to $w$. It is easy to see that $G$ being strongly connected is equivalent to its adjacency matrix $A$ being \textit{irreducible}, i.e. for every $v,w$, the entry $(A^n)_{wv}$ is positive for some $n>0$.
\end{defn}

For reasons explained in the introduction, it would be convenient to have a strongly connected flow graph associated to a given veering triangulation. However, a brief search through the veering triangulation census (\cite{GSS}) gives examples which this property fails, for instance: \texttt{eLAkbbcdddhwqj\_2102} (which is layered) and \texttt{fLAMcaccdeejsnaxk\_20010} (which is non-layered). See \cite{GSS} for what these codes mean. For completeness, we also note that there are examples which this property holds: \texttt{cPcbbbdxm\_10} (which is layered), and \texttt{gLLAQbddeeffennmann\_011200} (which is non-layered). As we shall see later, further examples can be constructed by taking covers.

Our goal in this section is to analyze how this property can fail and explain how to prune the flow graph to make a version of this property hold.

\subsection{Strongly connected components} \label{subsec:PFcomp}

We first set up some notation for discussing strong connectivity.

\begin{defn} \label{defn:PFcomp}
Let $G$ be a directed graph, let $v,w$ be vertices of $G$. Write $v \gtrsim w$ if there is a directed edge path going from $v$ to $w$. Write $v \sim w$ if $v \gtrsim w$ and $v \lesssim w$. Note that $\sim$ is an equivalence relation, call equivalence classes of $\sim$ \textit{strongly connected components}.

Construct a directed graph $G/{\sim}$ with vertex set equal to the equivalence classes $[v]$ of $\sim$, and an edge from $[v]$ to $[w]$ if $v \gtrsim w$ and $[v] \neq [w]$. Call a strongly connected component $[v]$ of $G$ an \textit{infinitesimal component} if $[v]$ has an incoming edge in $G/{\sim}$.

A subset $V$ of $V(G)$ is called a \textit{minimal set} if it has the property that if $v \in V$ and $v \gtrsim w$ then $w \in V$.
\end{defn}

It is easy to see that a finite directed graph is a disjoint union of strongly connected graphs if and only if it has no infinitesimal components. Similarly, a finite directed graph is strongly connected if and only if it has no proper minimal sets. In this sense, infinitesimal components, or minimal sets in general, are the obstruction to strong connectivity. Hence our approach to understanding the failure of strong connectivity of a given flow graph is to analyze its infinitesimal components and minimal sets.

\subsection{Walls} \label{subsec:wall}

In this section, we introduce the concept of walls of a veering triangulation. These present one way in which infinitesimal components of flow graphs can arise. What we will show in \Cref{subsec:infcompproof} is that these account for all infinitesimal components of flow graphs. 

\begin{defn} \label{defn:wall}
Let $\Delta$ be a veering triangulation on a 3-manifold $M$. Suppose there are tetrahedra $t_{i,j}, 1 \leq i \leq w+1, j \in \mathbb{Z}/h$, where $w \geq 2$, such that for $2 \leq i \leq w$,

\begin{enumerate}
    \item The bottom edge of $t_{i,j}$ is the top edge of $t_{i,j+1}$
    \item If $i$ is odd, the bottom edge of $t_{i,j}$ are side edges of $t_{i-1,j}$ and $t_{i+1,j}$ and no other tetrahedra in $\Delta$. If $i$ is even, the bottom edge of $t_{i,j}$ are side edges of $t_{i-1,j+1}$ and $t_{i+1,j+1}$ and no other tetrahedra in $\Delta$.
\end{enumerate}

Then we call the indexed multiset $\{t_{i,j}\}$ a wall, and call $w$ the \textit{width} of the wall. We emphasize that in this definition, $t_{i,j}$ for different $i,j$ may not be different tetrahedra in $\Delta$.

Notice that for $w\geq 3$, it is possible to extract a proper subcollection of $\{t_{i,j}\}$ which will form a wall of smaller width. A wall will be called \textit{maximal} if the collection of tetrahedra $\{t_{i,j}\}$ cannot be enlarged into a wall of larger width. Similarly, one can always replace $h$ with a multiple of $h$ by renaming the tetrahedra appropriately. In the sequel we will implicitly assume that $h$ is the minimum possible value that satisfies the definition. 
\end{defn}

Suppose we have a wall $\{t_{i,j}\}$ of a veering triangulation $\Delta$. Notice that for $2 \leq i \leq w$, by \Cref{prop:fantet} and \Cref{defn:flowgraph}, the bottom edge of $t_{i,j}$ has only one outgoing edge in the flow graph $\Phi$, namely the edge going to the bottom edge of $t_{i,j+1}$. Hence there is a cycle $c_i$ in the flow graph passing through the top/bottom edges of $t_{i,j}$, for which there are edges entering $c_i$ but no edges exiting $c_i$. In other words, $c_i$ is an infinitesimal component of the flow graph for each $2 \leq i \leq w$. We call these $c_i$ the \textit{infinitesimal cycles} of the wall.

Meanwhile for each $j$ there is an edge of $\Phi$ going from the bottom edge of $t_{1,j}$ to the bottom edge of $t_{1,j+1}$, since the former is either the top edge of $t_{1,j+1}$ or a side edge of opposite color as the top edge of $t_{1,j+1}$. Let $c_1$ be the cycle formed by these edges. Similarly, there is a cycle $c_{w+1}$ passing through the bottom edges of $t_{w+1,j}$. We call $c_1$ and $c_{w+1}$ the \textit{boundary cycles} of the wall.

Here is a convenient way to visualize a wall, at least for those with width $w \geq 3$. Inside each tetrahedron, there is a unique quadrilateral carried by the unstable branched surface with boundary lying on the unstable train track. If $w \geq 3$, by applying \Cref{prop:fantet} to the bottom edges of $t_{i,j}$ for $2 \leq i \leq w$, we can see that the tetrahedra $t_{i,j}$ are all fan tetrahedra, and they are either all blue fan tetrahedra or all red fan tetrahedra. Hence the quadrilaterals that lie inside them are adjacent to each other in $B$, and their union forms a surface carried by $B$ that is the image of an annulus or M\"obius band under an immersion. The quadrilaterals form a tiling of the annulus or M\"obius band, with 4 quadrilaterals meeting at each vertex, resembling a tiled wall. In fact, this is the reason why we have chosen to call the collection of tetrahedra a wall. In particular, we note that in this case (1) in the definition holds for $i=1,w+1$ as well, and so the infinitesimal and boundary cycles $c_i$ lie within the quadrilateral tiling as vertical loops.

If $w=2$, we can try to repeat the above argument, but the picture is not as clean. By the same argument as above, $t_{1,j}$ and $t_{3,j}$ are fan tetrahedra, but the bottom edge of $t_{1,j}$ may be different from the top edge of $t_{1,j+1}$, and similarly for $t_{3,j}$. As a result, the union of quadrilaterals inside the tetrahedra do not nicely tile up an annulus or M\"obius band, but instead at some vertices some quadrilaterals from the bottom might `peel away'.

We remark that the quadrilaterals we considered above are among the ones considered in \cite{HRST11}, and this idea of looking at how quadrilaterals tile up also appeared there.

In \Cref{fig:width4wall}, we present 3 ways of illustrating a wall of width $4$ in order to aid the reader's intuition. The first picture shows the quadrilateral tiling as mentioned above (\Cref{fig:width4wall} top left). The second picture shows a layered view in terms of a folding sequence of the unstable train track on the faces of tetrahedra in the wall (\Cref{fig:width4wall} right). The third picture shows the portion of the unstable branched surface in a small neighborhood of the wall (\Cref{fig:width4wall} bottom left). The flow graph contains vertical lines as subgraphs in the picture. These are the infinitesimal and boundary cycles of the wall.

\begin{figure}
    \centering
    \fontsize{14pt}{14pt}\selectfont
    \resizebox{!}{10cm}{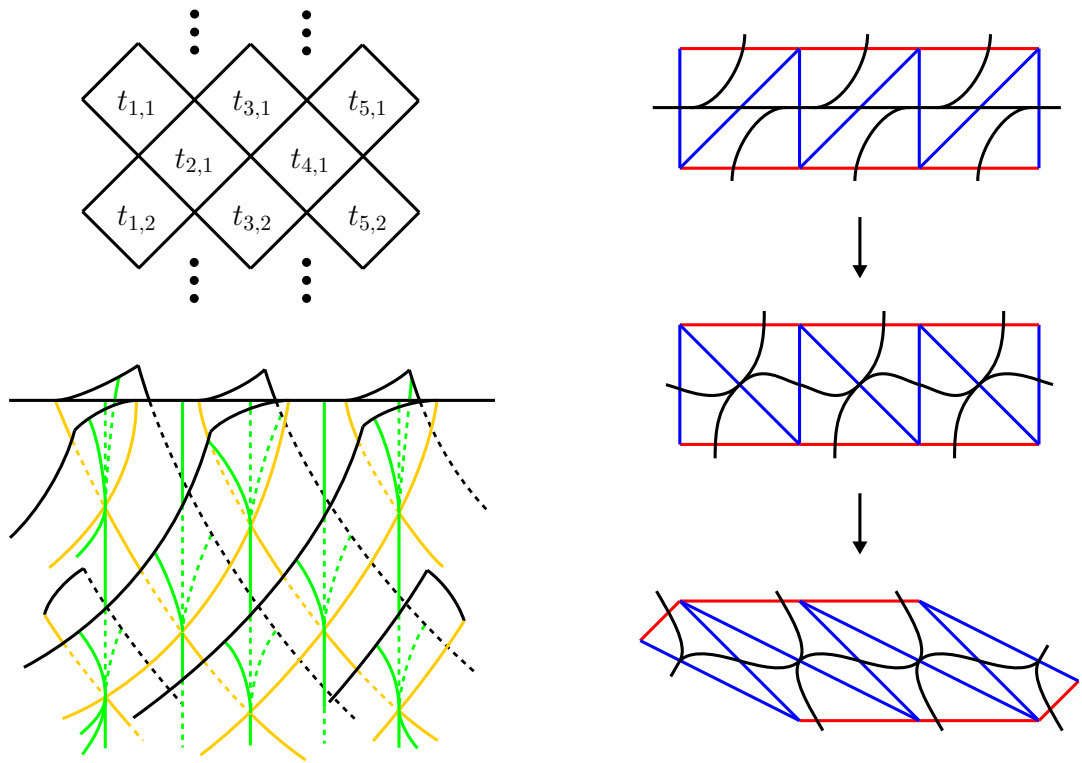}
    \caption{Illustrating a wall of width $4$ from three viewpoints. Top left: tiling with quadrilaterals. Right: folding sequence of the unstable train track. Bottom left: unstable branched surface in a small neighborhood.}
    \label{fig:width4wall}
\end{figure}

In \Cref{fig:width2wall}, we also demonstrate the same 3 viewpoints for a wall of width $2$, since as pointed out before, the combinatorics for width $2$ walls are slightly more general than that for higher width walls. 

\begin{figure}
    \centering
    \fontsize{14pt}{14pt}\selectfont
    \resizebox{!}{12cm}{
\begingroup%
  \makeatletter%
  \providecommand\color[2][]{%
    \errmessage{(Inkscape) Color is used for the text in Inkscape, but the package 'color.sty' is not loaded}%
    \renewcommand\color[2][]{}%
  }%
  \providecommand\transparent[1]{%
    \errmessage{(Inkscape) Transparency is used (non-zero) for the text in Inkscape, but the package 'transparent.sty' is not loaded}%
    \renewcommand\transparent[1]{}%
  }%
  \providecommand\rotatebox[2]{#2}%
  \newcommand*\fsize{\dimexpr\f@size pt\relax}%
  \newcommand*\lineheight[1]{\fontsize{\fsize}{#1\fsize}\selectfont}%
  \ifx\svgwidth\undefined%
    \setlength{\unitlength}{471.38673215bp}%
    \ifx\svgscale\undefined%
      \relax%
    \else%
      \setlength{\unitlength}{\unitlength * \real{\svgscale}}%
    \fi%
  \else%
    \setlength{\unitlength}{\svgwidth}%
  \fi%
  \global\let\svgwidth\undefined%
  \global\let\svgscale\undefined%
  \makeatother%
  \begin{picture}(1,0.98826209)%
    \lineheight{1}%
    \setlength\tabcolsep{0pt}%
    \put(0,0){\includegraphics[width=\unitlength,page=1]{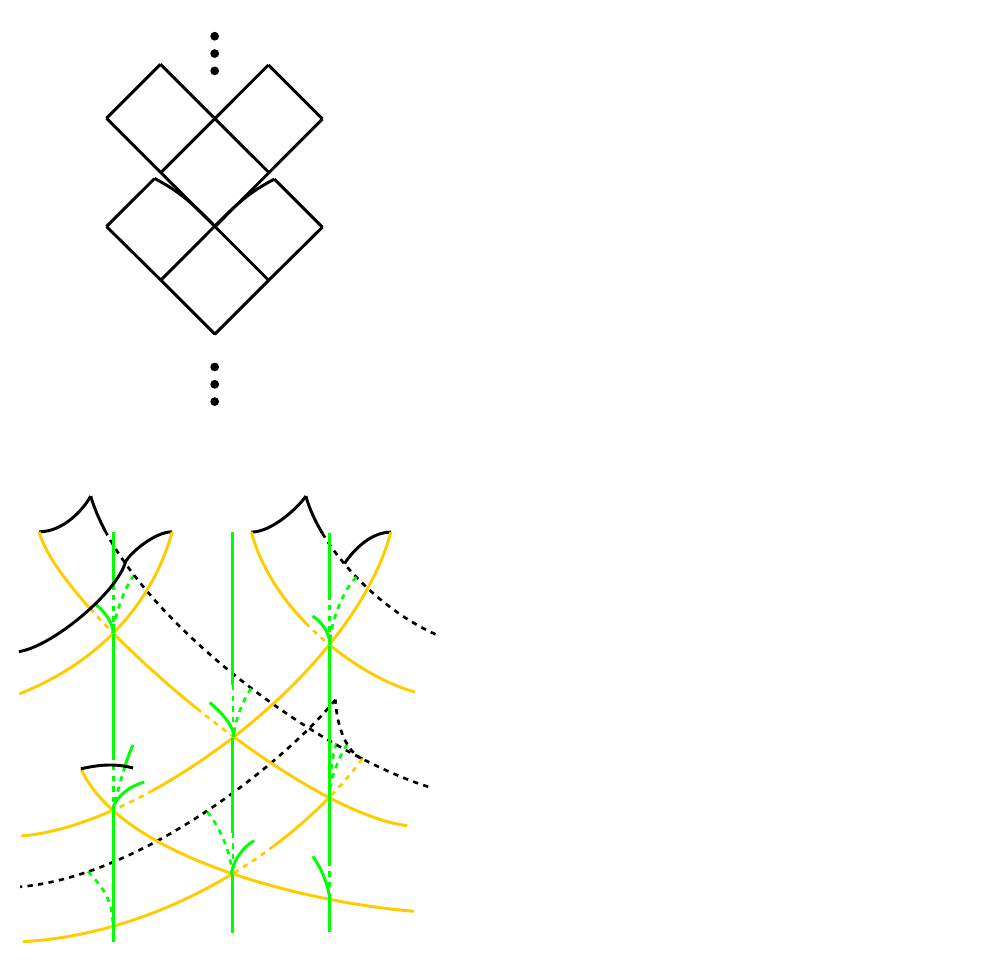}}%
    \put(0.14317274,0.85692331){\color[rgb]{0,0,0}\makebox(0,0)[lt]{\lineheight{1.25}\smash{\begin{tabular}[t]{l}$t_{1,1}$\end{tabular}}}}%
    \put(0.25622503,0.8569253){\color[rgb]{0,0,0}\makebox(0,0)[lt]{\lineheight{1.25}\smash{\begin{tabular}[t]{l}$t_{3,1}$\end{tabular}}}}%
    \put(0.14317274,0.74387303){\color[rgb]{0,0,0}\makebox(0,0)[lt]{\lineheight{1.25}\smash{\begin{tabular}[t]{l}$t_{1,2}$\end{tabular}}}}%
    \put(0.25622503,0.74387505){\color[rgb]{0,0,0}\makebox(0,0)[lt]{\lineheight{1.25}\smash{\begin{tabular}[t]{l}$t_{3,2}$\end{tabular}}}}%
    \put(0.19969888,0.80039917){\color[rgb]{0,0,0}\makebox(0,0)[lt]{\lineheight{1.25}\smash{\begin{tabular}[t]{l}$t_{2,1}$\end{tabular}}}}%
    \put(0.19969836,0.68854957){\color[rgb]{0,0,0}\makebox(0,0)[lt]{\lineheight{1.25}\smash{\begin{tabular}[t]{l}$t_{2,2}$\end{tabular}}}}%
    \put(0,0){\includegraphics[width=\unitlength,page=2]{width2wall.pdf}}%
  \end{picture}%
\endgroup%
}
    \caption{Illustrating a wall of width $2$ from the same three viewpoints as \Cref{fig:width4wall}.}
    \label{fig:width2wall}
\end{figure}

We caution the reader that $c_i$ for different $i$ may be the same cycle in $\Phi$. In fact, call a wall \textit{twisted} if $t_{i,j}=t_{w+2-i,j+h'}$ for some $h'$, and call a wall \textit{untwisted} otherwise. For a twisted wall, $c_i$ and $c_{w+2-i}$ are the same cycle for each $i$; for a untwisted wall, $c_i$ are distinct cycles. Equivalently, a wall is twisted if and only if the surface tiled by the quadrilaterals is homeomorphic to a M\"obius band.

Another point of caution is that two distinct maximal walls $\{t_{i,j}\}, \{t'_{i,j}\}$ can share a tetrahedron, say $t_{i,j}=t'_{i',j'}$. By maximality, this is only possible if $i=1$ or $w+1$ and $i'=1$ or $w'+1$, where $w$ and $w'$ are the widths of the walls respectively. In this case, the appropriate boundary cycles of the two walls touch within the shared tetrahedron. In fact, this behaviour can happen within a single wall as well, that is, a tetrahedron can appear as $t_{1,j}$ or as $t_{w+1,j}$ for more than one value of $j$. However, a tetrahedron can only appear at most twice in walls, corresponding to the $2$ side edges of the same color as the top edge. In particular a vertex of $\Phi$ can lie in at most $2$ boundary cycles.

\subsection{Classification of infinitesimal components} \label{subsec:infcompproof}

We begin our analysis of infinitesimal components of flow graphs. Throughout this subsection, fix a minimal set $V$ of the flow graph $\Phi$ associated to a veering triangulation $\Delta$. Recall that vertices of $\Phi$ are edges of $\Delta$, so it makes sense to say whether $e$ is in $V$ for an edge $e$ of $\Delta$.

Each branch of the unstable train track $\tau$ in a face of $\Delta$ is dual to some edge of that face. Let $\tau'$ be the union of those branches dual to edges of $\Delta$ that lie in $V$. Up to rotation, there are 5 configurations for the portion of $\tau'$ lying on a face of $\Delta$. We show and name these 5 types of faces in \Cref{fig:facetype}.

\begin{figure}
    \centering
        \fontsize{14pt}{14pt}\selectfont
    \resizebox{!}{8cm}{
\begingroup%
  \makeatletter%
  \providecommand\color[2][]{%
    \errmessage{(Inkscape) Color is used for the text in Inkscape, but the package 'color.sty' is not loaded}%
    \renewcommand\color[2][]{}%
  }%
  \providecommand\transparent[1]{%
    \errmessage{(Inkscape) Transparency is used (non-zero) for the text in Inkscape, but the package 'transparent.sty' is not loaded}%
    \renewcommand\transparent[1]{}%
  }%
  \providecommand\rotatebox[2]{#2}%
  \newcommand*\fsize{\dimexpr\f@size pt\relax}%
  \newcommand*\lineheight[1]{\fontsize{\fsize}{#1\fsize}\selectfont}%
  \ifx\svgwidth\undefined%
    \setlength{\unitlength}{285.62160219bp}%
    \ifx\svgscale\undefined%
      \relax%
    \else%
      \setlength{\unitlength}{\unitlength * \real{\svgscale}}%
    \fi%
  \else%
    \setlength{\unitlength}{\svgwidth}%
  \fi%
  \global\let\svgwidth\undefined%
  \global\let\svgscale\undefined%
  \makeatother%
  \begin{picture}(1,0.97347657)%
    \lineheight{1}%
    \setlength\tabcolsep{0pt}%
    \put(0,0){\includegraphics[width=\unitlength,page=1]{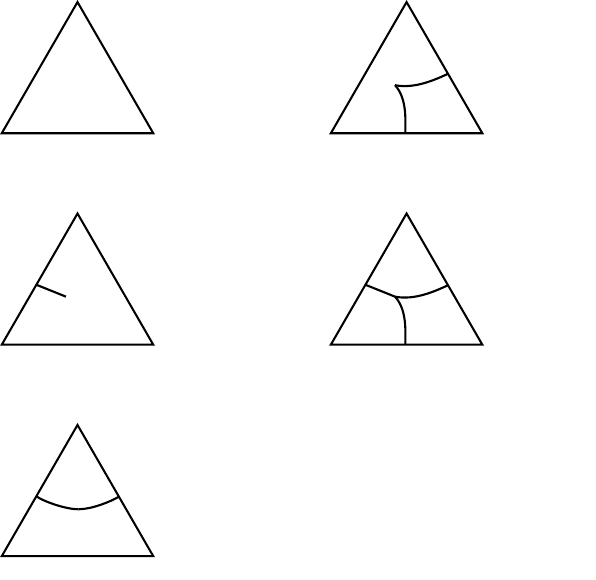}}%
    \put(0.28030729,0.36603158){\color[rgb]{0,0,0}\makebox(0,0)[lt]{\lineheight{1.25}\smash{\begin{tabular}[t]{l}(i)solated\end{tabular}}}}%
    \put(0.28030729,0.72140732){\color[rgb]{0,0,0}\makebox(0,0)[lt]{\lineheight{1.25}\smash{\begin{tabular}[t]{l}(b)lank\end{tabular}}}}%
    \put(0.28030729,0.01090647){\color[rgb]{0,0,0}\makebox(0,0)[lt]{\lineheight{1.25}\smash{\begin{tabular}[t]{l}(e)dge\end{tabular}}}}%
    \put(0.83333472,0.72140407){\color[rgb]{0,0,0}\makebox(0,0)[lt]{\lineheight{1.25}\smash{\begin{tabular}[t]{l}(c)usp\end{tabular}}}}%
    \put(0.83333472,0.36603385){\color[rgb]{0,0,0}\makebox(0,0)[lt]{\lineheight{1.25}\smash{\begin{tabular}[t]{l}(f)ull\end{tabular}}}}%
  \end{picture}%
\endgroup%
}
    \caption{The 5 possible types of faces, labelled by letters b,i,e,c,f.}
    \label{fig:facetype}
\end{figure}

We also consider the configurations for the portion of $\tau'$ lying on the boundary of a tetrahedron. Since $V$ is minimal, and since the flow graph $\Phi$ contains the cycles formed by edges going from the top edge to the bottom edge of each tetrahedron, the top edge of a tetrahedron lies in $V$ if and only if the bottom edge lies in $V$. Again by minimality, if a side edge of opposite color to the top edge is in $V$, then the bottom edge is in $V$, hence the top edge will be in $V$ as reasoned above. From these restrictions, we can enumerate 13 configurations for the portion of $\tau'$ on the boundary of a tetrahedron, up to rotation and reflection. We show and name these 13 types of tetrahedra in \Cref{fig:tettype}.

\begin{figure}
    \centering
        \fontsize{14pt}{14pt}\selectfont
    \resizebox{!}{18cm}{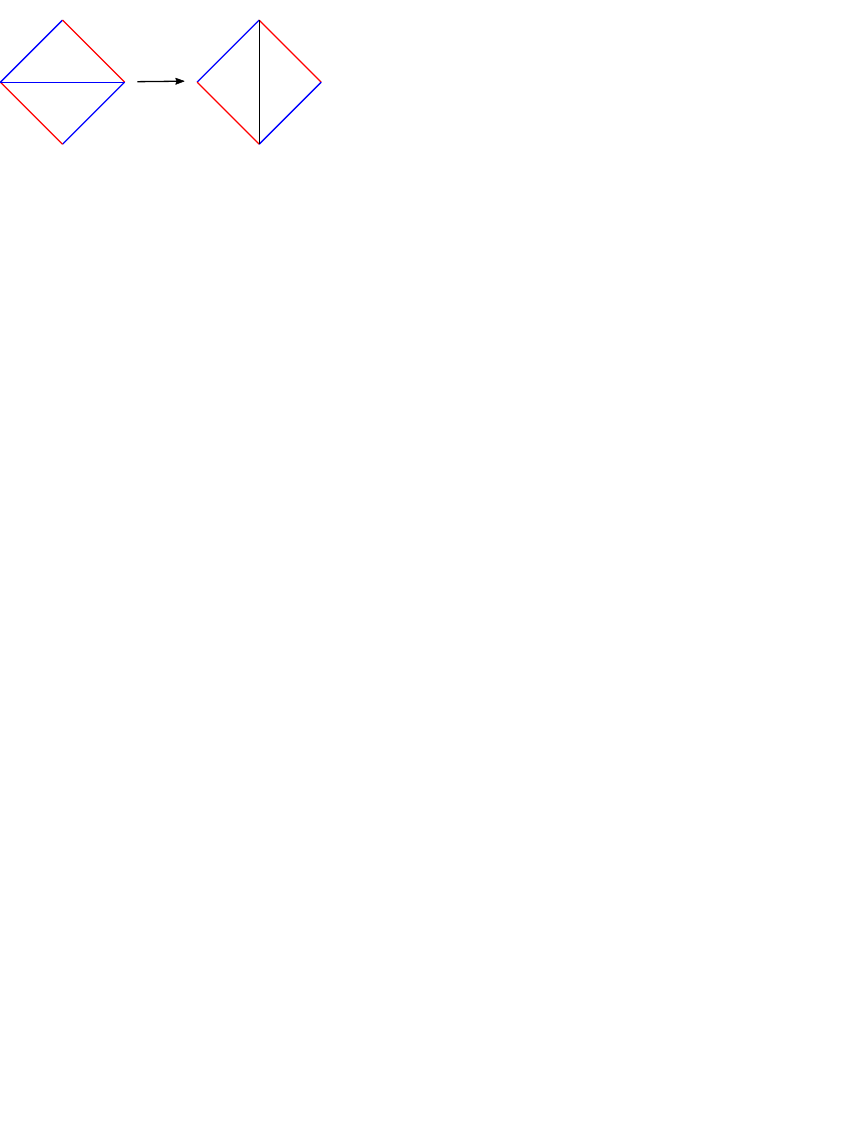}
    \caption{The 13 possible types of tetrahedra, labelled by Roman numerals 0-XII. We drew the tetrahedra by their top faces followed by their bottom faces. The tetrahedron of type (VI)-(X) inside the box are eliminated by \Cref{lemma:tettypeelim}.}
    \label{fig:tettype}
\end{figure}

\begin{lemma} \label{lemma:tettypeelim}
Tetrahedra of type (VI)-(X) will not appear.
\end{lemma}
\begin{proof}
We use a double-counting argument. $$\sum \text{\# type (c) upper faces} - \text{\# type (c) lower faces} =0$$ where the sum is over all tetrahedra of $\Delta$, since each face belongs to the upper face of exactly one tetrahedron and the lower face of exactly one tetrahedron. By inspection of \Cref{fig:tettype}, the number of type (c) upper faces is always greater or equal to the number of type (c) lower faces for each tetrahedron; for type (VI),(VIII),(IX),(X) tetrahedra, strict inequality holds, so these in fact cannot occur. 

A similar argument using type (i) faces eliminates the possiblity of type (VII) tetrahedra (now that type (VIII) tetrahedra are eliminated).
\end{proof}

We now analyze what happens if we have a type (i) face. Let $f$ be such a face, and $e$ be the unique edge of $f$ that lies in $V$. 

$e$ is the top edge of one tetrahedron, the bottom edge of one tetrahedron, and the side edges of two stacks of tetrahedra, one stack on each side of $e$. The face $f$ determines a side of $e$, and we claim that the stack of tetrahedra on that side consists of one tetrahedron only. 

Suppose otherwise. Then the top edge of the tetrahedron on the bottom of the stack is of opposite color to $e$ (by \Cref{prop:fantet}). Looking at \Cref{fig:tettype}, we claim that this tetrahedron must be of type (XI) or (XII). This is because these are the only types of tetrahedron with a side edge in $V$ that is of opposite color to the top edge. As a consequence, the face of this tetrahedron which meet $e$ must be of type (f). But again by looking at \Cref{fig:tettype}, the face $f$ being of type (i) forces the face below it and sharing the edge $e$ to be of type (i) too (and the tetrahedron between them being of type (I) or (II)). In fact, the same is true for the face above $f$ and sharing the edge $e$. Hence inducting upwards and downwards, the faces on the side of $f$ that are adjacent to $e$ must all be of type (i). This gives us a contradiction at the bottom of the stack of tetrahedra.

The tetrahedra on the top and bottom of $e$ have to be of type (III) or (IV). In fact, they are either both of type (III) or both of type (IV), since by the reasoning above, a type (i) face on a side of $e$ forces all of the faces in that stack with edge $e$ to be of type (i).

If the tetrahedra on the top and bottom of $e$ are of type (III), the stack on the side of $e$ opposite to $f$ also consists of one tetrahedron only by the reasoning two paragraphs above. In particular, $e$ has degree $4$ in this case. Note that there is still the freedom of the side tetrahedra being of type (I) or (II).

At this point, we recall the quadrilaterals that we considered in \Cref{subsec:wall}. By looking at the surface formed by the union of the quadrilaterals in the tetrahedra we have considered so far, this gives us a good way of keeping track of our argument. For example, in the situation of the paragraph above, we have $4$ tetrahedra that share an edge $e$, hence we have 4 quadrilaterals sharing a vertex at $e$, see \Cref{fig:wallarg1} left.

\begin{figure}
    \centering
    \fontsize{14pt}{14pt}\selectfont
    \resizebox{!}{4cm}{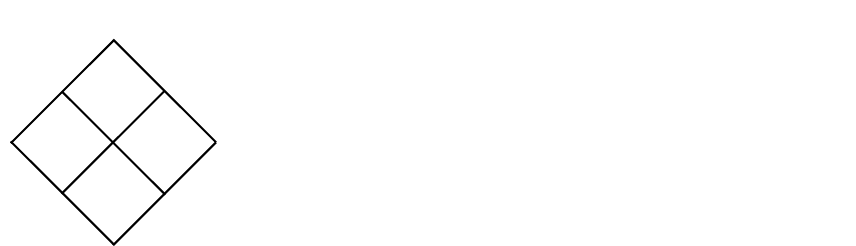}
    \caption{The types of the tetrahedra around $e$ split into three cases. The Roman numeral inside a quadrilateral is the type of the tetrahedron the quadrilateral belongs to.}
    \label{fig:wallarg1}
\end{figure}

Now, our analysis can be extended in the vertical direction. The top edge of the tetrahedron above $e$ is adjacent to a type (i) face. So we can repeat our arguments on there. This gives a group of 4 tetrahedra centered around that edge. In terms of the quadrilaterals, we are extending the tiling vertically. See \Cref{fig:wallarg2} left. A caveat of this picture, however, is that the top and bottom vertices of the quadrilaterals representing type (I) and (II) tetrahedra may not match up, since they might be representing different edges of the adjacent type (III) tetrahedron. See, for example, the situation of $t_{1,1}$ and $t_{1,2}$ in \Cref{fig:width2wall}. We denote this by making the quadrilaterals peel away slightly when the vertices do not match up. Since there are only finitely many tetrahedra, the vertical extension must loop back on itself eventually, giving us a width $2$ wall whose infinitesimal cycle lies in $V$.

We still need to tackle the case when the tetrahedra on the top and bottom of $e$ are of type (IV). We claim that the stack on the side of $e$ opposite to $f$ also consists of one tetrahedron only. Suppose otherwise, then arguing as before, the tetrahedron on the bottom of the stack is of type (XI) or (XII) and its faces that meet $e$ are of type (f). But we know that the face at the bottom of the stack that meets $e$ is of type (e), so this is not the case. 

Looking at \Cref{fig:tettype}, this single tetrahedron must be of type (IV) or (V), since these are the only types of tetrahedra that have a side edge in $V$ meeting two type (e) faces. If the side tetrahedron is of type (IV), it has two type (i) faces, and the tetrahedra sharing those type (i) faces must be of type (I) or (II). In particular, we get a cluster of 6 tetrahedra. The corresponding 6 quadrilateral tile up a region as shown in \Cref{fig:wallarg1} center.

Now repeat the argument vertically as in the last case. We obtain a width $3$ wall eventually, for which the infinitesimal cycles lie in $V$.

Finally, we tackle the case where $e$ has type (IV) tetrahedra above and below and a type (V) tetrahedron on the side opposite to $f$. See \Cref{fig:wallarg1} right. Let $e'$ be the top edge of this type (V) tetrahedron. By repeating our arguments up to this point vertically, we know that the tetrahedron on top of $e'$ is of type (IV) or (V). It is not of type (IV). Otherwise $e'$ has a side where all the faces are of type (i), contradicting the fact that the type (V) tetrahedron below it has only type (e) faces. So the sides of $e'$ have type (e) faces on the top and bottom of their stacks of tetrahedra. We have argued that this implies the stacks of tetrahdra on the sides of $e'$ have one tetrahedron each, and that these side tetrahedra are of type (IV) or (V). The same argument also applies to the bottom edge of the original type (V) tetrahedron. In terms of our picture with quadrilaterals, this extends the tiling horizontally.

Now continue the analysis inductively, both horizontally and vertically. Again, the vertical extension must loop back on itself at some point. This implies that the horizontal extension has to stop at some point, otherwise it will loop back on itself and we will get a Klein bottle or a torus tiled by quadrilaterals from type (V) tetrahedra, which contradicts $\Delta$ having a strict angle structure (see \cite[Theorem 1.5 and Corollary 3.11]{HRST11}). The horizontal extension stops by hitting a column of type (IV) tetrahedra, after which the tiling is `capped off' by type (I) or (II) tetrahedra. See \Cref{fig:wallarg2} right. Hence in every case, we get a wall with infinitesimal cycles contained in $V$.

\begin{figure}
    \centering
    \fontsize{14pt}{14pt}\selectfont
    \resizebox{!}{6cm}{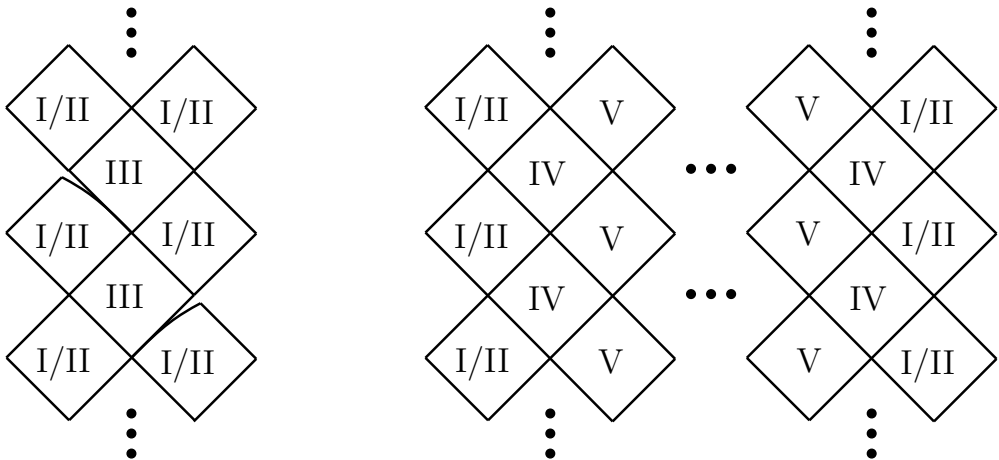}
    \caption{The patterns of quadrilaterals that can occur in our proof. Left: a width $2$ wall. Right: a width $\geq 3$ wall.}
    \label{fig:wallarg2}
\end{figure}

We claim that if $V$ is a proper subset of $V(\Phi)$, it consists solely of these infinitesimal cycles of walls. Suppose otherwise. Since these infinitesimal cycles have no incoming edges from vertices in $V$, we can remove all of them from $V$ and we would still have a non-empty minimal set. Rename $V$ as this new minimal set, and reclassify the faces and tetrahedra of $\Delta$ according to \Cref{fig:facetype,fig:tettype} with respect to this new minimal set.

After this reclassification, there are no type (i) faces anymore, since those must be contained in walls by our analysis above, and we have eliminated all strongly connected components of $V$ within walls. In particular, we can only have tetrahedra of type (0), (V), (XI), and (XII).

Suppose we have a type (e) face $f$. Let $e$ be the unique edge of $f$ not lying in $V$. Consider the stack of tetrahedra on the side of $e$ determined by $f$. Looking at \Cref{fig:tettype}, with the limited types of tetrahedra we can have, as we move downwards the stack starting from $f$, we can only get to faces of type (e). So the stack must end on a type (e) face, but then the tetrahedron directly below this face, which is also the tetrahedron below $e$, cannot be any of the types in \Cref{fig:tettype}, since none of those have a type (e) top face and have a top edge that does not lie in $V$. This contradiction shows that there are no type (e) faces anymore.

And so the only types of tetrahedra we can have are type (0) and (XII). In this case, it is clear that a tetrahedron sharing a face with a type (0) or a type (XII) tetrahedron must be of type (0) or type (XII) itself, respectively. Because $M$ is connected, $V$ is either empty or the entirety of $V(\Phi)$. This proves the claim and concludes our analysis of minimal sets of $\Phi$, which we summarize as:

\begin{thm} \label{thm:infinitesimal}
For a given veering triangulation $\Delta$, the infinitesimal components of its flow graph $\Phi$ are exactly the infinitesimal cycles of walls. Furthermore, $\Phi/{\sim}$ is a rooted height 1 tree, i.e. it has a unique vertex $v_0$ such that there is a directed edge from $v_0$ to every other vertex.
\end{thm}
\begin{proof}
We showed that any proper minimal set of $\Phi$ is a disjoint union of infinitesimal cycles of walls. This implies that every proper minimal set of $\Phi/{\sim}$ is a disjoint union of vertices. Furthermore, each infinitesimal cycle of a wall has incoming edges, hence each vertex of $\Phi/{\sim}$ which is an infinitesimal component has an incoming edge. This is enough to imply the second statement of the theorem.
\end{proof}

We use this theorem to justify our statement earlier that for a finite cover $\widetilde{\Delta}$ of a veering triangulation $\Delta$, $\widetilde{\Delta}$ has a strongly connected flow graph if and only if $\Delta$ has a strongly connected flow graph. The flow graph of $\widetilde{\Delta}$ is a covering of that of $\Delta$, and so the forward implication is easy. For the converse, if the flow graph of $\widetilde{\Delta}$ is not strongly connected, it contains infinitesimal cycles of a wall. Such a wall projects down to a wall of $\Delta$ by \Cref{defn:wall}, thus $\Delta$ contains infinitesimal cycles as well.

\begin{rmk} \label{rmk:LMTABwalls}
With this theorem, we can also discuss how walls are related to the material in \cite{LMT21}.

We first recall some terminology from \cite{LMT21}. The \textit{dual graph} to a veering triangulation $\Delta$ is defined to be the 1-skeleton, i.e. the branch locus, of the unstable branched surface $B$, with the edges oriented by the coorientations on $\Delta^{(2)}$. An \textit{AB cycle} is then defined to be a loop carried by the dual graph which only makes \textit{anti-branching turns}, i.e. does not follow components of the branch locus at each vertex of $B$. The lift of an AB cycle to the universal cover $\widetilde{M}$ determines a properly embedded plane carried by $\widetilde{B}$, by taking the union over all sectors lying below the lifted AB cycle. Such a plane is called a \textit{dynamic plane}. The total number of lifted AB cycles in a dynamic plane is defined to be the \textit{width} of the dynamic plane. If there are $2$ or more such lifted AB cycles, the region bounded in between them is called an \textit{AB region}. For more details, we refer the reader to \cite[Section 2.2, 3.1]{LMT21}.

Now one can count that there are $\lceil \frac{w}{2} \rceil$ or $w$ adjacent AB cycles within a wall of width $w$, depending on whether the wall is twisted or not. These lift to $w$ lifted AB cycles in a dynamic plane, which bound AB regions in between. Conversely, if in the universal cover $\widetilde{B}$ carries a dynamic plane containing an AB region, then the portion of the flow graph within the AB region carries a line which only has incoming edges, hence quotients to an infinitesimal cycle of a wall.
\end{rmk}

\subsection{Reduced flow graph}

Given \Cref{thm:infinitesimal}, a way of arranging for strong connectivity is to simply throw away all the infinitesimal cycles.

\begin{defn} \label{defn:reducedflowgraph}
The \textit{reduced flow graph} $\Phi_{\red}$ of a veering triangulation is the flow graph with all infinitesimal cycles in walls and the edges that enter the cycles removed.
\end{defn}

Thus by \Cref{thm:infinitesimal}, $\Phi_{\red}$ is strongly connected. Because of the simple nature of infinitesimal cycles, $\Phi_{\red}$ inherits some of the properties of $\Phi$. We end this section with two examples of this, which will be useful in \Cref{sec:pAflow}.

\begin{lemma} \label{lemma:flowgraphsorbits}
For every cycle $c$ of $\Phi$, $c$ or $c^2$ is homotopic to a cycle of $\Phi_{\red}$ in $M$.
\end{lemma}
\begin{proof}
The cycles of $\Phi$ are those of $\Phi_{\red}$ and the infinitesimal cycles of walls in $\Phi$. Each infinitesimal cycle of a wall is homotopic (isotopic, even) to a boundary cycle in $M$, unless the wall has even width and is twisted, in which case one has to double the infinitesimal cycle $c_{\frac{w}{2}}$ before it is homotopic to a boundary cycle in $M$.
\end{proof}

\begin{lemma} \label{lemma:redflowgraphcompl}
The components of $B \backslash \backslash \Phi_{\red}$ are annulus or M\"obius bands with tongues.

Moreover, the arcs of attachment of the triangular sectors criss-cross along the annulus/M\"obius band. More precisely, the arcs on the annulus/M\"obius band lift to $y=\pm x + \frac{2i}{w}, i \in \mathbb{Z}$ in the universal cover $[0,1] \times \mathbb{R}$, for some $w \geq 1$. These arcs are subintervals of the branch locus of $B$ and are oriented downwards (i.e. decreasing $y$ in the above model). See \Cref{fig:redflowgraphcompl}.
\end{lemma}

\begin{figure}
    \centering
    \resizebox{!}{12cm}{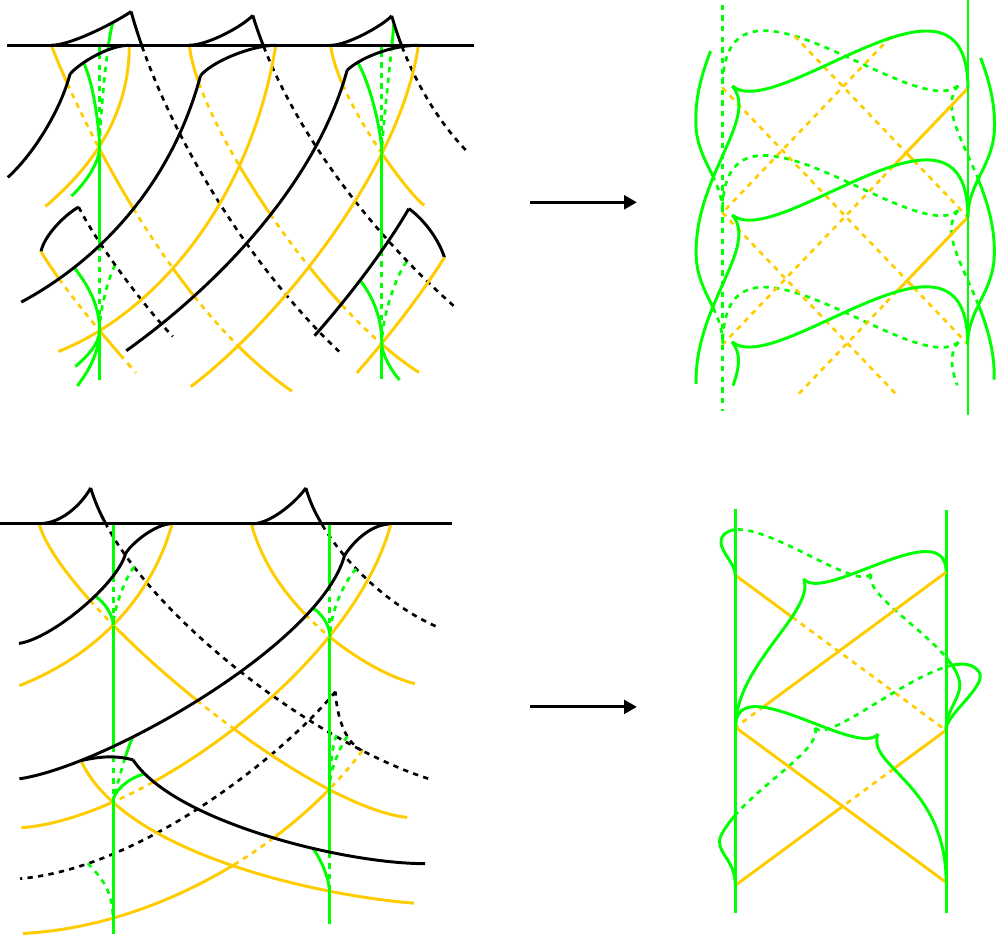}
    \caption{The complementary regions of $\Phi_{\red}$ in $B$ are annuli/M\"obius bands with tongues. The attaching arcs form a criss-cross pattern on the annulus/M\"obius band. Here we show the situation in the neighborhood of the width $4$ wall from \Cref{fig:width4wall} (top) and the width $2$ wall from \Cref{fig:width2wall} (bottom).}
    \label{fig:redflowgraphcompl}
\end{figure}

\begin{proof}
By \Cref{lemma:flowgraphcompl}, it suffices to look at the components that contain infinitesimal cycles of some wall. These components are obtained by gluing components of $B \backslash \backslash \Phi$ along the infinitesimal cycles and the edges that enter them.

To describe these gluings more precisely, let us introduce some terminology. For a component $J$ of $B \backslash \backslash \Phi$, we call the vertex of a tongue not lying on the annulus/M\"obius band the \textit{tip} of the tongue. If one cuts open $\partial J$ along all the tips, the resulting set has a natural structure as an oriented train track, where the orientation is induced from that of $\Phi$. Under this structure, each component of this set will be a cycle with some entering edges. We call each such component a \textit{crown}.

To obtain a component of $B \backslash \backslash \Phi_{\red}$ containing infinitesimal cycles, certain components of $B \backslash \backslash \Phi$ are glued together by identifying crowns in their boundaries. Under such a gluing, the union of annuli/M\"obius bands is an annulus/M\"obius band, while collections of tongues that share a tip are glued together to form individual tongues, showing that such a component of $B \backslash \backslash \Phi_{\red}$ is an annulus/M\"obius band with tongues as well.

The criss-cross pattern of the arcs of attachment follows from the zig-zag pattern for components of $B \backslash \backslash \Phi$ and from the description of the gluing.

We remark that this argument actually shows that all the infinitesimal cycles contained in a component of $B \backslash \backslash \Phi_{\red}$ belong to a single wall. In this case, $w$ in the statement of the lemma is the width of the wall. 
\end{proof}

\begin{rmk} \label{rmk:stablegluesides}
We remark that if $w \geq 3$ for a component of $B \backslash \backslash \Phi_{\red}$ as in the lemma, the tongues attached along the $y=x+\frac{2i}{N+1}$ arcs must all lie on one side and the tongues attached along the $y=-x+\frac{2i}{N+1}$ arcs must all lie on the opposite side, in the universal cover. In contrast, the sides of attachment do not have fixed patterns for $w=1,2$.
\end{rmk}

\section{Finiteness of layered veering triangulations} \label{sec:finiteness}

In this section, we fix the gap in the proof of \cite[Theorem 6.2]{Ago11}. For completeness, we explain the setup of the original theorem then proceed to show a way of fixing its proof, albeit with a weaker bound.

We first have to explain layered veering triangulations.

Let $\phi:S_{g,n} \to S_{g,n}$ be a pseudo-Anosov homeomorphism on a finite type surface with $\chi(S_{g,n})=2-2g-n<0$. $S^\circ$ will denote the surface obtained by removing the singularities of the stable and unstable foliations for $\phi$, and $\phi^\circ$ will denote the restriction of $\phi$ to $S^\circ$. Write $T(\phi^\circ)$ for the mapping torus of $\phi^\circ$.

In \cite{Ago11}, it is shown that there exists a periodic folding sequence of train tracks $\tau_0 \rightsquigarrow ...  \rightsquigarrow \tau_N$, i.e. train tracks such that $\tau_{i+1}$ is obtained from $\tau_i$ using a folding move and $\phi(\tau_N)=\tau_0$, on $S^\circ$. The sequence of ideal triangulations $\delta_i$ of $S^\circ$ dual to $\tau_i$ are then related to one another by diagonal switches in quadrilaterals, and $\phi^\circ$ sends $\delta_N$ to $\delta_0$. 

A veering triangulation can be constructed in the following way: Start with $\delta_0$ on $S^\circ$ and attach a flat tetrahedron to the bottom that effects the diagonal switch from $\delta_0$ to $\delta_1$. The bottom boundary of the complex can be identified with $(S^\circ,\delta_1)$. We inductively add tetrahedra to the bottom, until the bottom boundary of the complex can be identified with $(S^\circ,\delta_N)$. Finally, we glue this bottom boundary to the top boundary to get a triangulation $\Delta$ of $T(\phi^\circ)$, using the fact that $\phi^\circ$ sends $\delta_N$ to $\delta_0$. We color an edge of $\delta_i$ that is dual to a small branch according to the direction of smoothing at the endpoints, using \Cref{fig:layeredvtcolors}. It can be checked that this determines a well-defined edge coloring of $\Delta$ which makes it into a veering triangulation.

\begin{figure}
    \centering
    \resizebox{!}{3cm}{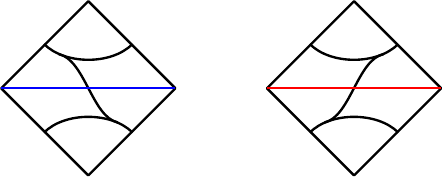}
    \caption{Defining colors of edges of $\delta_i$ that are dual to small branches. This determines a well-defined edge coloring of $\Delta$ which makes it into a veering triangulation.}
    \label{fig:layeredvtcolors}
\end{figure}

We construct a directed graph according to this description. Start with a set of $e$ vertices given by the set of edges in $\delta_0$. Notice that $\delta_1$ and $\delta_0$ differ by one edge exactly, say $\delta_1 \backslash e_1=\delta_0 \backslash e_0$. Add a vertex corresponding to $e_1$, and add three edges going from the three elements of $E(\delta_0)$ dual to the three branches of $\tau_0$ which fold onto the branch of $\tau_1$ dual to $e_1$. One of these edges goes from $e_0$ to $e_1$, call this edge \textit{vertical}, and call the other two edges \textit{slanted}. Inductively, for each $i$, add a vertex that corresponds to the new edge in $\delta_i$ and add three edges according to the folding move $\tau_i \rightsquigarrow \tau_{i+1}$, one vertical and two slanted. Call the resulting directed graph the \textit{cut open flow graph} $\Phi \backslash \backslash S$. See \Cref{fig:cutopenflowgraph}, where we draw vertical and slanted edges as vertical and slanted respectively.

\begin{figure}
    \centering
    \fontsize{14pt}{14pt}\selectfont
    \resizebox{!}{6cm}{
\begingroup%
  \makeatletter%
  \providecommand\color[2][]{%
    \errmessage{(Inkscape) Color is used for the text in Inkscape, but the package 'color.sty' is not loaded}%
    \renewcommand\color[2][]{}%
  }%
  \providecommand\transparent[1]{%
    \errmessage{(Inkscape) Transparency is used (non-zero) for the text in Inkscape, but the package 'transparent.sty' is not loaded}%
    \renewcommand\transparent[1]{}%
  }%
  \providecommand\rotatebox[2]{#2}%
  \newcommand*\fsize{\dimexpr\f@size pt\relax}%
  \newcommand*\lineheight[1]{\fontsize{\fsize}{#1\fsize}\selectfont}%
  \ifx\svgwidth\undefined%
    \setlength{\unitlength}{402.00516255bp}%
    \ifx\svgscale\undefined%
      \relax%
    \else%
      \setlength{\unitlength}{\unitlength * \real{\svgscale}}%
    \fi%
  \else%
    \setlength{\unitlength}{\svgwidth}%
  \fi%
  \global\let\svgwidth\undefined%
  \global\let\svgscale\undefined%
  \makeatother%
  \begin{picture}(1,0.38807778)%
    \lineheight{1}%
    \setlength\tabcolsep{0pt}%
    \put(0,0){\includegraphics[width=\unitlength,page=1]{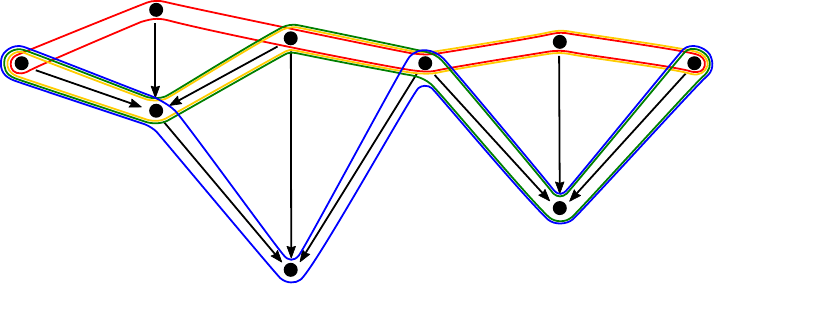}}%
    \put(0.92976614,0.31908401){\color[rgb]{1,0,0}\makebox(0,0)[lt]{\lineheight{1.25}\smash{\begin{tabular}[t]{l}$E(\delta_0)$\end{tabular}}}}%
    \put(0.92976614,0.23712137){\color[rgb]{1,0.8,0}\makebox(0,0)[lt]{\lineheight{1.25}\smash{\begin{tabular}[t]{l}$E(\delta_1)$\end{tabular}}}}%
    \put(0.92976614,0.14916944){\color[rgb]{0,0.50196078,0}\makebox(0,0)[lt]{\lineheight{1.25}\smash{\begin{tabular}[t]{l}$E(\delta_2)$\end{tabular}}}}%
    \put(0.92976614,0.05881879){\color[rgb]{0,0,1}\makebox(0,0)[lt]{\lineheight{1.25}\smash{\begin{tabular}[t]{l}$E(\delta_3)$\end{tabular}}}}%
    \put(0,0){\includegraphics[width=\unitlength,page=2]{cutopenflowgraph.pdf}}%
  \end{picture}%
\endgroup%
}
    \caption{The cut open flow graph for a layered veering triangulation. There is a natural identification between adjacent layers $E(\delta_i)$ and $E(\delta_{i+1})$, which is represented in the figure by moving one element directly downwards. This divides the cut open flow graph into columns. The cut open reduced flow graph for the veering cellulation can be obtained by deleting some columns.}
    \label{fig:cutopenflowgraph}
\end{figure}

The flow graph of $\Delta$ can be obtained from $\Phi \backslash \backslash S$ by identifying the vertices at the bottom corresponding to $E(\delta_N)$ to those at the top for $E(\delta_0)$ according to how $\phi^\circ$ sends $\delta_N$ to $\delta_0$. 

Now define another directed graph $G$ by setting the set of vertices to be the set of edges in $\delta_0$, and placing an edge from $i$ to $j$ for every directed edge path in the cut open flow graph $\Phi \backslash \backslash S$ which starts at $i$ and ends at $(\phi^\circ)^{-1}(j)$. Let $A$ be the adjacency matrix of $G$ (which we defined in \Cref{defn:PF}).

We note that $A$ can also be defined as the transition matrix describing how weights on the branches of $\tau_0$ distribute under the sequence of folding moves $\tau_0 \rightsquigarrow ... \rightsquigarrow \tau_N$ and the return map $\phi^\circ$.

Meanwhile, let $\lambda=\lambda(\phi)>1$ be the dilatation of $\phi$. The train tracks $\tau_i$ in fact carry the unstable measured lamination of $\phi$, hence the transverse measure on the leaves of the foliation collapse down to weights on the branches of $\tau_0$. These in turn define an eigenvector $(w_i)$ of $A$ with eigenvalue $\lambda$. 

The assertion of \cite[Theorem 6.2]{Ago11} is that if $\lambda^{2g-2+\frac{2}{3}n} \leq P$, then the veering triangulation $\Delta$ has at most $\frac{P^9-1}{2}$ tetrahedra.

We recall the proof presented in \cite{Ago11}, but phrased in the language here. We first claim that $G$ has at least $2N+e$ edges. This is because for every slanted edge in $\Phi \backslash \backslash S$, we can take an edge path from $E(\delta_0)$ to $E(\delta_N)$ by concatenating paths of vertical edges to the back and front of the slanted edge. Also, for every vertex $i$ in $E(\delta_0)$, there is an edge path from $i$ to $E(\delta_N)$ that consists entirely of vertical edges. Since there are $2N$ slanted edges and $e$ vertices in $E(\delta_0)$, we find $2N+e$ edges in $G$.

Also notice that by an index calculation, we have $e \leq 9(2g-2+\frac{2}{3}n)$. See \cite[Lemma 6.1]{Ago11} for details of this.

Now assume for the moment that $G$ is strongly connected. Then \cite{Ago11} proceeds by using an estimate of Ham and Song (\cite{HS07}). We repeat Ham and Song's argument here since it is a good warm-up for the proof which we will present later. Fix a vertex $i$ of $G$ and fix a spanning tree $T$ of $G$ rooted at $i$, i.e. $T$ is a subgraph of $G$ which is a tree, and for every vertex $j$ of $G$ there is an edge path from $j$ to $i$ within $T$. Then there are at least $2N+1$ edges in $G \backslash T$, and for every edge $d$ in $G \backslash T$, we can find a path in $G$ ending at $i$ and of length $e$ by prepending to the path travelling across $d$ then to $i$ within $T$ an arbitrary path at the front. These paths will be distinct, hence using the fact that $(w_i)$ is an eigenvector of $A^e$ of eigenvalue $\lambda^e$, and summing over the paths, we have $\lambda^e w_i \geq \Sigma w_j \geq (2N+1) \min_j w_j$. Taking the minimum over the left hand side, $\lambda^e \min w_i \geq (2N+1) \min w_j$, hence $N \leq \frac{\lambda^e-1}{2} \leq \frac{P^9-1}{2}$.

The problem with this argument however, is that $G$ is not always strongly connected.

In fact, $G$ is strongly connected if and only if the flow graph $\Phi$ is strongly connected. For if $\Phi$ is strongly connected, then for every pair of vertices $(i,j)$ in $E(\delta_0)$, their images can be connected by an edge path $\alpha$ in $\Phi$. The preimage of such an edge path under $\Phi \backslash \backslash S \to \Phi$ is a collection of paths $\alpha_1,...,\alpha_s$, for which the ending point of $\alpha_i$, which lies in $E(\delta_N)$, is sent by $\phi$ to the starting point of $\alpha_{i+1}$, which lies in $E(\delta_0)$. Hence the $\alpha_i$ define an edge path $\alpha'$ in $G$ connecting $i$ to $j$.

Conversely, if $\Phi$ is not strongly connected, then we can find a infinitesimal cycle $c$ of a wall. As above, $c$ lifts to a collection of paths $c_1,...,c_s$ in $\Phi \backslash \backslash S$ and determines a cycle $c'$ in $G$. In fact, $c_i$ will consist entirely of vertical edges by the definition of the infinitesimal cycles of a wall. The vertices of $G$ that lie in $c'$ are exactly those elements of $E(\delta_0)$ that have image in $c$. Moreover, $c'$ has no outgoing edges in $G'$, for otherwise there is an outgoing path of $c$ in $\Phi$.

This motivates us to consider instead the full subgraph of $G$ obtained by restricting to the set of vertices that have image in the reduced flow graph $\Phi_{\red}$ in $\Phi$. We call this subgraph $G_{\red}$. By the argument above, $G_{\red}$ can also be obtained by removing $c'$ for all infinitesimal cycles $c$ of $\Phi$. 

We also set the \textit{cut open reduced flow graph}, $\Phi_{\red} \backslash \backslash S$, to be the preimage of $\Phi_{\red} \subset \Phi$ in $\Phi \backslash \backslash S$. Let $E_{\red}(\delta_i)$ be the subset of $E(\delta_i) \subset \Phi \backslash \backslash S$ that lie in $\Phi_{\red} \backslash \backslash S$. $G_{\red}$ can be obtained by gluing $E_{\red}(\delta_N)$ at the bottom of $\Phi_{\red} \backslash \backslash S$ to $E_{\red}(\delta_0)$ at the top. Hence arguing as above, we can see that $G_{\red}$ is strongly connected.

The strategy now is to apply Ham and Song's argument on $G_{\red}$ to bound the number of vertices in $\Phi_{\red}$, then bound the number of vertices in $\Phi \backslash \Phi_{\red}$. 

We first set up some notation. Let $N'$ be the number of vertices in $\Phi_{\red}$. Let $e'$ be the number of elements in $E_{\red}(\delta_0)$.

We claim that $G_{\red}$ has at least $2N'+e'$ edges. This follows from the same argument as for $G$: we can find an edge path in $\Phi_{\red} \backslash \backslash S$ from $E_{\red}(\delta_0)$ to $E_{\red}(\delta_N)$ for every slanted edge in $\Phi_{\red} \backslash \backslash S$ and for every vertex in $E_{\red}(\delta_0)$.

Meanwhile, recall that $G \backslash G_{\red}$ is the union of $c'$ for the infinitesimal cycles $c$ of $\Phi$, where each $c'$ has no outgoing edges. Hence the adjacency matrix $A_{\red}$ of $G_{\red}$ is a submatrix of $A$ where all other entries in the same row are zeros. As a consequence, $(w_i)_{i \in G_{\red}}$ is an eigenvector of $A_{\red}$ with eigenvalue $\lambda$. 

Hence repeating Ham and Song's argument, we have $N' \leq \frac{\lambda^{e'}-1}{2}$.

Now suppose we have a width $w$ wall. For simplicity, first suppose that the wall is untwisted. Then there are $w-1$ infinitesimal cycles and $2$ boundary cycles of the wall, all of the same length $h$. The $(w-1)h$ vertices in the infinitesimal cycles are discarded in $\Phi_{\red}$, but the $2h$ vertices in the boundary cycles remain. Similarly, if the wall is twisted, there are $(w-1)h$ vertices discarded while $2h$ vertices remain, for an appropriate $h$. In other words, if we let $W$ be the maximum width of a wall in the veering triangulation $\Delta$, then for every vertex in a boundary cycle, there are at most $\frac{W-1}{2}$ vertices in the infinitesimal cycles, which we discard when passing to $\Phi_{\red}$.

Meanwhile, as discussed in \Cref{subsec:wall}, each vertex of $\Phi$ can appear at most twice in the collection of all boundary cycles. Hence we conclude that the number of discarded vertices $N-N'$  is at most $N'(W-1)$, or $N \leq N'W$. Thus it remains to bound $W$.

To that end, let $c_2,...,c_W$ be the infinitesimal cycles in a width $W$ wall, and let $c_1,c_{W+1}$ be the boundary cycles of the wall. Again, for simplicity, first suppose that the wall is untwisted. Then there are corresponding disjoint cycles $c'_1,...,c'_{W+1}$ in $G$. The length of each cycle $c'_i$ is given by the number of paths in the lift of $c_i$ in the cut open flow graph. But this number is also equal to the intersection number of $c_i$ with $S^\circ$ in $T(\phi^\circ)$. Since the $c_i$ are parallel to each other, this number is of the same value $L$ for each $i=1,...,W+1$ and hence at most $\frac{e}{W+1}$ by the pigeonhole principle. Similarly, if the wall is twisted, the infinitesimal and boundary cycles determine cycles in $G$ of lengths at most $\frac{2e}{W+1}$.

Now pick a vertex of $\Phi_{\red}$ that lies in $c_1$. One of its preimages in $\Phi_{\red} \backslash \backslash S$ has three incoming edges. $c_1$ passes through one of these edges, and for each of the remaining two, we can construct a path from $E_{\red}(\delta_0)$ to $E_{\red}(\delta_N)$ by concatenating vertical edges to the back and front of the edge. These two paths determine two edges that enter $c'_1$ at a vertex $i$ in $G_{\red}$. Since $G_{\red}$ is strongly connected, we can locate two edge paths in $G_{\red}$ of length $L$ that end at the two edges respectively. Suppose these two paths start at vertices $j$ and $k$. Then using the fact that $(w_i)$ is an eigenvector of $A^L_{\red}$ with eigenvalue $\lambda^L$, we have $\lambda^L w_i \geq w_i + w_j + w_k$, which implies 

$$w_i \geq \frac{2}{\lambda^L-1} \min_j w_j$$

But for every two vertices $i,j$ in $G_{\red}$, by strong connectivity, there is an edge path of length $\leq e'$ from $i$ to $j$, hence $w_i \geq \lambda^{-e'} w_j$. So $\frac{\max_j w_j}{\min_i w_i} \leq \lambda^{e'}$. Applying this to the above inequality.

$$ \lambda^{e'} \geq \frac{2}{\lambda^L-1} $$

$$ \lambda^L \geq 2 \lambda^{-e'}+1 $$

Hence,

$$ \frac{2e}{W+1} \geq L \geq \frac{\log(2 \lambda^{-e'}+1)}{\log \lambda} $$

$$ W \leq \frac{2 \log \lambda^e}{\log(2 \lambda^{-e'}+1)} -1 $$

Putting everything together, 

\begin{align*}
N \leq& \frac{\lambda^{e'}-1}{2} (\frac{2 \log \lambda^e}{\log(2 \lambda^{-e'}+1)}-1)\\
\leq& \frac{\lambda^e-1}{2} (\frac{2 \log \lambda^e}{\log(2 \lambda^{-e}+1)}-1)\\
\leq& \frac{P^9-1}{2} (\frac{2 \log P^9}{\log(2 P^{-9}+1)}-1)
\end{align*}

We record this as a theorem.

\begin{thm} \label{thm:boundveertet}
If $M$ is the punctured mapping torus of a pseudo-Anosov homeomorphism $\phi: S_{g,n} \to S_{g,n}$, where the normalized dilatation $\lambda(\phi)^{2g-2+\frac{2}{3}n} \leq P$, then $M$ has a veering triangulation with at most $\frac{P^9-1}{2} (\frac{2 \log P^9}{\log(2 P^{-9}+1)}-1)$ tetrahedra.
\end{thm}

\begin{rmk} \label{rmk:boundcompare}
$\frac{P^9-1}{2} (\frac{2 \log P^9}{\log(2 P^{-9}+1)}-1)$ is asymptotically $\frac{9}{2} P^{18}\log P$ as $P \to \infty$, so we have worsened the exponent on the bound in \cite[Theorem 6.2]{Ago11} by a factor of $2+\epsilon$. 

We remark that the original bound of $\frac{P^9-1}{2}$ still holds for veering triangulations that have strongly connected flow graph. However, it is not clear if this is the case for `most' $\phi$ or if there is a way to tell if this is the case just from $\phi$.

It also seems likely that there is room for improvement for our bound in the general case. For example, it should be possible to obtain better bounds on $W$ by finding more paths that enter $c'_1$, which should be easy when $L$ is large. When $W$ or $L$ is large, it should also be possible to bound $e'$ more effectively than just using $e$, which is what we have done here. See \Cref{sec:questions} for another discussion on how one might improve the bound.
\end{rmk}

\section{Pseudo-Anosov flows} \label{sec:pAflow}

In this section we will reprove Schleimer and Segerman's result that a veering triangulation induces a pseudo-Anosov flow without perfect fits on suitable Dehn fillings. We use the following notations to simplify the statement. If $s=(s_i)$ is a collection of slopes on each boundary component of $\overline{M}$, we write $M(s)$ to mean the closed 3-manifold obtained by Dehn filling $\overline{M}$ along the slopes recorded by $s$. If $a=(a_i)$ and $b=(b_i)$ are collections of slopes on each boundary component of $\overline{M}$, then by $|\langle a,b \rangle| \geq n$ we mean that the geometric intersection numbers between $a_i$ and $b_i$ on each boundary component is at least $n$.

\begin{thm} \label{thm:vtpAflow}
Suppose $M$ admits a veering triangulation. Let $l=(l_i)$ denote the collection of ladderpole curves on each boundary component. Then $M(s)$ admits a transitive pseudo-Anosov flow $\phi$ if $|\langle s,l \rangle| \geq 2$. 

Furthermore, there are closed orbits $c_i$ isotopic to the cores of the filling solid tori, such that each $c_i$ is $|\langle s_i,l_i \rangle|$-pronged, and $\phi$ is without perfect fits relative to $\{c_i\}$.
\end{thm}

We will recall the definitions of pseudo-Anosov flows, transitivity, and no perfect fits in \Cref{subsec:pAproof}. 

A subtle point of the theorem is that there are actually two common notions of pseudo-Anosov flows on 3-manifolds in the literature, which we differentiate by calling them topological pseudo-Anosov flows and smooth pseudo-Anosov flows. \Cref{thm:vtpAflow} holds for both notions, due to \Cref{thm:top2smooth}. We will explain this technicality more in \Cref{subsec:pAproof,subsec:smooth}.

The proof can be outlined as follows. We first thicken up the reduced flow graph $\Phi_{\red}$ in $M$ to $N(\Phi_{\red})$ by replacing its edges with flow boxes. This set can be considered as a subset of a neighborhood of the unstable branched surface, $N(B)$, naturally. By understanding the complement of $\Phi_{\red}$ in $B$, we are able to glue faces of $N(\Phi_{\red})$ across its complementary regions in $N(B)$. Similarly, by understanding the complement of $B$ in $M$, we are able to glue faces of $N(B)$ across its complementary regions in $M(s)$. These gluings preserve the (singular) 1-dimensional foliation on the flow boxes, hence that descends down to a (honest) 1-dimensional foliation on $M(s)$, which can be parametrized into a topological flow, for which we show is pseudo-Anosov, transitive, and without perfect fits (relative to the orbits $\{c_i\}$).

\begin{rmk}
It is not too difficult to show that a veering triangulation induces a (topological) pseudo-Anosov flow on $M(s)$ for $|\langle s,l \rangle| \geq 2$, using the tool of dynamic pairs developed by Mosher in \cite{Mos96}. Specifically, one can apply the proof of \cite[Proposition 2.6.2]{Mos96} to $(B, \Phi)$ to produce a dynamic pair in $M(s)$, which by \cite[Theorem 3.4.1]{Mos96} gives rise to a pseudo-Anosov flow. The more challenging part however, at least from this approach, is to show that such a pseudo-Anosov flow is transitive and has no perfect fits. 

In our proof, we use a lot of the same ideas as \cite{Mos96}, but most notably we skip over constructing the `stable branched surface' in a dynamic pair, and instead construct a pseudo-Anosov flow directly from the `unstable branched surface' $B$ and the `dynamic train track' $\Phi_{\red}$, to use the terminology from \cite{Mos96}. This allows us to analyze the pseudo-Anosov flow using the special properties of $\Phi_{\red}$ and $B$, proving transitivity and no perfect fits.
\end{rmk}

\subsection{Thickening up $\Phi_{\red}$}

We know that $\Phi_{\red}$ is strongly connected by construction. This is equivalent to its adjacency matrix $A \in Hom(\mathbb{R}^{V(\Phi_{\red})}, \mathbb{R}^{V(\Phi_{\red})})$ being irreducible. As such, by the Perron-Frobenius Theorem, $A$ has a positive eigenvector $(w_v)$ with eigenvalue $\lambda \geq 1$. We know that $\lambda>1$ since each vertex of $\Phi_{\red}$ has three incoming edges. Meanwhile, $A$ being irreducible implies $A^T$ is irreducible as well, and so the latter has a positive eigenvector $(w'_v)$ with the same eigenvalue $\lambda >1$.

We remark in passing that each vertex of $\Phi_{\red}$ having $3$ incoming edges in fact implies that $\lambda=3$ and $w_v=1$ for all $v$. This knowledge, however, will play no role in the construction at all. We simply wish to point out that the value of $\lambda$ has nothing to do with the dilatation factor of the monodromy when $\Delta$ is layered.

Recall that $\Phi_{\red}$ naturally sits on $B$ inside $M$. We will thicken up $\Phi_{\red}$ by replacing each edge $e$ of $\Phi_{\red}$ going from $v$ to $w$ by a flow box $$Z_e \cong \{(s,u,t) \in \mathbb{R}^3: |s|\leq w_v \lambda^{t-1}, |u| \leq w'_w \lambda^{-t}, t \in [0,1] \}$$ See \Cref{fig:flowbox}. There are two 2-dimensional foliations on $Z_e$: the first one by leaves of the form $\{u=u_0 \lambda^{-t} \}_{u_0}$, which we will call the stable foliation, and the second one by leaves of the form $\{s=s_0 \lambda^{t-1}\}_{s_0}$, which we will call the unstable foliation. There is also an oriented 1-dimensional foliation on $Z_e$ by curves $\{(s_0\lambda^{t-1}, u_0\lambda^{-t},t): t \in [0,1]\}_{s_0,u_0}$, oriented by decreasing $t$. Notice that the leaves of the stable and unstable foliations intersect transversely along leaves of the oriented 1-dimensional foliation.

\begin{figure}
    \centering
    \fontsize{16pt}{16pt}\selectfont
    \resizebox{!}{4cm}{
\begingroup%
  \makeatletter%
  \providecommand\color[2][]{%
    \errmessage{(Inkscape) Color is used for the text in Inkscape, but the package 'color.sty' is not loaded}%
    \renewcommand\color[2][]{}%
  }%
  \providecommand\transparent[1]{%
    \errmessage{(Inkscape) Transparency is used (non-zero) for the text in Inkscape, but the package 'transparent.sty' is not loaded}%
    \renewcommand\transparent[1]{}%
  }%
  \providecommand\rotatebox[2]{#2}%
  \newcommand*\fsize{\dimexpr\f@size pt\relax}%
  \newcommand*\lineheight[1]{\fontsize{\fsize}{#1\fsize}\selectfont}%
  \ifx\svgwidth\undefined%
    \setlength{\unitlength}{180.36318513bp}%
    \ifx\svgscale\undefined%
      \relax%
    \else%
      \setlength{\unitlength}{\unitlength * \real{\svgscale}}%
    \fi%
  \else%
    \setlength{\unitlength}{\svgwidth}%
  \fi%
  \global\let\svgwidth\undefined%
  \global\let\svgscale\undefined%
  \makeatother%
  \begin{picture}(1,1.29843748)%
    \lineheight{1}%
    \setlength\tabcolsep{0pt}%
    \put(0,0){\includegraphics[width=\unitlength,page=1]{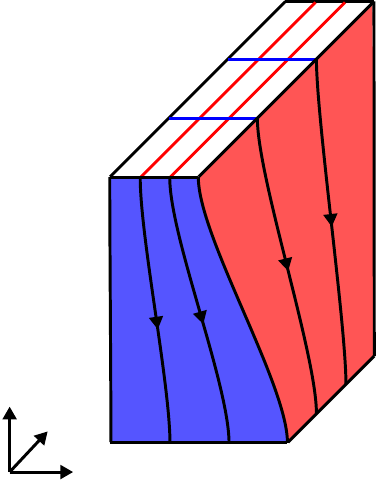}}%
    \put(-0.00755312,0.24596661){\color[rgb]{0,0,0}\makebox(0,0)[lt]{\lineheight{1.25}\smash{\begin{tabular}[t]{l}$t$\end{tabular}}}}%
    \put(0.20152163,0.01759146){\color[rgb]{0,0,0}\makebox(0,0)[lt]{\lineheight{1.25}\smash{\begin{tabular}[t]{l}$u$\end{tabular}}}}%
    \put(0.1189819,0.16816847){\color[rgb]{0,0,0}\makebox(0,0)[lt]{\lineheight{1.25}\smash{\begin{tabular}[t]{l}$s$\end{tabular}}}}%
  \end{picture}%
\endgroup%
}
    \caption{A flow box.}
    \label{fig:flowbox}
\end{figure}
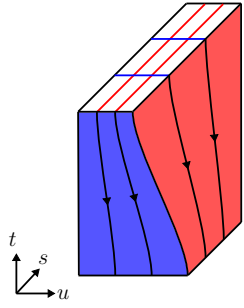

Give $Z_e$ its Euclidean metric induced from $\mathbb{R}^3$. There is a natural way of gluing up the $Z_e$ across their top and bottom faces at the vertices of $\Phi_{\red}$, preserving the metric on those faces, by definition of $w_v$ and $w'_v$. A priori there is a freedom for $Z_e$ to twist along $e$, we get rid of this by requiring that the framing on $Z_e$ induced from its $u$ coordinate match up with the framing on $e$ induced from $B$. See \Cref{fig:thicken}. We call the resulting set $N(\Phi_{\red})$, since it is a regular neighborhood of $\Phi_{\red}$ in $M$.

\begin{figure}
    \centering
    \fontsize{14pt}{14pt}\selectfont
    \resizebox{!}{8cm}{
\begingroup%
  \makeatletter%
  \providecommand\color[2][]{%
    \errmessage{(Inkscape) Color is used for the text in Inkscape, but the package 'color.sty' is not loaded}%
    \renewcommand\color[2][]{}%
  }%
  \providecommand\transparent[1]{%
    \errmessage{(Inkscape) Transparency is used (non-zero) for the text in Inkscape, but the package 'transparent.sty' is not loaded}%
    \renewcommand\transparent[1]{}%
  }%
  \providecommand\rotatebox[2]{#2}%
  \newcommand*\fsize{\dimexpr\f@size pt\relax}%
  \newcommand*\lineheight[1]{\fontsize{\fsize}{#1\fsize}\selectfont}%
  \ifx\svgwidth\undefined%
    \setlength{\unitlength}{498.12827813bp}%
    \ifx\svgscale\undefined%
      \relax%
    \else%
      \setlength{\unitlength}{\unitlength * \real{\svgscale}}%
    \fi%
  \else%
    \setlength{\unitlength}{\svgwidth}%
  \fi%
  \global\let\svgwidth\undefined%
  \global\let\svgscale\undefined%
  \makeatother%
  \begin{picture}(1,0.71490139)%
    \lineheight{1}%
    \setlength\tabcolsep{0pt}%
    \put(0,0){\includegraphics[width=\unitlength,page=1]{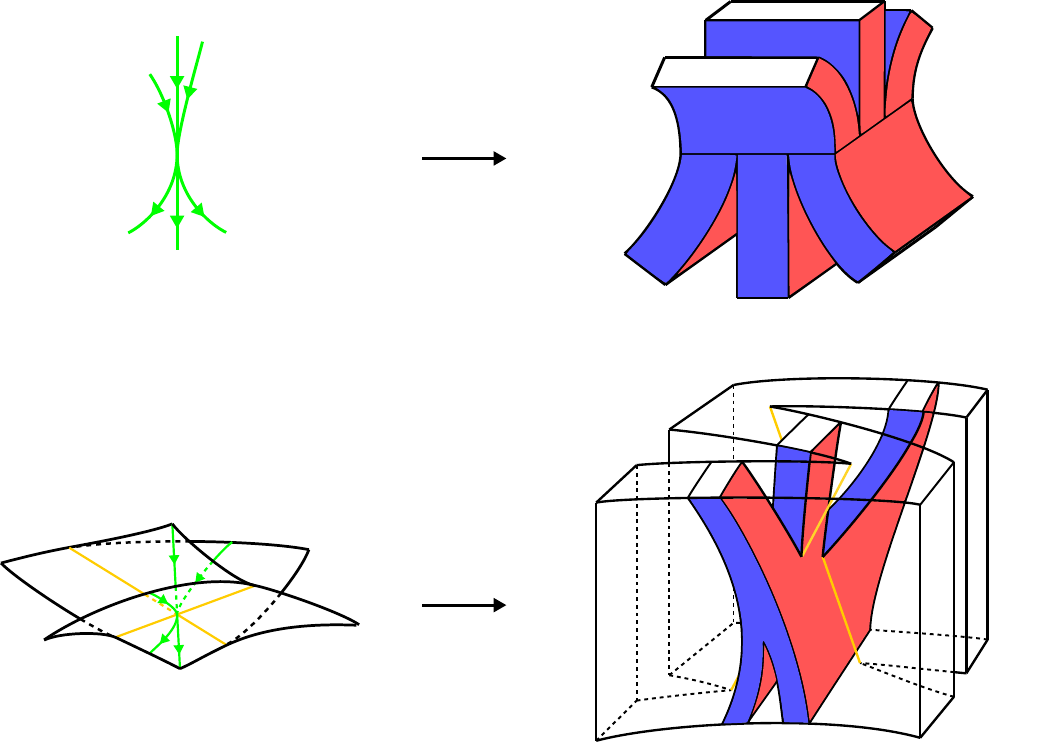}}%
    \put(0.88977488,0.6177298){\color[rgb]{1,0,0}\makebox(0,0)[lt]{\lineheight{1.25}\smash{\begin{tabular}[t]{l}$\partial^s N(\Phi_{red})$\end{tabular}}}}%
    \put(0.51727148,0.56069522){\color[rgb]{0,0,1}\makebox(0,0)[lt]{\lineheight{1.25}\smash{\begin{tabular}[t]{l}$\partial^u N(\Phi_{red})$\end{tabular}}}}%
  \end{picture}%
\endgroup%
}
    \caption{Thickening up $\Phi_{\red}$ to $N(\Phi_{\red})$, which naturally embeds in $N(B)$.}
    \label{fig:thicken}
\end{figure}

Note that the oriented 1-dimensional foliations on $Z_e$ piece together to give a decomposition of $N(\Phi_{\red})$. This decomposition is almost an oriented 1-dimensional foliation, except the `leaves' are oriented 1-manifolds possibly with train track singularities branching off in forward and backward directions, arising from the multiple outgoing and incoming edges at the vertices. Despite this, for convenience we will still refer to this decomposition as an oriented 1-dimensional foliation. 

Similarly, the stable and unstable foliations on $Z_e$ each piece together to give a decomposition of $N(\Phi_{\red})$ which is almost a 2-dimensional foliation, except that the `leaves' have branching. We will refer to these as the stable and unstable foliations on $N(\Phi_{\red})$ respectively. 

Denote $\partial^s Z_e=\{u= \pm w'_w \lambda^{-t}\}$ and $\partial^u Z_e=\{s= \pm w_v \lambda^{t-1}\}$. We will call $\partial^s N(\Phi_{\red}) := \bigcup_e \partial^s Z_e$ the \textit{stable boundary} of $N(\Phi_{\red})$ and $\partial^u N(\Phi_{\red}) := \bigcup_e \partial^u Z_e$ the \textit{unstable boundary} of $N(\Phi_{\red})$. We will also call the collection of the closures of intervals where the interior of the top and bottom faces of $Z_e$ meets $\partial^s N(\Phi_{\red})$ or $\partial^u N(\Phi_{\red})$ the branch locus. These are exactly the places where the leaves of the oriented 1-dimensional foliation on $N(\Phi_{\red})$ have branching.

Now consider an $I$-fibered neighborhood $N(B)$ of $B$ in the cusped model. This means that $N(B)$ is a closed regular neighborhood of $B$, with a map $M \to M$ restricting to a projection on $N(B) \to B$ with $I$-fibers and a homeomorphism outside of $N(B)$, and so that the boundary of $N(B)$ can be decomposed into surfaces with boundary, the interior on which the projection is a local homeomorphism, and the boundary curves correspond to the components of the branch locus of $B$. We will often conflate a component of the branch locus of $B$ with the corresponding circle on $\partial N(B)$. We will also call the collection of circles in the latter setting the branch locus of $N(B)$. 

$N(\Phi_{\red})$ can be arranged to be a subset of $N(B)$, in such a way that $\partial^u N(\Phi_{\red}) \subset \partial N(B)$, the branch locus of $N(\Phi_{\red})$ is a subset of the branch locus of $N(B)$, and $N(\Phi_{\red})$ is saturated with respect to the $I$-fibering of $N(B)$.

Note that we do not require the intervals $[-w_v \lambda^{t_0-1}, w_v \lambda^{t_0-1}] \times \{ u_0 \} \times \{ t_0 \} \subset Z_e \subset N(\Phi_{\red})$ to coincide with the $I$-fibers of $N(B)$. Indeed, this is impossible since the branch locus of $N(\Phi_{\red})$ are parallel along such intervals, while branch locus of $N(B)$ project to transverse curves on $B$. One can arrange for this property by `splitting' $N(\Phi_{\red})$ slightly, but since this does not aid our construction, we will not do so. 

\subsection{Gluing along the stable boundary} \label{subsec:stableglue}

The next step is to glue $N(\Phi_{\red})$ along its stable boundary across its complementary regions in $N(B)$, so that the foliations on $N(\Phi_{\red})$ descend to respective foliations on $N(B)$. To perform the gluing, we have to understand the complementary regions in question. 

By \Cref{lemma:redflowgraphcompl}, the complementary regions of $\Phi_{\red}$ in $B$ are annuli or M\"obius bands with tongues. The complementary regions of $N(\Phi_{\red})$ in $N(B)$ are $I$-fibered neighborhoods of these, see \Cref{fig:stableglue1} (where the red intervals will come into play later). Fix one of these components $K$. The boundary of $K$ is naturally divided into $\partial^s K \cup \partial^u K$, where $\partial^s K$ is identified to $\partial^s N(\Phi_{\red})$ in $M$ and $\partial^u K \subset \partial N(B)$. In particular $\partial^s K$ inherits the oriented 1-dimensional foliation on $\partial^s N(\Phi_{\red})$. Furthermore, if we call the $I$-fibers over the tips of the tongues of $B \backslash \backslash \Phi_{\red}$ the \textit{cusps} of $K$, $\partial^s K$ can be divided along the cusps into two or one components (depending on whether the corresponding component of $B \backslash \backslash \Phi_{\red}$ is an annulus or a M\"obius band with tongues respectively). Call these components the \textit{stable faces} of $K$. 

Meanwhile, recall that $N(\Phi_{\red})$, hence $\partial^s N(\Phi_{\red})$, inherits the Euclidean metric of the flow box $Z_e$. The oriented 1-dimensional foliation contracts the metric in the transverse direction. Hence the oriented foliation on each stable face of $K$ has exactly one $S^1$ leaf, and all other leaves enter through the cusps of $K$ and spiral into the $S^1$ leaf. 

\begin{figure}
    \centering
    \resizebox{!}{8cm}{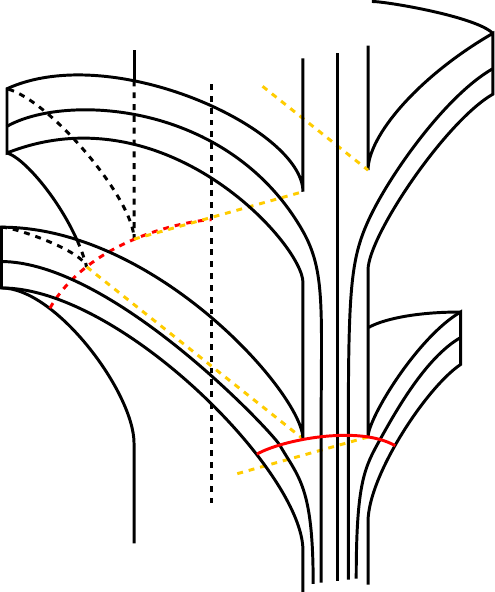}
    \caption{Complementary regions of $N(\Phi_{\red})$ in $N(B)$. We construct an $I$-fibering conjugating the flows on the stable face, and collapse along the $I$-fibers.}
    \label{fig:stableglue1}
\end{figure}

More rigorously, parametrize the oriented 1-dimensional foliation on $\partial^s N(\Phi_{\red})$ in some way, for example according to the $t$ coordinate of each $Z_e$, so that we can consider it as a forward semiflow. Take the intervals $[-w_v, w_v] \times \{ \pm w'_w \lambda^{-1} \} \times \{ 1 \} \subset Z_e \subset N(\Phi_{\red})$, with the Euclidean metrics, as sections of the flow, and note that the first return map contracts the metric by $\lambda^{-1}$, then apply Banach fixed point theorem. 

We claim that there is an $I$-fibering of $K$, transverse to the stable faces and parallel to $\partial^u K$, for which the induced homeomorphism between the stable faces preserves the oriented 1-dimensional foliations on them, i.e. send leaves to leaves in an orientation preserving way. 

Suppose first that the corresponding component of $B \backslash \backslash \Phi_{\red}$ is an annulus with tongues. Choose short local sections to the foliation near the unique $S^1$ leaf on each stable face of $K$, call them $J, J'$. The return maps of the oriented foliations on $J, J'$, which we call $h, h'$, are \textit{contracting} by our analysis above, i.e. $h^k(J)$ is a strictly decreasing collection of subintervals with $\bigcap_{k=1}^\infty h^k(J)$ being a point, and similarly for $h'^k(J')$. Note that these points could lie on endpoints of $J,J'$ if one side of the annulus has no tongues attached.

Construct the $I$-fibering on $K$ in the following steps:
\begin{enumerate}
    \item Define the $I$-fibering on the cusps of $K$, where the $I$-fibering is uniquely determined but degenerate, in the sense that the $I$-fibers are just points.
    \item Extend the $I$-fibering to the triangular faces of $\partial^u K$, so that the base of the triangles, which are among the branch locus of $N(B)$, are $I$-fibers.
    \item Extend the $I$-fibering across subintervals of the leaves of the foliation on the stable faces which start on the branch locus and end at the interior of $J$ or $J'$.
    \item Extend the $I$-fibering across subintervals of the leaves which start on the interior of the cusps of $K$ and end at the interior of $J$ or $J'$, using the fact that the union of these subintervals is a finite union of rectangles foliated as products.
    \item Apply \Cref{lemma:expandextendbb} below by taking $f$ to be the homeomorphism induced by portion of the $I$-fibering already defined. Then use the extended $f$ given by the lemma to construct the $I$-fibering between $J$ and $J'$, and complete the construction by extending the $I$-fibering along the leaves as they go around the stable faces.
\end{enumerate}
See \Cref{fig:stableglue2} for a graphical summary of these steps.

\begin{figure}
    \centering
    \fontsize{16pt}{16pt}\selectfont
    \resizebox{!}{6cm}{
\begingroup%
  \makeatletter%
  \providecommand\color[2][]{%
    \errmessage{(Inkscape) Color is used for the text in Inkscape, but the package 'color.sty' is not loaded}%
    \renewcommand\color[2][]{}%
  }%
  \providecommand\transparent[1]{%
    \errmessage{(Inkscape) Transparency is used (non-zero) for the text in Inkscape, but the package 'transparent.sty' is not loaded}%
    \renewcommand\transparent[1]{}%
  }%
  \providecommand\rotatebox[2]{#2}%
  \newcommand*\fsize{\dimexpr\f@size pt\relax}%
  \newcommand*\lineheight[1]{\fontsize{\fsize}{#1\fsize}\selectfont}%
  \ifx\svgwidth\undefined%
    \setlength{\unitlength}{277.55422913bp}%
    \ifx\svgscale\undefined%
      \relax%
    \else%
      \setlength{\unitlength}{\unitlength * \real{\svgscale}}%
    \fi%
  \else%
    \setlength{\unitlength}{\svgwidth}%
  \fi%
  \global\let\svgwidth\undefined%
  \global\let\svgscale\undefined%
  \makeatother%
  \begin{picture}(1,1.37912733)%
    \lineheight{1}%
    \setlength\tabcolsep{0pt}%
    \put(0,0){\includegraphics[width=\unitlength,page=1]{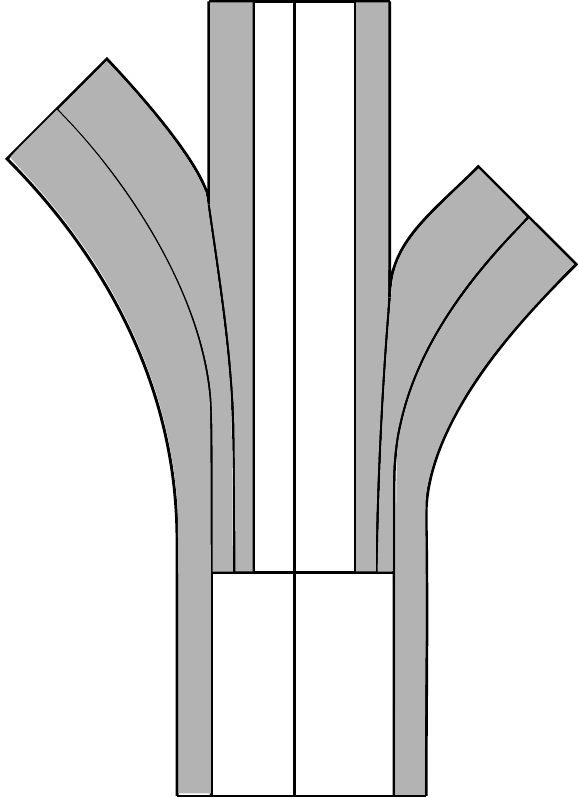}}%
    \put(-0.00573157,1.21058335){\color[rgb]{0,0,0}\makebox(0,0)[lt]{\lineheight{1.25}\smash{\begin{tabular}[t]{l}(1)\end{tabular}}}}%
    \put(0.26409984,1.19552839){\color[rgb]{0,0,0}\makebox(0,0)[lt]{\lineheight{1.25}\smash{\begin{tabular}[t]{l}(2)\end{tabular}}}}%
    \put(0.23553399,0.78788607){\color[rgb]{0,0,0}\makebox(0,0)[lt]{\lineheight{1.25}\smash{\begin{tabular}[t]{l}(4)\end{tabular}}}}%
    \put(0.3880139,0.20498834){\color[rgb]{0,0,0}\makebox(0,0)[lt]{\lineheight{1.25}\smash{\begin{tabular}[t]{l}(5)\end{tabular}}}}%
    \put(0.33128599,0.89954662){\color[rgb]{0,0,0}\makebox(0,0)[lt]{\lineheight{1.25}\smash{\begin{tabular}[t]{l}(3)\end{tabular}}}}%
  \end{picture}%
\endgroup%
}
    \caption{The steps for constructing the $I$-fibering on each component of $N(B) \backslash \backslash N(\Phi_{\red})$.}
    \label{fig:stableglue2}
\end{figure}

\begin{lemma} \label{lemma:expandextendbb}
Let $h:J \to J, h': J' \to J'$ be injective contracting maps on intervals. Let $f:J \backslash h(J) \to J' \backslash h'(J')$ be a given homeomorphism. Then there is a unique way of extending $f$ to a homeomorphism $J \to J'$ so that $h' f=f h$.
\end{lemma}
\begin{proof}
Write $J^{(k)}=h^k(J), J'^{(k)}=h'^k(J')$. By assumption, $J^{(k)}$ is a shrinking subinterval of $J$ which contracts to a point $s$ and $J'^{(k)}$ is a shrinking subinterval of $J'$ which contracts to a point $s'$. Extend $f$ by defining it to be $h'^k f (h^k)^{-1}$ on $J^{(k)} \backslash J^{(k+1)}$. Finally define $f(s)=s'$.
\end{proof}

If the corresponding component of $B \backslash \backslash \Phi_{\red}$ is a M\"obius band with tongues instead, we choose two local sections $\overline{J_1}, \overline{J_2}$ on the one stable face near the $S^1$ leaf, so that we have return maps $\overline{J_1} \to \overline{J_2}, \overline{J_2} \to \overline{J_1}$ for the oriented foliation. (We can pick $\overline{J_1}$ first, then pick $\overline{J_2}$ near $\overline{J_1}$ using the fact that the return map $\overline{J_1} \to \overline{J_1}$ is contracting.) Let $s_1, s_2$ be where the $S^1$ leaf meets $\overline{J_1}, \overline{J_2}$. $s_1, s_2$ have to be in the interior of $\overline{J_1}, \overline{J_2}$ in this setting. Let $J_i$ be the subinterval of $\overline{J_i}$ to the left of $s_i$, $J'_i$ be the subinterval of $\overline{J_{i+1}}$ to the right of $s_{i+1}$ (indices taken mod $2$, and left/right measured relative to the direction of the flow and a fixed orientation on the stable face). First construct the $I$-fibering on subintervals of leaves that start on cusps of $K$ and end at the interior of $\overline{J_1}$ or $\overline{J_2}$, then apply \Cref{lemma:expandextend} below (with $m=2$) to construct the $I$-fibering across a neighborhood of the periodic trajectory as above.

\begin{lemma} \label{lemma:expandextend}
Let $h_i:J_i \to J_{i+1}, h'_i:J'_i\to J'_{i+1}$ be injective contracting maps on intervals $J_1,...,J_m,J'_1,...,J'_m$ (indices taken mod $m$). Here, contracting means \{$h_{mk+i-1} \cdots h_i (J_i)\}_k$ is a decreasing collection of subintervals and $\bigcap_{k=1}^\infty h_{mk+i-1} \cdots h_i (J_i)$ is a single point for each $i$, and similarly for $h'_i$. Let $f_i:J_i \backslash h_{i-1}(J_{i-1}) \to J'_i \backslash h'_{i-1}(J'_{i-1})$ be given homeomorphisms. Then there is a unique way of extending $f_i$ to homeomorphisms $J_i \to J'_i$ so that $h'_i f_i=f_{i+1} h_i$ for all $i$.
\end{lemma}
\begin{proof}
This can be proved as in the previous lemma (which is the `baby case' when $m=1$), with more annoying bookkeeping. We remark that we stated this lemma for all $m \geq 1$ because we will need this generality in the next subsection.
\end{proof}

\begin{rmk} \label{rmk:kopell}
The homeomorphism across the stable faces induced by such an $I$-fibering cannot be made to preserve the Euclidean structure on $\partial^s N(\Phi_{\red})$ in general. In particular, the homeomorphism will not, in general, take an interval of the form $[-w_v \lambda^{t_0-1}, w_v \lambda^{t_0-1}] \times \{ \pm w'_w \lambda^{-t_0} \} \times \{ t_0 \} \subset Z_e \subset N(\Phi_{\red})$ to another interval of this form.

Also, the homeomorphism cannot be chosen to be a diffeomorphism in general (with respect to the smooth structures induced by the Euclidean structures). This arises from the fact that one cannot replace `homeomorphisms' by `diffeomorphisms' in \Cref{lemma:expandextend} in general by Kopell's Lemma (\cite{Kop70}). 
\end{rmk}

Now collapse across all such components $K$ along the $I$-fiberings. This collapses $N(\Phi_{\red})$ onto $N(B)$. The oriented 1-dimensional foliation on $N(\Phi_{\red})$ descends to a decomposition of $N(B)$. The `leaves' of which are still not all 1-manifolds, but now we only have branching in the backwards direction, since all the forward branching occurs along some cusp of some $K$, and we have collapsed all forward trajectories starting from those points. As before, we will still call this decomposition an oriented 1-dimensional foliation for convenience of notation.

Similarly, the stable and unstable foliations on $N(\Phi_{\red})$ descend to decompositions of $N(B)$, which we call the stable and unstable foliations on $N(B)$ respectively. The leaves of the stable foliation on $N(B)$ are in fact 2-manifolds by the same argument as the last paragraph, but those of the unstable foliation on $N(B)$ have branching. 

\begin{rmk} \label{rmk:stableglueboundary}
For the purposes of the next step of the construction, it is worth noting that when restricted to the unstable boundary, the collapse takes $\partial^u N(\Phi_{\red})$ onto $\partial N(B)$, under which the restriction of the oriented 1-dimensional foliation of $N(\Phi_{\red})$ to $\partial^u N(\Phi_{\red})$ descends to the restriction of the oriented 1-dimensional foliation of $N(B)$ to $\partial N(B)$.

In particular, we can recover the latter in the following way: take the image of $\Phi_{\red}$ on the boundary of $M \backslash \backslash B$, thicken it up within $\partial (M \backslash \backslash B)$ by replacing each edge by a flow box $\{ (u,t): |u| \leq w'_w \lambda^{-t}, t \in [0,1]\}$ foliated by $\{(u_0 \lambda^{-t}, t): t \in [0,1] \}_{u_0}$ and piecing together their top and bottom edges at the vertices of $\Phi_{\red}$. The $I$-fiberings on the components $K$ constructed above, when restricted to $\partial^u K$, will determine $I$-fiberings of the complementary regions, and we collapse along these $I$-fibers to get the oriented foliation on $\partial N(B)$.
\end{rmk}

Before we move on, we will construct some local sections to the foliation which will be used to prove no perfect fits in \Cref{subsec:pAproof}. Fix a component $c$ of the branch locus of $B$. Consider the subset of double points of the branch locus of $B$ that lie on $c$. These are each dual to some veering tetrahedron in $\Delta$. Let $V_c$ be the subset of vertices of $\Phi_{\red}$ which are the bottom edges of one of these veering tetrahedra. Cyclically order $V_c=\{v_1,...,v_s\}$ according to the order in which $c$ meets the corresponding double points. 

Meanwhile, for a vertex $v$ of $\Phi_{\red}$, let $R_v$ be the union of the top faces of the flow boxes $Z_e$ as $e$ varies over the edges of $\Phi_{\red}$ that exit $v$. Equivalently, this is also the union of the bottom faces of $Z_e$ as $e$ varies over the edges of $\Phi_{\red}$ that enter $v$. These are embedded rectangles transverse to the foliation on $N(\Phi_{\red})$.

For a collection $V_c=\{v_1,..,v_s\}$ as above, each $R_{v_i} \subset N(\Phi_{\red})$ contains a subinterval of $c$. Here we remind the reader that we are using the same name for a component of the branch locus of $B$ and the corresponding component of the branch locus of $N(B)$. Let $\partial^- R_{v_i}$ be the side of $R_{v_i}$ which $c$ enters, and let $\partial^+ R_{v_i}$ be the side of $R_{v_i}$ which $c$ exits. Note that $\partial^+ R_{v_i}$ and $\partial^- R_{v_{i+1}}$ lie on stable faces of the same component of $N(B) \backslash \backslash N(\Phi_{\red})$ (indices taken mod $s$). We drew one instance of these as red intervals in \Cref{fig:stableglue1}.

After collapsing along the components of $N(B) \backslash \backslash N(\Phi_{\red})$, the image of the union $\bigcup_{i=1}^s R_{v_i}$ in $N(B)$ contains $c$, since for each component $K$ of $N(B) \backslash \backslash N(\Phi_{\red})$, the bases of the triangular faces of $\partial^u K$ are among the $I$-fibers that we chose. The images of $\partial^+ R_{v_i}$ and $\partial^- R_{v_{i+1}}$ may not match up away from $c$. However, what we will show is that we can at least find surfaces $Q_i$ which are unions of finite subintervals of leaves of the 1-dimensional foliation on $N(B)$ which connect the image of $\partial^+ R_{v_i}$ to a subinterval of that of $\partial^- R_{v_{i+1}}$. See \Cref{fig:stableglue3} top.

\begin{figure}
    \centering
    \resizebox{!}{8cm}{
\begingroup%
  \makeatletter%
  \providecommand\color[2][]{%
    \errmessage{(Inkscape) Color is used for the text in Inkscape, but the package 'color.sty' is not loaded}%
    \renewcommand\color[2][]{}%
  }%
  \providecommand\transparent[1]{%
    \errmessage{(Inkscape) Transparency is used (non-zero) for the text in Inkscape, but the package 'transparent.sty' is not loaded}%
    \renewcommand\transparent[1]{}%
  }%
  \providecommand\rotatebox[2]{#2}%
  \newcommand*\fsize{\dimexpr\f@size pt\relax}%
  \newcommand*\lineheight[1]{\fontsize{\fsize}{#1\fsize}\selectfont}%
  \ifx\svgwidth\undefined%
    \setlength{\unitlength}{285.22374417bp}%
    \ifx\svgscale\undefined%
      \relax%
    \else%
      \setlength{\unitlength}{\unitlength * \real{\svgscale}}%
    \fi%
  \else%
    \setlength{\unitlength}{\svgwidth}%
  \fi%
  \global\let\svgwidth\undefined%
  \global\let\svgscale\undefined%
  \makeatother%
  \begin{picture}(1,0.79681514)%
    \lineheight{1}%
    \setlength\tabcolsep{0pt}%
    \put(0,0){\includegraphics[width=\unitlength,page=1]{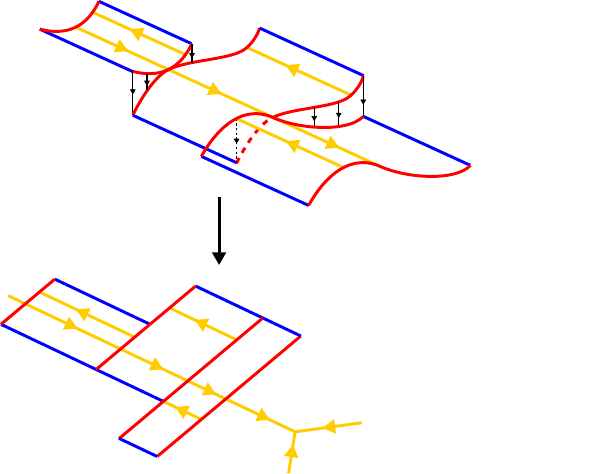}}%
    \put(0.89180643,0.51408166){\color[rgb]{0,0,0}\makebox(0,0)[lt]{\lineheight{1.25}\smash{\begin{tabular}[t]{l}$M$\end{tabular}}}}%
    \put(0.89180643,0.07265981){\color[rgb]{0,0,0}\makebox(0,0)[lt]{\lineheight{1.25}\smash{\begin{tabular}[t]{l}$\mathcal{O}$\end{tabular}}}}%
    \put(0.48661755,0.74011318){\color[rgb]{0,0,0}\makebox(0,0)[lt]{\lineheight{1.25}\smash{\begin{tabular}[t]{l}$R_{v_i}$\end{tabular}}}}%
    \put(0.62029373,0.63438243){\color[rgb]{0,0,0}\makebox(0,0)[lt]{\lineheight{1.25}\smash{\begin{tabular}[t]{l}$Q_i$\end{tabular}}}}%
    \put(0.69179981,0.57272323){\color[rgb]{0,0,0}\makebox(0,0)[lt]{\lineheight{1.25}\smash{\begin{tabular}[t]{l}$R_{v_{i+1}}$\end{tabular}}}}%
    \put(0.37114225,0.65312908){\color[rgb]{1,0.8,0}\makebox(0,0)[lt]{\lineheight{1.25}\smash{\begin{tabular}[t]{l}$c$\end{tabular}}}}%
  \end{picture}%
\endgroup%
}
    \caption{The surfaces $P_c$ are unions of images of $R_{v_i}$ and connecting surfaces $Q_i$.}
    \label{fig:stableglue3}
\end{figure}
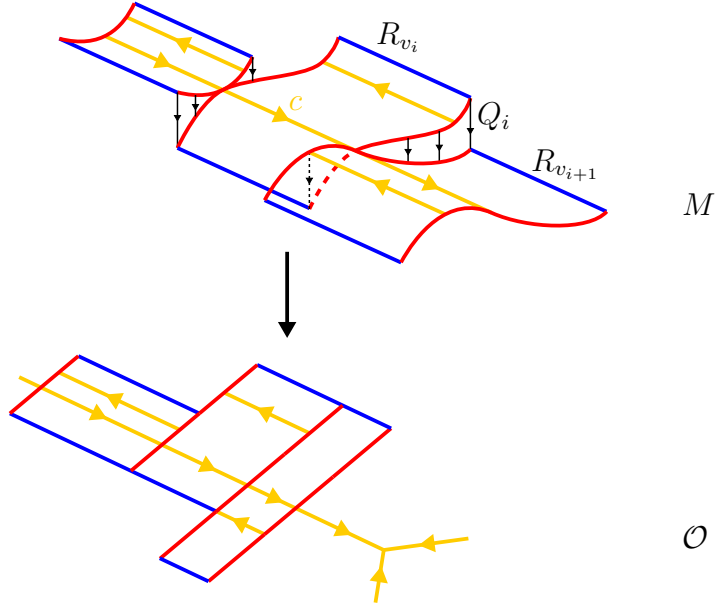

Let $K$ be the component of $N(B) \backslash \backslash N(\Phi_{\red})$ which contains both $\partial^+ R_{v_i}$ and $\partial^- R_{v_{i+1}}$. To construct $Q_i$, consider the universal cover $\widetilde{K}$ of $K$. Note that there is a notion of height on $\widetilde{K}$, with respect to the $y$-coordinate in the model described in \Cref{lemma:redflowgraphcompl}, as opposed to $K$ where that coordinate is circular. The oriented 1-dimensional foliations on the stable faces of $K$ lift to oriented 1-dimensional foliations on the stable faces of $\widetilde{K}$. Similarly, the $I$-fibering on $K$ lifts to an $I$-fibering on $\widetilde{K}$. A component of the branch locus of $K$ is a subinterval of $c$ connecting $\partial^+ R_{v_i}$ and $\partial^- R_{v_{i+1}}$. Lift $\partial^+ R_{v_i}$ and $\partial^- R_{v_{i+1}}$ together with the connecting subinterval to $\widetilde{K}$. We abuse notation and still call the lifts of $\partial^+ R_{v_i}$ and $\partial^- R_{v_{i+1}}$ by the same name.

We claim that the criss-cross property stated in \Cref{lemma:redflowgraphcompl} implies that the component of $\partial^u \widetilde{K}$ containing an endpoint of $\partial^+ R_{v_i}$ does not lie below that containing the endpoint of $\partial^- R_{v_{i+1}}$ on the same side. To check this, let $J$ be the component of $B \backslash \backslash \Phi_{\red}$ corresponding to $K$ and consider the universal cover $\widetilde{J}$ of $J$. Without loss of generality, we can assume that the lifted connecting subinterval corresponds to $y=-x$ in the branch locus of $\widetilde{J}$, under the model described in \Cref{lemma:redflowgraphcompl}. Notice that in this case $\partial^+ R_{v_i}$ lies along the stable face of $\widetilde{K}$ corresponding to $x=0$ while $\partial^- R_{v_{i+1}}$ lies along that corresponding to $x=1$.

Meanwhile, there are two sides to $[0,1] \times \mathbb{R} \subset \widetilde{J}$ for which its tongues are attached along. Correspondingly, the components of $\partial^u \widetilde{K}$ also lie on two sides, and we need to check our claim on the two sides separately. On the side for which the tongue is attached along $y=-x$, the component of $\partial^u \widetilde{K}$ containing the endpoint of $\partial^+ R_{v_i}$ is the same as that containing the endpoint of $\partial^- R_{v_{i+1}}$, unless the tongue attached along $y=x-2$ is done so on this same side (hence $w=1$ or $2$), in which case the former is higher up than the latter. 

On the other side, the component of $\partial^u \widetilde{K}$ containing the endpoint of $\partial^+ R_{v_i}$ is strictly higher up than that containing the endpoint of $\partial^- R_{v_{i+1}}$, unless $w=1$ and the tongue attached along $y=x-2$ is done so on the same side as that along $y=-x$, in which case the two component coincide.

This claim implies that if we transfer $\partial^+ R_{v_i}$ to the other stable face of $\widetilde{K}$ using the homeomorphism induced by the $I$-fibering, and call the image $\partial^{++} R_{v_i}$ temporarily, then the leaves of the 1-dimensional foliation passing through $\partial^{++} R_{v_i}$ will also pass through $\partial^- R_{v_{i+1}}$. Hence we can take the union of subintervals of leaves going between $\partial^{++} R_{v_i}$ and $\partial^- R_{v_{i+1}}$, project down to $K$ then collapse to $N(B)$ to get $Q_i$. 

Here we caution that $\partial^{++} R_{v_i}$ may not lie strictly above $\partial^- R_{v_{i+1}}$, that is, some of the leaves that form $Q_i$ may be oriented from $\partial^- R_{v_{i+1}}$ to $\partial^{++} R_{v_i}$, as indicated in \Cref{fig:stableglue3}. The important feature here, rather, is that $Q_i \cap \partial^+ R_{v_i} = \partial^+ R_{v_i}$ and $Q_i \cap \partial^- R_{v_{i+1}} \subset \partial^- R_{v_{i+1}}$.

Let $P_c$ be the union of the images of $R_{v_i}$ and $Q_i$. Notice that $P_c$ is an immersed annulus with corners in $N(B)$, which contains $c$ and intersects the same leaves (of the oriented 1-dimensional foliation) as the images of $R_{v_i}$. $P_c$ is not strictly speaking a local section, since the $Q_i$ are tangent along the flow. One can perturb $P_c$ so that it becomes transverse to the flow everywhere, but since we ultimately only use the structure of $P_c$ after projecting to the orbit space, this is not necessary for our arguments. We will point out more features of $P_c$ in \Cref{subsec:pAproof}.

\subsection{Gluing along the unstable boundary} \label{subsec:unstableglue}

We will follow the same strategy to glue $N(B)$ along its boundary across its complementary regions in $M(s)$, so that the foliations on $N(B)$ descend to respective foliations on $M(s)$, for suitable $s$.

\begin{lemma} \label{lemma:unstableglue}
The components of $M \backslash \backslash B$ are (once-punctured cusped polygons)$\times S^1$. Fix one of these components $T$, $\partial T$ is naturally decomposed into a number of annulus faces meeting along cusp circles. $\Phi_{\red}$ on each annulus face consists of one or two oriented circles along with some branches attached inductively which exit through the cusp circles. Each oriented circle of $\Phi_{\red} \cap \partial T$ has isotopy class equal to the ladderpole slope, in particular they are parallel. See \Cref{fig:unstableglue} left. 
\end{lemma}
\begin{proof}
The first sentence is \Cref{prop:complbranchsurf}, and the second sentence follows easily. For the third sentence, note that $\Phi_{\red}$ on each annulus face $A$ of $\partial T$ is an oriented train track with branches exiting through the boundary and with only diverging switches. This forces $\Phi_{\red} \cap A$ to be a union of parallel circles with branches attached inductively which exit through the cusp circles. There cannot be more than 2 circles on each annulus or else the inner circles will have no branches attached. These will then contradict $\Phi_{\red}$ being strongly connected, since these circles can have no outgoing edges. Meanwhile, each cusp circle meets $\Phi_{\red}$, otherwise there will be a component of $B \backslash \backslash \Phi_{\red}$ carrying a circle in its branch locus, contradicting \Cref{lemma:redflowgraphcompl}. These two observations imply the third sentence.

To show the fourth sentence, it suffices to show that the oriented circles of $\Phi_{\red} \cap \partial T$ are oriented coherently. We will show this using \cite[Lemma 5.4]{LMT20}. (Again, we caution the reader that the flow graph in \cite{LMT20} is oriented in the opposite direction compared to this paper.) Their lemma implies that each cycle of $\Phi \cap \partial T$ has slope equal to the ladderpole slope. Since $\Phi_{\red}$ is a subgraph of $\Phi$, the same statement holds for $\Phi_{\red} \cap \partial T$. We remark that the third sentence of our lemma follows from their lemma as well.
\end{proof}

\begin{figure}
    \centering
    \resizebox{!}{6cm}{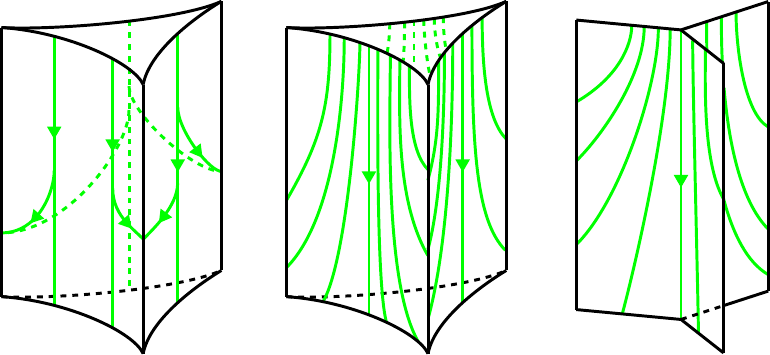}
    \caption{Collapsing across complementary regions of $M(s) \backslash \backslash N(B)$.}
    \label{fig:unstableglue}
\end{figure}

Recall that we can recover the oriented 1-dimensional foliation on the boundary of each component of $M \backslash \backslash N(B)$ from the portion of $\Phi_{\red}$ on the boundary of the corresponding component of $M \backslash \backslash B$, as described in \Cref{rmk:stableglueboundary}. From \Cref{lemma:unstableglue}, it follows that the foliation on each annulus will consist of a band of $S^1$ leaves, with all other leaves spiralling out of the band and exiting through the cusp circles. Using the transverse measure $du$ on the rectangles $\{ (u,t): |u| \leq w'_w \lambda^{-t}, t \in [0,1] \}$, one can see that the band of $S^1$ leaves actually just consists of one leaf. For if the band has nonzero width, pick a local section to the foliation within the band and observe that the backward return map of the foliation has to be contracting, providing a contradiction. See \Cref{fig:unstableglue} middle.

Now fix a component $T$ of $M \backslash \backslash N(B)$, fix an identification $T \cong$ (punctured cusped $k$-gon)$\times S^1$ and take the orientation on the punctured cusped $k$-gon to be the one induced from that of $T \subset M$ and the ladderpole slope. Label the annulus faces of $\partial T$ as $A^{(1)},...,A^{(k)}$ cyclically according to the orientations we have chosen. Suppose a slope $s$ is given on the torus end of $T$, whose geometric intersection number with the ladderpole class $l$ is at least $1$. We want to produce a `pronged $I$-fibering' of $T(s)$ preserving the leaves on $A^{(j)}$. By this we mean a decomposition of $T(s)$ into intervals and $|\langle s,l \rangle|$-prongs, so that the endpoints of the prongs lie along the unique $S^1$ leaves on the $A^{(j)}$, and so that the homeomorphisms on the halves of $A^{(j)}$ induced from the interval fibers preserve the leaves of the oriented 1-dimensional foliations.

Let $m$ be the meridian of $T$ in the description $T \cong$ (punctured cusped $k$-gon)$\times S^1$, i.e. the isotopy class of $\partial$(punctured cusped $k$-gon)$\times \{t_0\}$. We pick $\frac{|\langle s,l \rangle|}{k}$ local sections $J^{(j)}_1,...,J^{(j)}_{\frac{|\langle s,l \rangle|}{k}}$ near the $S^1$ leaf on $A^{(j)}$, so that we have backward return maps $J^{(j)}_i \to J^{(j)}_{i+1}$ (sub-indices taken mod $\frac{|\langle s,l \rangle|}{k}$) of the oriented foliation. Let $s^{(j)}_i$ be where $J^{(j)}_i$ meets the $S^1$ leaf on $A^{(j)}$. First define the fibering on the cusp circles of $T(s)$, where it is uniquely determined and degenerate, then extend the fibering over subintervals of leaves that start at the interior of $J^{(j)}_i$ and end on the cusp circles, using the fact that the union of these trajectories are unions of rectangles foliated as products. Finally apply \Cref{lemma:expandextend} to $I_i=$ subinterval of $J^{(j)}_i$ to the left of $s^{(j)}_i$ and $I'_i=$ subinterval of $J^{(j+1)}_{i\pm \langle s,m \rangle}$ to the right of $s^{(j+1)}_{i\pm \langle s,m \rangle}$ for each $j$ (mod $k$), where the $\pm$ is taken to be the sign of $\langle s,l \rangle$, to produce the pronged $I$-fibering in the remaining subset of $T(s)$. 

Now given a collection of slopes $s$ on each torus end of $M$ with $|\langle s,l \rangle| \geq 1$, collapse along the fibers of the pronged $I$-fibering in each corresponding component of $M(s) \backslash \backslash N(B)$. See \Cref{fig:unstableglue} right for an illustration of a collapsed component of $M(s) \backslash \backslash N(B)$ with $|\langle s,l \rangle|=3$. The oriented 1-dimensional foliation on $N(B)$ descends to a decomposition of $M(s)$, which is a genuine oriented 1-dimensional foliation, since all the branching of the leaves in $N(B)$ occurs along the branch locus of $N(B)$, and we have collapsed all backward trajectories ending at those points. In particular this oriented 1-dimensional foliation can be continuously parametrized into a topological flow. 

For each torus end of $M$, call the image of the $S^1$ leaves on $A^{(j)}$ after collapsing the \textit{core orbit} of the corresponding end.

\subsection{Showing pseudo-Anosovity and no perfect fits} \label{subsec:pAproof}

We have constructed our flow at this point. It remains to check that it satisfies the properties described in \Cref{thm:vtpAflow}. We first recall the relevant definitions.

\begin{defn} \label{defn:phorbit}
Consider the map $\begin{pmatrix} \lambda^{-1} & 0\\ 0 & \lambda \end{pmatrix}: \mathbb{R}^2 \to \mathbb{R}^2$. This preserves the foliations of $\mathbb{R}^2$ by horizontal and vertical lines. Let $\phi_{n,0,\lambda}:\mathbb{R}^2 \to \mathbb{R}^2$ be the lift of this map over $z \mapsto z^{\frac{n}{2}}$ that preserves the lifts of the quadrants. (When $n$ is odd, one has to choose a branch of $z \mapsto z^{\frac{n}{2}}$ but it is easy to see that the result is independent of the choice.) Let $\phi_{n,k, \lambda}: \mathbb{R}^2 \to \mathbb{R}^2$ be the composition of $\phi_{n,0,\lambda}$ and rotating by $\frac{2\pi k}{n}$ anticlockwise. Also pull back the foliations of $\mathbb{R}^2$ by horizontal and vertical lines. The resulting two singular foliations are preserved by $\phi_{n,k}$, call them $l^s, l^u$ respectively. Let $\Phi_{n,k,\lambda}$ be the mapping torus of $\phi_{n,k, \lambda}$, $\Lambda^s, \Lambda^u$ be suspensions of $l^s, l^u$ respectively, and consider the suspension flow on $\Phi_{n,k, \lambda}$. Call the suspension of the origin the \textit{pseudo-hyperbolic orbit} of $\Phi_{n,k,\lambda}$. The local behaviour of the flow near the pseudo-hyperbolic orbit serves as the local model for singular orbits in a pseudo-Anosov flow.
\end{defn}

\begin{defn} \label{defn:pAflow}
A \textit{smooth pseudo-Anosov flow} on a closed smooth 3-manifold $N$ is a continuous flow $\phi^t$, along with a path metric $d$ on $N$, which is induced from a Riemannian metric $g$ away from the singular orbits, satisfying:
\begin{itemize}
    \item [(S1)] There is a finite collection of closed orbits $\{\gamma_1, ..., \gamma_s \}$, called the \textit{singular orbits} such that $\phi^t$ is smooth away from the singular orbits.
    \item [(S2)] Away from the singular orbits, there is a splitting of the tangent bundle into three $\phi^t$-invariant line bundles $TM=E^s \oplus E^u \oplus T\phi^t$ such that $$|d\phi^t(v)| < C \lambda^{-t} |v|$$ for every $v \in E^s, t>0$, and $$|d\phi^t(v)| < C \lambda^t |v|$$ for every $v \in E^u, t<0$, for some $C, \lambda>1$.
    \item [(S3)] Each singular orbit $\gamma_i$ has a neighborhood $N_i$ and a map $f_i$ sending $N_i$ to a neighborhood of the pseudo-hyperbolic orbit in $\Phi_{n_i, k_i, \lambda}$, for some $n_i \geq 3$, such that
    \begin{itemize}
        \item $f_i$ is bi-Lipschitz on $N_i$ and smooth away from $\gamma_i$,
        \item $f_i$ preserves the orbits, and
        \item $f_i$ sends $E^s, E^u$ to line bundles tangent to $\Lambda^s, \Lambda^u$ respectively.
    \end{itemize}
    In this case, we say that $\gamma_i$ is \textit{$n_i$-pronged}.
\end{itemize}

A \textit{topological pseudo-Anosov flow} on a closed 3-manifold $N$ is a continuous flow $\phi^t$ satisfying:
\begin{itemize}
    \item [(T1)] There is a finite collection of closed orbits $\{\gamma_1, ..., \gamma_s \}$, called the \textit{singular orbits}, and two singular 2-dimensional foliations $\Lambda^s, \Lambda^u$, called the stable and unstable foliations respectively, which are non-singular away from the singular orbits.
    \item [(T2)] Away from the singular orbits, every point has a neighborhood which is a \textit{flow box}, i.e. a set of the form $I_s \times I_u \times I_t$ such that the flow lines are the lines with constant $s$ and $u$, the stable and unstable foliations are the foliations by planes with constant $u$ or $s$ coordinate respectively.
    \item [(T3)] There is a \textit{Markov partition}, i.e. there is a finite collection of flow boxes $\{ I^{(i)}_s \times I^{(i)}_u \times I_t \}_i$ covering $N$ with disjoint interiors, such that $$(I^{(i)}_s \times I^{(i)}_u \times \{1\}) \cap (I^{(j)}_s \times I^{(j)}_u \times \{0\}) = \bigcup_k J^{(ij.k)}_s \times I^{(i)}_u \times \{1\} = \bigcup_k I^{(j)}_s \times J^{(ji.k)}_u \times \{0\}$$ for some finite collection of subintervals $J^{(ij,k)}_s \subset I^{(i)}_s$ and $J^{(ji,k)}_u \subset J^{(j)}_u$.
    \item [(T4)] Pick a path metric $d$ on $M$. For every $p,q$ on the same stable leaf, there exists an orientation-preserving homeomorphism $T:(-\infty,\infty) \to (-\infty,\infty)$ such that $$\lim_{t \to \infty} d_N(\phi^t(p),\phi^{T(t)}(q))=0$$ Respectively, for every $p,q$ on the same unstable leaf, there exists an orientation-preserving homeomorphism $T:(-\infty,\infty) \to (-\infty,\infty)$ such that $$\lim_{t \to -\infty} d_N(\phi^t(p),\phi^{T(t)}(q))=0$$ 
    \item [(T5)] Each singular orbit $\gamma_i$ has a neighborhood $N_i$ and a continuous map $f_i$ sending $N_i$ to a neighborhood of the pseudo-hyperbolic orbit in $\Phi_{n_i, k_i, \lambda}$, for some $n_i \geq 3, \lambda>1$, such that
    \begin{itemize}
        \item $f_i$ preserves the orbits, and
        \item $f_i$ preserves $\Lambda^s, \Lambda^u$ on the two sets.
    \end{itemize}
    In this case, we say that $\gamma_i$ is \textit{$n_i$-pronged}.
\end{itemize}
\end{defn}

\begin{defn}
A (smooth/topological) pseudo-Anosov flow is \textit{transitive} if it has a dense orbit.
\end{defn}

It has long been a folklore fact that the notions of smooth and topological pseudo-Anosov flows are essentially equivalent in the case when the flows are transitive. The easier direction is that a transitive smooth pseudo-Anosov flow is a topological pseudo-Anosov flow. Indeed, the only nontrivial facts to check are that $E^s \oplus T\phi^t$ and $E^u \oplus T\phi^t$ are integrable (away from singular orbits), which follows from stable manifold theory (see for example, \cite[Chapter 17.4]{KH95}), and that there exists a Markov partition, which follows from the arguments in \cite[Section 2]{Rat73}. For the other direction, we have

\begin{thm} \label{thm:top2smooth}
Given a transitive topological pseudo-Anosov flow $\phi^t$ on a closed 3-manifold $N$, there exists a homeomorphism $F:N \to N$ and a smooth pseudo-Anosov flow $\hat{\phi^t}$ (with respect to some smooth structure on $N$), such that $F$ maps the trajectories of $\phi^t$ to that of $\hat{\phi^t}$ (preserving their orientations), i.e. $\phi^t$ is \textit{$C^0$-orbit equivalent} to $\hat{\phi^t}$.
\end{thm}

This theorem has been proven recently by Shannon in \cite{Sha21} in the case of transitive Anosov flows, i.e. when there are no singular orbits. His methods in fact generalize immediately to the general case of transitive pseudo-Anosov flows. We will explain more carefully how to apply his arguments to prove \Cref{thm:top2smooth} in \Cref{subsec:smooth}. We remark that it is still open whether \Cref{thm:top2smooth} is true without the hypothesis of transitivity. Also see \cite[Section 3.1]{Mos96} for a related discussion. 

Before we recall the definition of no perfect fits, we need a nontrivial fact: Suppose a closed 3-manifold $N$ admits a (smooth/topological) pseudo-Anosov flow. Lift everything up to the universal cover $\widetilde{N}$. It is shown in \cite[Proposition 4.1]{FM01} that the orbit space $\mathcal{O}$ of the flow on $\widetilde{N}$ is homeomorphic to $\mathbb{R}^2$, and inherits two (possibly singular) 1-dimensional foliations $\mathcal{O}^s, \mathcal{O}^u$.

\begin{rmk}
\cite{FM01} only deals with smooth pseudo-Anosov flows, but given \Cref{thm:top2smooth}, the facts stated above holds for topological pseudo-Anosov flows as well. Alternatively, the proof of \cite[Proposition 4.1]{FM01} is valid verbatim for topological pseudo-Anosov flows.
\end{rmk}

We will in fact generalize the definition of no perfect fits slightly, compared to the usual definition found in, for example, \cite{Fen12}.

\begin{defn} \label{defn:perfectfit}
Let $\phi$ be a pseudo-Anosov flow on a closed 3-manifold $N$, and let $\{c_1,...,c_k\}$ be a collection of closed orbits of $\phi$. Lift these up to a flow $\widetilde{\phi}$ on the universal cover $\widetilde{N}$ together with a collection of orbits $\{ \widetilde{c_i} \}$ which are the preimages of $\{c_i\}$.

A \textit{perfect fit rectangle} is a rectangle-with-one-ideal-vertex properly embedded in $\mathcal{O}$ such that 2 opposite sides of the rectangle lie along leaves of $\mathcal{O}^s$ and the remaining 2 opposite sides lie along leaves of $\mathcal{O}^u$, and such that the restrictions of $\mathcal{O}^s$ and $\mathcal{O}^u$ to the rectangle foliate it as a product, i.e. conjugate to the foliations of $[0,1]^2 \backslash \{(1,1)\}$ by vertical and horizontal lines. See \Cref{fig:perfectfitdefn}. 

The collection of orbits $\{ \widetilde{c_i} \}$ determines a set $\mathcal{C}$ in $\mathcal{O}$. We will say that $\phi$ has no perfect fits relative to $\{c_1,...,c_k \}$ if there are no perfect fit rectangles in $\mathcal{O}$ disjoint from $\mathcal{C}$.
\end{defn}

\begin{figure}
    \centering
    \fontsize{18pt}{18pt}\selectfont
    \resizebox{!}{4cm}{
\begingroup%
  \makeatletter%
  \providecommand\color[2][]{%
    \errmessage{(Inkscape) Color is used for the text in Inkscape, but the package 'color.sty' is not loaded}%
    \renewcommand\color[2][]{}%
  }%
  \providecommand\transparent[1]{%
    \errmessage{(Inkscape) Transparency is used (non-zero) for the text in Inkscape, but the package 'transparent.sty' is not loaded}%
    \renewcommand\transparent[1]{}%
  }%
  \providecommand\rotatebox[2]{#2}%
  \newcommand*\fsize{\dimexpr\f@size pt\relax}%
  \newcommand*\lineheight[1]{\fontsize{\fsize}{#1\fsize}\selectfont}%
  \ifx\svgwidth\undefined%
    \setlength{\unitlength}{231.3339041bp}%
    \ifx\svgscale\undefined%
      \relax%
    \else%
      \setlength{\unitlength}{\unitlength * \real{\svgscale}}%
    \fi%
  \else%
    \setlength{\unitlength}{\svgwidth}%
  \fi%
  \global\let\svgwidth\undefined%
  \global\let\svgscale\undefined%
  \makeatother%
  \begin{picture}(1,0.81198695)%
    \lineheight{1}%
    \setlength\tabcolsep{0pt}%
    \put(0,0){\includegraphics[width=\unitlength,page=1]{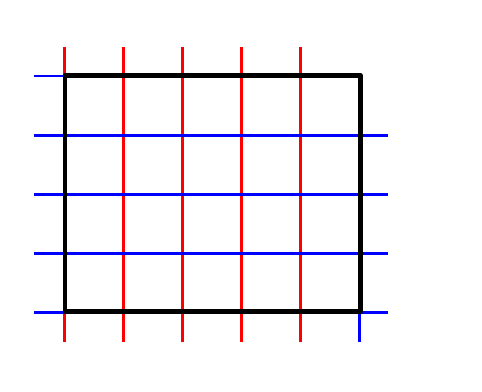}}%
    \put(0.81560745,0.37395544){\color[rgb]{0,0,0}\makebox(0,0)[lt]{\lineheight{1.25}\smash{\begin{tabular}[t]{l}$F$\end{tabular}}}}%
    \put(0.39449472,0.73215106){\color[rgb]{0,0,0}\makebox(0,0)[lt]{\lineheight{1.25}\smash{\begin{tabular}[t]{l}$G$\end{tabular}}}}%
    \put(-0.00785188,0.36925456){\color[rgb]{0,0,0}\makebox(0,0)[lt]{\lineheight{1.25}\smash{\begin{tabular}[t]{l}$H$\end{tabular}}}}%
    \put(0.38866443,0.01828728){\color[rgb]{0,0,0}\makebox(0,0)[lt]{\lineheight{1.25}\smash{\begin{tabular}[t]{l}$K$\end{tabular}}}}%
    \put(0,0){\includegraphics[width=\unitlength,page=2]{perfectfitdefn.pdf}}%
  \end{picture}%
\endgroup%
}
    \caption{A perfect fit rectangle.}
    \label{fig:perfectfitdefn}
\end{figure}

Given a perfect fit rectangle with sides $F,H$ along leaves of $\mathcal{O}^s$, sides $G,K$ along leaves of $\mathcal{O}^u$, and sides $F,G$ adjacent to the ideal vertex, notice that we can always choose a smaller perfect fit rectangle with sides $F,G,H',K'$ where $H'$ is closer to $F$ than $H$ and $K'$ is closer to $G$ than $K$. We will frequently use this observation when analyzing perfect fit rectangles.

A pseudo-Anosov flow is without perfect fits according to the definition in \cite{Fen12} if and only if it is without perfect fits relative to $\varnothing$ in \Cref{defn:perfectfit}. This is in turn equivalent to the pseudo-Anosov flow having no perfect fits relative to the set of its singular orbits, since by definition a perfect fit rectangle cannot contain a singular point in its interior, and we can always choose a smaller perfect fit rectangle that does not contain any singular points on the boundary as well.

It is immediate from their definitions that the property of transitivity and no perfect fits relative to a collection of orbits is preserved under $C^0$-orbit equivalence. Hence in view of \Cref{thm:top2smooth}, it suffices for us to show \Cref{thm:vtpAflow} with `pseudo-Anosov flow' taken to mean topological pseudo-Anosov flow.

We first verify that our flow on $M(s)$ is a topological pseudo-Anosov flow. For (T1), take the collection of singular orbits to be the core orbits of the ends of $M$ with $|\langle s,l \rangle| \geq 3$, and take $\Lambda^s$ and $\Lambda^u$ to be the stable and unstable foliations on $M(s)$ respectively.

For (T2), one can construct flow box neighborhoods (away from the singular orbits) in $M(s)$ by piecing together those in $N(\Phi_{\red})$.

We will call the images of the flow boxes $Z_e$ in $M(s)$ after collapsing $N(B) \backslash \backslash N(\Phi_{\red})$ and $M(s) \backslash \backslash N(B)$ by the same name. These form a Markov partition for the flow by definition, verifying (T3).

Take $d_{M(s)}$ to be the path metric on $M(s)$ induced by the Euclidean metrics $d^{\mathrm{Eucl}}_{Z_e}$ on $Z_e$. This means that we define the length of a path by summing over the $d^{\mathrm{Eucl}}_{Z_e}$-lengths of the portions within each $Z_e$, and define the distance between two points to be the infimum length of paths between them. Here for a path lying on more than one $Z_e$, one is allowed to calculate its length by the Euclidean metric on any of them. Note that since each $Z_e$ is compact, the identity map on $Z_e$ is uniformly continuous with respect to $d^{\mathrm{Eucl}}_{Z_e}$ and $d_{M(s)}|_{Z_e}$.

Let us prove (T4) in the case when $p,q$ lie on the same stable leaf; the proof for when $p,q$ lie on the same unstable leaf is symmetric. Notice that it suffices to prove this when $p,q$ are close to each other, so we can assume that $p,q$ lie in the same flow box $Z_e$. We parametrize the orbits through $p$ and $q$ so that they go through each flow box in unit time. Then we would have $$d^{\mathrm{Eucl}}_{Z_e}(\phi^t(p),\phi^t(q))<C \lambda^{-t}$$ whenever $\phi^t(p)$ and $\phi^t(q)$ lie in the flow box $Z_e$ (where $C$ depends on $p,q$). So (T4) follows from uniform continuity between $d^{\mathrm{Eucl}}_{Z_e}$ and $d_{M(s)}$.

For (T5), the neighborhoods $N_i$ can be constructed by piecing up the flow box neighborhoods in $N(\Phi_{\red})$ again. We then have to find conjugations $f_i$ between $N_i$ and neighborhoods of the pseudo-hyperbolic orbits in $\Phi_{n_i,k_i,\lambda}$ for $n_i=|\langle s_i,l_i \rangle|$. These can be constructed by choosing local sections to the flow in $N_i$ and in $\Phi_{n_i,k_i,\lambda}$, then applying \Cref{lemma:expandextend} to the prongs that are the intersections of these local sections with the singular stable and unstable leaves, taking $f$ there to be an arbitrary homeomorphism. This constructs $f_i$ on the singular stable and unstable leaves on the local sections, then $f_i$ is uniquely determined on the local sections (after possibly shrinking them), by requiring that it preserves the stable and unstable foliations. Then we complete the construction of $f_i$ by following the trajectories as they go around $N_i$. Observe that $|\langle s, l \rangle| \geq 2$ is required here, otherwise we will have 1-pronged singular orbits. (Meanwhile `2-pronged singular orbits' are just non-singular orbits.)

We then show that the flow is transitive. Pick an infinite path in $\Phi_{\red}$ which contains all finite paths in $\Phi_{\red}$ as sub-paths. Such a path can be constructed using the fact that $\Phi_{\red}$ is strongly connected. Write the path as $v_1 \overset{e_1}{\to} v_2 \overset{e_2}{\to} ...$. We claim that there is an infinite flow line in $M(s)$ passing through $Z_{e_1},Z_{e_2},...$ in that order. To show this claim, consider the subset in $Z_{e_1}$ of points whose forward trajectories flow through $Z_{e_1},...,Z_{e_k}$. These subsets will be nonempty decreasing sub-flow boxes of $Z_{e_1}$, hence taking the intersection as $k \to \infty$, we get points in $Z_{e_1}$ whose flow lines have the required property. Pick one of these flow lines and denote it as $l$. Now for any point $y \in M(s)$, suppose the forward and backward trajectories of $y$ pass through $Z_{d_0}, Z_{d_1}, Z_{d_2}...$ and $Z_{d_0}, Z_{d_{-1}}, Z_{d_{-2}},...$ respectively. Note that these sequences may not be unique, due to the fact that leaves of the 1-dimensional foliation on $N(\Phi_{\red})$ have forward and backward branching, but we can always just make a choice. Now fix some $N>0$, the flow line $l$ passes through $Z_{d_{-N}},...,Z_{d_N}$ in that order at some point, by construction. Since both $l$ and the backward flow line through $y$ pass through $Z_{d_{-N}},...,Z_{d_0}$, they intersect $Z_{d_0}$ in flow lines not more than $w_{v_{-N}} \lambda^{-N}$ apart in the $d^{\mathrm{Eucl}}_{Z_{d_0}}$ metric. Similarly, looking in the forward direction, $l$ and the flow line through $y$ are not more than $w'_{w_N} \lambda^{-N}$ apart in the $d^{\mathrm{Eucl}}_{Z_{d_0}}$ metric. Hence we conclude that $l$ passes through a $(\max w_v+\max w'_v) \lambda^{-N}$ $d^{\mathrm{Eucl}}_{Z_{d_0}}$-neighborhood of $y$. Letting $N \to \infty$ and using uniform continuity between $d^{\mathrm{Eucl}}_{Z_{d_0}}$ and $d_{M(s)}$, this proves transitivity.

Finally we prove that the flow has no perfect fits relative to the core orbits, which we denote as $c_i$. To do this we will make use of the surfaces $P_c$ constructed in \Cref{subsec:unstableglue}. There, we defined these as subsets of $N(B)$, but here we will consider their images in $M(s)$ after collapsing $M \backslash \backslash N(B)$ and still refer to them by the same name. Consider the lifts of $P_c$ to the universal cover $\widetilde{M(s)}$. We denote these as $\widetilde{P_{\widetilde{c}}}$, where $\widetilde{c}$ ranges over components of the branch locus of $\widetilde{B}$ that are lifts of $c$, which we will conflate with the corresponding components of the branch locus of $\widetilde{N(B)}$. After collapsing $\widetilde{M(s)} \backslash \backslash \widetilde{N(B)}$, each $\widetilde{c} \subset \widetilde{M(s)}$ lies on a leaf of $\widetilde{\Lambda^u}$ and is transverse to the flow, meeting exactly those flow lines which spiral out of the lift of the core orbit produced by collapsing the component of $\widetilde{M(s)} \backslash \backslash \widetilde{N(B)}$ which $\widetilde{c}$ opens up towards. Moreover, as one goes along $\widetilde{c}$ in its orientation lifted from $c$, we meet flow lines that spiral closer to the lift of the core orbit. Hence $\widetilde{c}$ projects to an open half leaf of $\mathcal{O}^u$ on $\mathcal{O}$, which under the orientation of $\widetilde{c}$, limits to the lift of the core orbit (as a point in $\mathcal{O}$) in the forwards direction and escapes to infinity in $\mathcal{O}$ in the backwards direction. Since the projection is a homeomorphism, we will call this open half leaf $\widetilde{c}$ as well for convenience.

Meanwhile, each $\widetilde{P_{\widetilde{c}}}$ projected to $\mathcal{O}$ will be a surface containing $\widetilde{c}$. More precisely, recall that $P_c$ is the union of the images of $R_{v_i}$ and $Q_i$. Let $\widetilde{R_{v_i}} \subset \widetilde{\Phi_{\red}}$ be the lifts of $R_{v_i}$ that meet $\widetilde{c}$ and $\widetilde{Q_i} \subset \widetilde{N(B)}$ the lifts of $Q_i$ that meet $\widetilde{c}$. Then $\widetilde{P_c}$ is the union of the images of $\widetilde{R_{v_i}}$ and $\widetilde{Q_i}$ after collapsing $\widetilde{N(B)} \backslash \backslash \widetilde{N(\Phi_{\red})}$ and $\widetilde{M(s)} \backslash \backslash \widetilde{N(B)}$. The image of each $\widetilde{R_{v_i}}$ will be projected homeomorphically to a rectangle in $\mathcal{O}$ along $\widetilde{c}$ since these are transverse to the flow, while the image of each $\widetilde{Q_i}$ is tangent to the flow hence will be projected into sides of the rectangles.

The way these rectangles fit together is described by the following lemma.

\begin{lemma} \label{lemma:pcstructure}
The projection of each $\widetilde{P_{\widetilde{c}}}$ is homeomorphic to $\{(x,y) \in (-\infty,0) \times \mathbb{R}: r_1(x) \leq y \leq r_2(x)\}$ where $r_1$ ($r_2$, respectively) is a non-increasing (non-decreasing, respectively) negative (positive, respectively) piecewise constant function, $\widetilde{c}$ is sent to $(-\infty,0) \times \{0\}$ in the direction of increasing $x$, and the restriction of $\Lambda^s$ ($\Lambda^u$, respectively) are sent to the foliation by vertical (horizontal, respectively) lines. See \Cref{fig:stableglue3}.
\end{lemma}
\begin{proof}
The rectangles in $\mathcal{O}$ that are the projections of $\widetilde{R_{v_i}}$ have nonoverlapping interiors and cover $\widetilde{c}$, so the only thing that remains to be shown is the monotonicity of $r_1$ and $r_2$.

To show this, recall that $Q_i \cap \partial^+ R_{v_i} = \partial^+ R_{v_i}$ and $Q_i \cap \partial^- R_{v_{i+1}} \subset \partial^- R_{v_{i+1}}$. This implies that the line that is the image of $\widetilde{Q_i}$ is the right vertical side of the rectangle that is the projection of $\widetilde{R_{v_i}}$ and a subset of the left vertical side of the rectangle that is the projection of $\widetilde{R_{v_{i+1}}}$. In simpler terms, the rectangles get taller or are of equal height as one goes along $\widetilde{c}$, which establishes the lemma.
\end{proof}

For convenience, from now on let us write $\widetilde{P_{\widetilde{c}}}$ and $\widetilde{R_{v_i}}$ for their images in $\mathcal{O}$ as well.

To establish no perfect fits, we will basically argue that if there is a perfect fit rectangle in $\mathcal{O}$ disjoint from the images of $\widetilde{c_i}$, then we can reduce it to sit inside one of these $\widetilde{P_{\widetilde{c}}}$ with the ideal vertex at a corner of the rectangle with larger $x$ coordinate. But from the monotonicity of $r_1$ and $r_2$, such a perfect fit rectangle cannot be properly embedded in $\widetilde{P_{\widetilde{c}}}$, giving us a contradiction.

So suppose we have a perfect fit rectangle in $\mathcal{O}$ disjoint from the images of $\widetilde{c_i}$ with sides $F,H$ along leaves of $\mathcal{O}^s$, sides $G,K$ along leaves of $\mathcal{O}^u$, and sides $F,G$ adjacent to the ideal vertex. Lift $H$ to $\widetilde{M(s)}$. By flowing forward (and possibly rechoosing $K$ to be closer to $G$) we can assume that the lift of $H$ is short enough to be contained within the lift of some flow box, which we denote by $\widetilde{Z_e}$. In particular we can take the lift of $H$ to lie on the image of some $\widetilde{R_v}$. If the lift of $H$ does not meet the image of the branch locus of $\widetilde{N(B)}$ in its interior, i.e. it is fully contained in the bottom face of some $\widetilde{Z_e}$, then we can rechoose the lift of $H$ to be on the top face of $\widetilde{Z_e}$. When we do so the length of the lift (with respect to the Euclidean metric on the $\widetilde{Z_e}$'s) is increased by a factor of $\lambda$, so we cannot continue this operation indefinitely. In other words, we can choose a lift of $H$ which lies on the image of some $\widetilde{R_v}$ and whose interior meets the image of the branch locus of $\widetilde{N(B)}$.

Recall that the components of the branch locus of $\widetilde{N(B)}$ carry orientations induced from those of $B$. We split into two cases depending on how the lift of $H$ meets the image of the branch locus of $\widetilde{N(B)}$ with respect to this orientation. 

Case 1 is if the lift of $H$ meets the image of a component $\widetilde{c}$ of the branch locus oriented in the same direction as $K$ leaving $H$. In this case, by moving $K$ closer to $G$, we can suppose that $K$ lies within $\widetilde{c}$ on $\mathcal{O}$. This uses the assumption that the perfect fit rectangle is disjoint from the images of $\widetilde{c_i}$, otherwise $K$ could go beyond $\widetilde{c}$ on the leaf of $\mathcal{O}^s$. Now by construction, $\widetilde{P_{\widetilde{c}}}$ in $\mathcal{O}$ contains both $H$ and $K$. In the description of \Cref{lemma:pcstructure}, the direction at which $K$ leaves $H$ is the direction of increasing $x$, so the perfect fit rectangle must sit inside $\widetilde{P_{\widetilde{c}}}$ and with the ideal vertex at a corner with larger $x$ coordinate. See \Cref{fig:perfectfitarg1}. As pointed out, this gives us a contradiction.

\begin{figure}
    \centering
    \resizebox{!}{4cm}{
\begingroup%
  \makeatletter%
  \providecommand\color[2][]{%
    \errmessage{(Inkscape) Color is used for the text in Inkscape, but the package 'color.sty' is not loaded}%
    \renewcommand\color[2][]{}%
  }%
  \providecommand\transparent[1]{%
    \errmessage{(Inkscape) Transparency is used (non-zero) for the text in Inkscape, but the package 'transparent.sty' is not loaded}%
    \renewcommand\transparent[1]{}%
  }%
  \providecommand\rotatebox[2]{#2}%
  \newcommand*\fsize{\dimexpr\f@size pt\relax}%
  \newcommand*\lineheight[1]{\fontsize{\fsize}{#1\fsize}\selectfont}%
  \ifx\svgwidth\undefined%
    \setlength{\unitlength}{178.26761197bp}%
    \ifx\svgscale\undefined%
      \relax%
    \else%
      \setlength{\unitlength}{\unitlength * \real{\svgscale}}%
    \fi%
  \else%
    \setlength{\unitlength}{\svgwidth}%
  \fi%
  \global\let\svgwidth\undefined%
  \global\let\svgscale\undefined%
  \makeatother%
  \begin{picture}(1,0.65516939)%
    \lineheight{1}%
    \setlength\tabcolsep{0pt}%
    \put(0,0){\includegraphics[width=\unitlength,page=1]{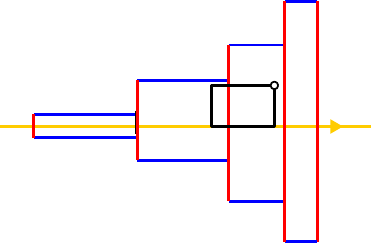}}%
    \put(0.63131425,0.25056032){\color[rgb]{0,0,0}\makebox(0,0)[lt]{\lineheight{1.25}\smash{\begin{tabular}[t]{l}$K$\end{tabular}}}}%
    \put(0.49752869,0.34618231){\color[rgb]{0,0,0}\makebox(0,0)[lt]{\lineheight{1.25}\smash{\begin{tabular}[t]{l}$H$\end{tabular}}}}%
  \end{picture}%
\endgroup%
}
    \caption{In case 1, the ideal vertex of the perfect fit rectangle must lie inside $\widetilde{P_{\widetilde{c}}}$, giving us a contradiction.}
    \label{fig:perfectfitarg1}
\end{figure}

Case 2 is if the lift of $H$ meets a component $\widetilde{c}$ of the branch locus oriented in the different direction as $K$ leaving $H$. As above, we can assume that $K$ lies within $\tilde{c}$ on $\mathcal{O}$. Recall that $\widetilde{P_{\widetilde{c}}}$ consists of a union of rectangles $\widetilde{R_{v_i}}$ along $\tilde{c}$, and $H$ is contained in one of these rectangles. Since $K$ is a finite interval, it only goes through finitely many $\widetilde{R_{v_i}}$, call this finite number the \textit{length} of $K$. Notice that it is sometimes possible to reduce the length of $K$ by moving $H$ closer to $F$ and from $\widetilde{R_{v_{i+1}}}$ to $\widetilde{R_{v_i}}$ while keeping it inside $\widetilde{P_{\widetilde{c}}}$. In any case, we can assume that the length of $K$ is the minimum possible value among all perfect fit rectangles that are disjoint from the images of $\widetilde{c_i}$ in $\mathcal{O}$.

If the length of $K$ is 1, then we have a perfect fit rectangle with $H$ and $K$ within a single rectangle $\widetilde{R_v}$. But then the perfect fit rectangle would not be properly embedded, so this is a contradiction.

Hence the remaining case is if we cannot move $H$ from $\widetilde{R_{v_{i+1}}}$ to $\widetilde{R_{v_i}}$ for some consecutive rectangles $\widetilde{R_{v_i}}, \widetilde{R_{v_{i+1}}}$ along $\widetilde{c}$. Let $\partial^+ \widetilde{R_{v_i}}$ be the side of $\widetilde{R_{v_i}}$ which $\widetilde{c}$ exits, and $\partial^- \widetilde{R_{v_{i+1}}}$ be the side of $\widetilde{R_{v_{i+1}}}$ which $\widetilde{c}$ enters. (This is consistent with the notation used in \Cref{subsec:stableglue}.) We can at least move $H$ onto the image of $\partial^- \widetilde{R_{v_{i+1}}}$ in $\mathcal{O}$ in this case. Then we can pick a lift of $H$ which lies in the image of the stable face of some component $\widetilde{K}$ of $\widetilde{N(B)} \backslash \backslash \widetilde{N(\Phi_{\red})}$ which contains $\partial^- \widetilde{R_{v_{i+1}}}$. The assumption that we cannot move $H$ to $\widetilde{R_{v_i}}$ means that an endpoint of $\partial^+ \widetilde{R_{v_i}}$ lies on a component of $\partial^u \widetilde{K}$ which meets the other stable face at a point whose image in $\widetilde{M(s)}$ has forward trajectory meeting this lift of $H$ in its interior. See \Cref{fig:perfectfitarg2} for an illustration of the situation.

We claim that the criss-cross property stated in \Cref{lemma:redflowgraphcompl} implies that at least one of the endpoints of $\partial^+ \widetilde{R_{v_i}}$ lies on a component of $\partial^u \widetilde{K}$ whose branch locus $L$ is oriented in the different direction as $\widetilde{c}$ and meets the stable face of $\widetilde{K}$ containing $\partial^- \widetilde{R_{v_{i+1}}}$ at a point $z$ on or above $\partial^- \widetilde{R_{v_{i+1}}}$. To check this, we let $J$ be the component of $B \backslash \backslash \Phi_{\red}$ corresponding to $K$, let $\widetilde{J}$ be the universal cover of $J$, and suppose that $\widetilde{c}$ lies along $y=-x$ in the branch locus of $\widetilde{J}$ under the model described in \Cref{lemma:redflowgraphcompl}. 

If $w \geq 2$, then the endpoint of $\partial^+ \widetilde{R_{v_i}}$ away from the side on which the tongue is attached along $y=-x$ to $\widetilde{J}$ lies on the component of $\partial^u \widetilde{K}$ whose branch locus corresponds to $y=x-\frac{2}{w}$, hence is oriented in the different direction as $\widetilde{c}$ and meets the stable face of $\widetilde{K}$ containing $\partial^- \widetilde{R_{v_{i+1}}}$ at a point strictly above $\partial^- \widetilde{R_{v_{i+1}}}$. This is the case that is illustrated in \Cref{fig:perfectfitarg2}.

If $w=1$ and if the tongue attached along $y=x-2$ is done so in the opposite side as that along $y=-x$, we have a similar conclusion: the endpoint of $\partial^+ \widetilde{R_{v_i}}$ away from the side on which the tongue is attached along $y=-x$ lies on the component of $\partial^u \widetilde{K}$ whose branch locus corresponds to $y=x-2$, hence is oriented in the different direction as $\widetilde{c}$, but now meeting the stable face of $\widetilde{K}$ containing $\partial^- \widetilde{R_{v_{i+1}}}$ at a point on $\partial^- \widetilde{R_{v_{i+1}}}$.

Finally, if $w=1$ and if the tongue attached along $y=x-2$ is done so in the same side as that along $y=-x$, then the endpoint of $\partial^+ \widetilde{R_{v_i}}$ on this common side lies on the component of $\partial^u \widetilde{K}$ whose branch locus corresponds to $y=x-2$, hence is oriented in the different direction as $\widetilde{c}$ and meets the stable face of $\widetilde{K}$ containing $\partial^- \widetilde{R_{v_{i+1}}}$ at a point on $\partial^- \widetilde{R_{v_{i+1}}}$.

Now, if, in the claim, the point $z$ lies on $\partial^- \widetilde{R_{v_{i+1}}}$ then we are in case 1 since the component of the branch locus on which $L$ in the claim lies on will meet the lift of $H$ and is oriented in the same direction as $K$ leaving $H$.

If $z$ is strictly above $\partial^- \widetilde{R_{v_{i+1}}}$, then up to moving $K$ closer to $G$, we can rechoose a lift of $H$ on the image of a higher $\widetilde{R_v}$ so that it meets a component of the branch locus of $\widetilde{N(B)}$ oriented in the different direction as $\widetilde{c}$, hence the same direction as $K$ leaving $H$. See \Cref{fig:perfectfitarg2}. Then we have reduced to case 1 again. Hence in any case, we see that the perfect fit rectangle cannot exist.

\begin{figure}
    \centering
    \fontsize{16pt}{16pt}\selectfont
    \resizebox{!}{12cm}{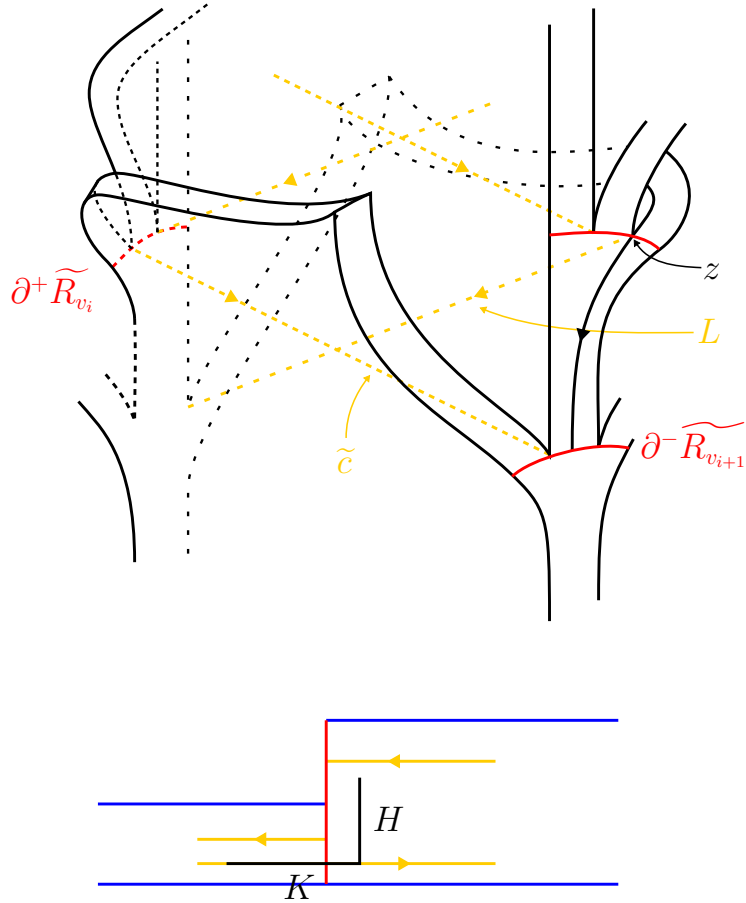}
    \caption{In case 2, up to possibly rechoosing a lift of $H$ that is higher up, we can assume that the lift of $H$ meets a component of the branch locus of $\widetilde{N(B)}$ oriented in the same direction as $K$ leaving $H$.}
    \label{fig:perfectfitarg2}
\end{figure}

\subsection{Proof of \Cref{thm:top2smooth}} \label{subsec:smooth}

By \cite{Bru95} and the standard fact that pseudo-Anosov flows are expansive (which can in turn be proved via a symbolic dynamics argument making use of a Markov partition), $\phi^t$ admits a Birkhoff section. Recall that this means there is an embedding of an oriented surface with boundary $\Sigma \hookrightarrow N$, such that $\partial \Sigma$ is a union of orbits, $\mathrm{int} \Sigma$ is positively transverse to $N$, and every point intersects $\Sigma$ in finite forward time. We briefly sketch an argument here.

For every point $x \in N$, take a local section to the flow near $x$. The local section is divided into $2n$ regions by the stable and unstable leaves containing $x$ (where $n=2$ unless $x$ lies on a singular orbit). Name them $A_1,B_1,...,A_n,B_n$ in a cyclic order. It is standard that the set of periodic orbits in a transitive pseudo-Anosov flow is dense. For example, this can again proved using a symbolic dynamics argument, similar to how we proved transitivity before. Use this fact to find non-singular periodic orbits $p_i$ passing through $A_i$. We can further assume that the stable and unstable leaves which $p_i$ lie on are orientable, since the set of periodic orbits satisfying this is dense as well. Encode the closed orbits $p_i$ by periodic symbolic sequences using a Markov partition, then concatenate these sequences to get a long periodic sequence which corresponds to a closed orbit $q$ in $N$ `shadowing' $p_1,...,p_n$, in that cyclic order. Now construct a $2n$-gon on the local section with vertices at where $p_1,q,p_2,q,...,p_n,q$ meet the local section. Here the multiple listings of $q$ refer to the multiple times $q$ meets the local section as it shadows the $p_i$'s. Connect up the edges $[q,p_k]$ and $[p_k,q]$ on the local section for each $k$ by following the flow. After perturbation, we get an immersed local Birkhoff section. See \Cref{fig:birkhoffsection}. Since $N$ is compact, we can take the union of finitely many such surfaces and have all points meet the union in finite forward time. Then we can perform some surgeries along the self-intersections to get a genuine Birkhoff section $\Sigma$. 

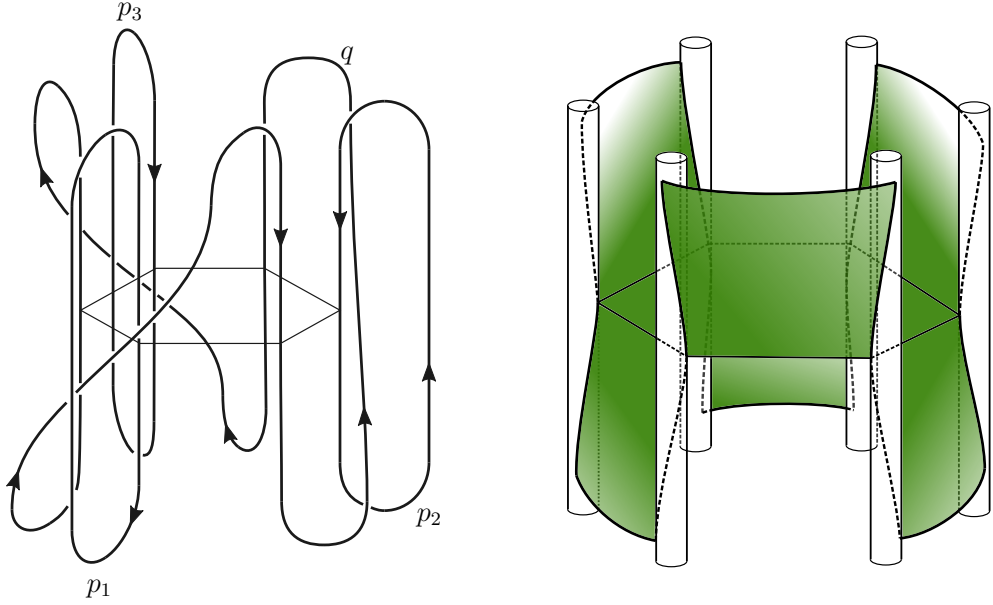
\begin{figure}
    \centering
    \fontsize{12pt}{12pt}\selectfont
    \resizebox{!}{8cm}{
\begingroup%
  \makeatletter%
  \providecommand\color[2][]{%
    \errmessage{(Inkscape) Color is used for the text in Inkscape, but the package 'color.sty' is not loaded}%
    \renewcommand\color[2][]{}%
  }%
  \providecommand\transparent[1]{%
    \errmessage{(Inkscape) Transparency is used (non-zero) for the text in Inkscape, but the package 'transparent.sty' is not loaded}%
    \renewcommand\transparent[1]{}%
  }%
  \providecommand\rotatebox[2]{#2}%
  \newcommand*\fsize{\dimexpr\f@size pt\relax}%
  \newcommand*\lineheight[1]{\fontsize{\fsize}{#1\fsize}\selectfont}%
  \ifx\svgwidth\undefined%
    \setlength{\unitlength}{411.26611671bp}%
    \ifx\svgscale\undefined%
      \relax%
    \else%
      \setlength{\unitlength}{\unitlength * \real{\svgscale}}%
    \fi%
  \else%
    \setlength{\unitlength}{\svgwidth}%
  \fi%
  \global\let\svgwidth\undefined%
  \global\let\svgscale\undefined%
  \makeatother%
  \begin{picture}(1,0.6113576)%
    \lineheight{1}%
    \setlength\tabcolsep{0pt}%
    \put(0,0){\includegraphics[width=\unitlength,page=1]{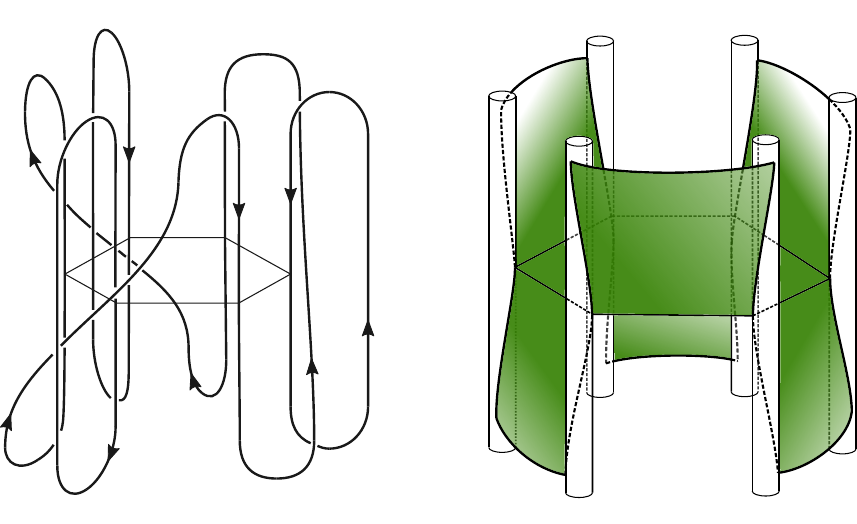}}%
    \put(0.08195817,0.00523011){\color[rgb]{0,0,0}\makebox(0,0)[lt]{\lineheight{1.25}\smash{\begin{tabular}[t]{l}$p_1$\end{tabular}}}}%
    \put(0.41711623,0.07718957){\color[rgb]{0,0,0}\makebox(0,0)[lt]{\lineheight{1.25}\smash{\begin{tabular}[t]{l}$p_2$\end{tabular}}}}%
    \put(0.11329534,0.59135114){\color[rgb]{0,0,0}\makebox(0,0)[lt]{\lineheight{1.25}\smash{\begin{tabular}[t]{l}$p_3$\end{tabular}}}}%
    \put(0.33961942,0.54724695){\color[rgb]{0,0,0}\makebox(0,0)[lt]{\lineheight{1.25}\smash{\begin{tabular}[t]{l}$q$\end{tabular}}}}%
  \end{picture}%
\endgroup%
}
    \caption{Left: pick a long periodic orbit $q$ which weaves around the shorter periodic orbits $p_1,...,p_n$. Right: take a $2n$-gon transverse to the flow with vertices at $p_1,q,...,p_n,q$ and connect up its sides to produce an immersed local Birkhoff section.}
    \label{fig:birkhoffsection}
\end{figure}

Our argument here is a generalization of Fried's argument in \cite{Fri83} to pseudo-Anosov flows. We remark that there is an alternative argument to construct $q$ in \cite{Bru95} without using symbolic dynamics.We can further assume that $p_i$ are orientation preserving paths on the stable and unstable leaves they lie on, since the set of orientation preserving periodic orbits is dense as well.

Let $\Sigma^\circ$ be $\Sigma$ without its boundary components, and let $\overline{\Sigma}$ be $\Sigma$ with all its boundary components collapsed to points. The first return map on $\Sigma$ restricts to a pseudo-Anosov homeomorphism $h^\circ$ on $\Sigma^\circ$, with the intersections of $\Lambda^s, \Lambda^u$ with $\Sigma^\circ$ acting as the stable and unstable foliations. This in turns induces a pseudo-Anosov homeomorphism $\overline{h}$ on $\overline{\Sigma}$. (Note that a map obtained this way on a general Birkhoff section may have 1-pronged singularities, but this is eliminated by requiring that the components of $\partial \Sigma$ lie on orientable stable and unstable leaves.) By classical results, there is a smooth structure on $\overline{\Sigma}$ such that $\overline{h}$ is smooth away from the singular points. Hence the mapping torus of $(\bar{\Sigma}, \bar{h})$ carries a smooth structure so that the suspension flow is smooth away from the singular points. Moreover, for each closed orbit of the suspension flow, there is a neighborhood diffeomorphic to a neighborhood of the pseudo-hyperbolic orbit of some $\Phi_{n,k,\lambda}$ in \Cref{defn:phorbit}. In particular this is true for the closed orbits which are suspensions of points of $\bar{\Sigma} \backslash \Sigma^\circ$. We call these closed orbits the surgery orbits for ease of notation.

Meanwhile, note that $N \backslash \partial \Sigma$ with the restricted $\phi^t$ flow is $C^0$-orbit equivalent to the suspension flow on the mapping torus $(\Sigma^\circ, h^\circ)$. In particular, there is a vector field on $N \backslash \partial \Sigma$ (with respect to some smooth structure) whose trajectories are equal to the flow lines of $\phi^t$. Using the orbit equivalence, we can transfer neighborhoods of the surgery orbits, minus the surgery orbits themselves, to neighborhoods of the ends of $N \backslash \partial \Sigma$, then fill back in the components of $\partial \Sigma$ to get neighborhoods of $\partial \Sigma$ in $N$. It remains to perform Dehn surgery on these neighborhoods, which we will do by cutting out small special shaped sub-neighborhoods and gluing similarly shaped regions back in, and arguing that the result is orbit equivalent to $\phi^t$ on the original manifold $N$. The first step is a restatement of the arguments in \cite{Goo83}, while the second step relies on \cite[Theorem 3.9]{Sha21}. 

Within each such neighborhood in $N$, locate and excise a small \textit{cross-shaped neighborhood} of $\partial \Sigma$ using the orbit equivalence map between $N \backslash \partial \Sigma$ and the mapping torus of $(\Sigma^\circ, h^\circ)$. Here we provide an illustration of these cross-shaped neighborhoods in \Cref{fig:crossnbd} and refer to \cite[Section 5.2]{Sha21} for their precise definition. The key feature of these neighborhoods is that their boundaries have alternating annulus faces, the flow transverse to half of the faces and tangent to the remaining faces. 

\begin{figure}
    \centering
    \resizebox{!}{8cm}{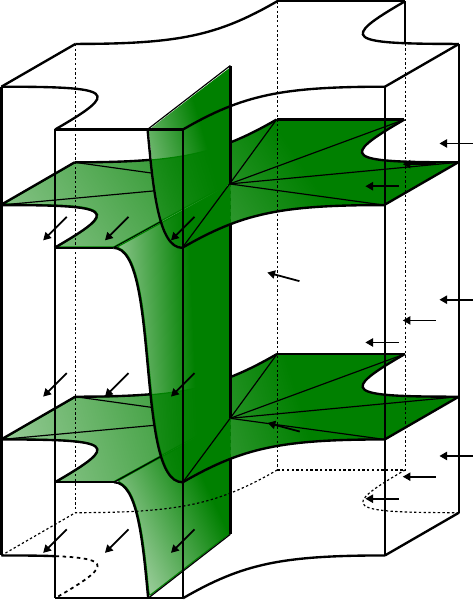}
    \caption{A cross-shaped neighborhood and the portion of the Birkhoff section constructed in it.}
    \label{fig:crossnbd}
\end{figure}

Then we glue some cross-shaped neighborhoods back in by matching up the vector fields at the boundary, so that the glued vector field determines a smooth flow (away from the singular orbits). See \cite[Section 5.3]{Sha21} for an explicit description of this gluing map. The key observation that makes this work is that there is a freedom in this gluing: one can do Dehn twists suitably along annulus faces which the flow is transverse to, and still have the vector fields match up. So far there is no difference between the Anosov case and the general pseudo-Anosov case. 

Then one uses a cone field argument to find the stable and unstable line bundles $E^s, E^u$ and show that the glued up flow is pseudo-Anosov, provided that the neighborhood we excised is small enough. This is done in the Anosov case in \cite[Section 5.4]{Sha21}, and is in fact more simple in the pseudo-Anosov case, since one only needs to construct $E^s, E^u$ away from the singular orbits. On the singular orbits, the definition only requires an orbit equivalence between a neighborhood in $N$ and a neighborhood in $\Phi_{n,k,\lambda}$, and these are provided by the orbit equivalence between $N \backslash \partial \Sigma$ and the mapping torus of $(\Sigma^\circ, h^\circ)$, since by construction the orbits of $\partial \Sigma$ are nonsingular. As remarked above, this method of constructing (pseudo-)Anosov flows on Dehn surgeries appeared back in \cite{Goo83}. 

Finally, we construct a Birkhoff section on this glued flow by taking the portion of the original Birkhoff section $\Sigma$ outside of the excised cross-shaped neighborhood, and extending it inside the glued cross-shaped neighborhood by taking a union of straight lines (with respect to the Euclidean metric on a canonical set of coordinates) towards the pseudo-hyperbolic orbit. Call this surface $\Sigma'$. We illustrate the portion of $\Sigma'$ within each cross-shaped neighborhood in \Cref{fig:crossnbd} and refer to \cite[Section 5.5]{Sha21} for precise formulas, where it is also shown that this indeed defines a Birkhoff section, and that the return map $h'$ is pseudo-Anosov. Moreover, $h$ and $h'$ induce the same automorphism on $\pi_1(\Sigma) \cong \pi_1(\Sigma')$, preserving the peripheral subgroups. Hence using \cite[Theorem 3.9]{Sha21}, which applies to pseudo-Anosov flows as well, as stated there, we get a $C^0$-orbit equivalence between $\phi^t$ on $N$ and the smooth pseudo-Anosov glued flow. 

\subsection{Properties of the construction} \label{subsec:prop}

We point out two properties of the pseudo-Anosov flows we constructed. We will phrase our proofs in terms of the initial topological pseudo-Anosov flow we constructed, but thanks to \Cref{thm:top2smooth}, one can just take an orbit equivalence map to prove the properties for the smooth pseudo-Anosov flow as well.

\begin{prop} \label{prop:unstablelamination}
The unstable lamination of the pseudo-Anosov flow, obtained by blowing air into the leaves of the unstable foliation that contain the core orbits, is carried by $B \subset M(s)$.
\end{prop}
\begin{proof}
The unstable lamination can also be obtained by taking the unstable foliation on $N(B)$ and blowing air into the non-manifold leaves. There is a projection $N(B) \to B$ by definition, and this projects the lamination to $B$.
\end{proof}

\begin{defn}
Suppose a pseudo-Anosov flow $\phi$ has a Markov partition as in \Cref{defn:pAflow}. Define a directed graph $G$ by letting the set of vertices be the flow boxes, and putting an edge from $(I^{(j)}_s \times I^{(j)}_u \times [0,1]_t)$ to $(I^{(i)}_s \times I^{(i)}_u \times [0,1]_t)$ for every $J^{(ij,k)}_s$. 

Notice that $G$ has a natural embedding in $N$ by placing the vertices in the interior of the corresponding flow box and placing the edges through the corresponding intersections $J^{(ij,k)}_s \times I^{(i)}_u \times \{1\}$. $G$ together with this embedding in $N$ is said to \textit{encode the Markov partition}. 
\end{defn}

\begin{prop} \label{prop:Markovpart}
The pseudo-Anosov flow constructed in \Cref{thm:vtpAflow} admits a Markov partition encoded by the reduced flow graph $\Phi_{\red}$.
\end{prop}
\begin{proof}
The flow boxes $Z_e$ form a Markov partition by construction, but notice that the graph that encodes this Markov partition is not $\Phi_{\red}$. Indeed, the flow boxes $Z_e$ correspond to edges, not vertices, of $\Phi_{\red}$, hence the graph encoding this Markov partition is instead a `dual' of $\Phi_{\red}$.

What we will do is extract from this an alternate Markov partition that is encoded by $\Phi_{\red}$. For each vertex $v$ of $\Phi_{\red}$, recall the rectangle $R_v$ considered in \Cref{subsec:stableglue}. Let $Z'_v$ be the closure of the union of trajectories in $M(s)$ that start in the image of the interior of $R_v$ and end when it meets the image of the interior of another (possibly the same) $R_{v'}$. These $Z'_v$ cover $M(s)$ with disjoint interiors, but they might not be flow boxes. Instead, they are in general homeomorphic to $\{(s,u,t): s \in I_s, u \in I_u, t \in [f(s,u),1]\}$ with a function $f$ with discontinuities along finitely many lines $u=u_0$. Each of these intervals of discontinuity correspond to an adjacent pair of flow boxes leaving $R_v$.

Consider the total number of intervals of discontinuity at the bottom of all the $Z'_v$. If this number is $0$, then we have a genuine flow box decomposition which gives a Markov partition encoded by $\Phi_{\red}$. 

If this number is positive, we claim that we can modify the $Z'_v$ to reduce it. Fix an interval of discontinuity $J$ of $Z'_v$. Suppose the two sides to it on the bottom of the $Z'_v$ lie on the image of $R_{v_1}$ and $R_{v_2}$. We can isotope $R_{v_1}$ near $J$, moving each point along a trajectory of the flow, so that $R_{v_1}$ matches up with $R_{v_2}$ along $J$, eliminating the discontinuity. See \Cref{fig:markovpartarg}. 

One might worry that this produces new intervals of discontinuity, but we argue that this will not happen. Using the same notation as above, since we only modify $R_{v_1}$ near $J$, the operation only possibly creates new intervals of discontinuity that lie on the same sides of $R_{v_1}$ and $R_{v_2}$ as $J$ and share an endpoint with $J$. Suppose $K$ is such an interval. We illustrate one example of such a $K$ in \Cref{fig:markovpartarg}. $J \cap K$ is a point $x$ that lies on the image of a component $c$ of the branch locus of $N(B)$ that passes through $R_{v_1}$ and $R_{v_2}$. In particular, $R_{v_1}$ and $R_{v_2}$ match up along $x$ already. (In \Cref{fig:markovpartarg}, we illustrate $c$ as the yellow line; compare with \Cref{fig:stableglue3} top.) So we can in fact isotope $R_{v_1}$ while fixing $K$, ensuring that we will not create a new interval of discontinuity at $K$. Doing this inductively, we can reduce the number of intervals of discontinuity to $0$.
\end{proof}

\begin{figure}
    \centering
    \fontsize{12pt}{12pt}\selectfont
    \resizebox{!}{8cm}{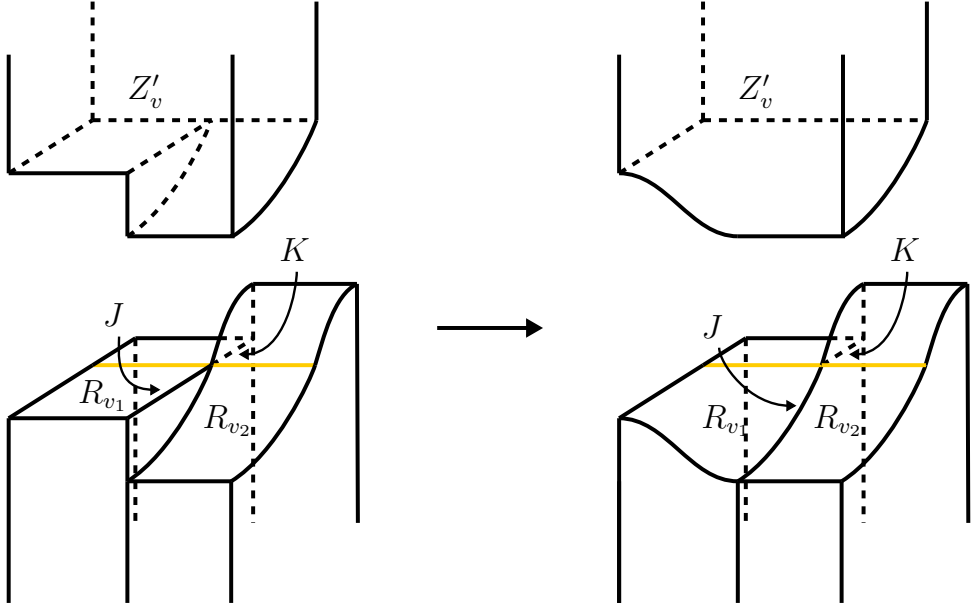}
    \caption{Modifying $R_v$ to eliminate the intervals of discontinuity, so that we obtain a decomposition by flow boxes.}
    \label{fig:markovpartarg}
\end{figure}

\begin{cor} \label{cor:symbdyn}
Given a cycle $l$ of $\Phi$ in $M(s)$, there is a periodic orbit $c$ of the flow such that $l$ is homotopic to $c$. Conversely, given a periodic orbit $c$ of the flow, there exists a cycle $l$ of $\Phi$ such that $l$ is homotopic to some multiple $c^k$, $k \geq 1$.
\end{cor}
\begin{proof}
This is mostly a straightforward consequence of \Cref{prop:Markovpart} and \Cref{lemma:flowgraphsorbits}.

Given a cycle $l$ of $\Phi$, we first assume that there is a homotopic cycle of $\Phi_{\red}$, which we also write as $l$. Then we can proceed as in the proof of transitivity to construct orbits of the flow passing through the sequence of flow boxes corresponding to the sequence of vertices of $\Phi_{\red}$ passed through by $l'$. A similar argument as in that proof shows that such an orbit will be unique, hence the orbit $c$ will be periodic, since the sequence of flow boxes is periodic. A homotopy between $l$ and $c$ can now be constructed using a straight line homotopy within each flow box.

We caution that $c$ may not be a primitive closed orbit. Indeed, the above construction in fact gives a $S^1$ leaf $c'$ of the singular 1-dimensional foliation on $N(\Phi_{\red})$. But when we quotient $N(\Phi_{\red})$ to $M(s)$, $c'$ may cover an orbit multiple times.

If $l$ is not homotopic to a cycle of $\Phi_{\red}$, then by the proof of \Cref{lemma:flowgraphsorbits}, $l$ must be an infinitesimal cycle within a twisted wall. Let $l'$ be the boundary cycle of that wall. Applying the argument above for this $l'$, we find a circular leaf $c'$ in $N(\Phi_{\red})$ lying on the stable boundary $\partial_s N(\Phi_{\red})$ which is homotopic to $l'$ in $M(s)$. When we quotient $N(\Phi_{\red})$ to $N(B)$, $c'$ double covers its image $c''$, which is hence homotopic to $l$. The image of $c''$ after quotienting to $M(s)$ will be the desired periodic orbit $c$ in this case. Again, we remark that $c$ may not be primitive.

Conversely, given a periodic orbit $c$ of the flow on $M(s)$, consider its preimage in $N(\Phi_{\red})$. This will be a disjoint union of circular leaves which cover $c$ possibly multiple times. Pick one of these circles $c'$. The periodic sequence of flow boxes which it meets will determine a cycle $l$ of $\Phi_{\red} \subset \Phi$. A homotopy between $l$ and $c$ can now be constructed using a straight line homotopy within each box as before.
\end{proof}

Combining \cite[Theorem A]{Lan20} and \cite[Proposition 5.7]{LMT20}, it is known that when $|\langle s,l \rangle| \geq 3$, the homology classes of the cycles of $\Phi$ in $M(s)$ span a cone dual to a face of the Thurston unit ball, namely the face determined by the negative of the Euler class of the veering triangulation $$-e(\Delta)=\sum \frac{1}{2}(\text{\# prongs}(\gamma)-2) \cdot [\gamma] \in H_1(M(s))$$ Hence \Cref{cor:symbdyn} implies that the homology classes of the periodic orbits of the pseudo-Anosov flow in \Cref{thm:vtpAflow} span the same dual cone. 

\begin{rmk} \label{rmk:LMTorbits}
\Cref{cor:symbdyn} and its proof also allows one to count the number of periodic orbits of the pseudo-Anosov flow using $\Phi$. The precise statements one can get are essentially the same as those shown in \cite[Theorem 6.1]{LMT21}, but again we remind the reader that our setting in this paper is opposite to that in \cite{LMT21}.

We elaborate on this a little more. The proof of \Cref{cor:symbdyn} shows that for each cycle $l$ of $\Phi$, there is a unique primitive periodic orbit $c$ of the flow homotopic to $l$, unless if $l$ determines a circular leaf that lies on the boundary of $N(\Phi_{\red})$ and multiple covers a periodic orbit of the flow under the quotient maps. If such a circular leaf lies on $\partial^s N(\Phi_{\red}) \backslash \backslash \partial^u N(\Phi_{\red})$, then $l$ lies on the boundary of a component of $B \backslash \backslash \Phi_{\red}$, and if the circular leaf multiple covers a periodic orbit, the component must be a M\"obius band with tongues. If such a circular leaf lies on $\partial^u N(\Phi_{\red})$, then $l$ can be homotoped into a torus end of $M$.

Conversely, the proofs of \Cref{cor:symbdyn} and \Cref{lemma:flowgraphsorbits} show that for each primitive periodic orbit $c$ of the flow, there is a unique cycle $l$ of $\Phi$ homotopic to $c$, unless the preimage of $c$ in $N(\Phi_{\red})$ multiple covers $c$ or if $l$ is homotopic to a boundary cycle of a wall. The former can only happen when if the preimage lies in the boundary of $N(\Phi_{\red})$. If the preimage lies in $\partial^s N(\Phi_{\red}) \backslash \backslash \partial^u N(\Phi_{\red})$, then $c$ is homotopic to the core of a component of $B \backslash \backslash \Phi_{\red}$ which is a M\"obius band with tongues as above. If the preimage lies in $\partial^u N(\Phi_{\red})$ then $c$ is a core orbit. In the latter case, $l$ is homotopic to an AB cycle, by the discussion in \Cref{rmk:LMTABwalls}.
\end{rmk}

\section{Discussion and further questions} \label{sec:questions}

We conclude with some questions coming out of this paper.

In \Cref{sec:infcomp}, we analyzed \emph{how} the flow graph of a veering triangulation can fail to be strongly connected. An equally interesting question is \emph{when} does the flow graph fail to be strongly connected. For example, for layered veering triangulations, one could ask if there is a topological criterion on the monodromy that determines whether the flow graph is strongly connected. Another interesting question is if having a strongly connected flow graph is a generic property among all veering triangulations.

We already pointed out in \Cref{sec:finiteness} that in our argument for \Cref{thm:boundveertet}, there is a lot of room for improvement for the bounds we used. A more interesting approach to try to obtain better bounds overall is to study the flow graphs themselves. It is shown in \cite{McM15} how to recover the dilatation of the pseudo-Anosov monodromy from the Teichm\"uller polynomial of the associated fibered face. The Teichm\"uller polynomial is in turn related to the veering polynomial of the layered veering triangulation, defined in \cite{LMT20} as the inclusion of the Perron polynomial of the flow graph in the 3-manifold. Now, the combinatorial rigidity of veering triangulations might impose restrictions on the graph isomorphism type of flow graphs, which might in turn have consequences for the behavior of their Perron polynomials, possibly giving some bounds on dilatations. We remark that Parlak has computed some data for veering polynomials of veering triangulations in the census (see \cite{Par21}). This data might be helpful for extrapolating patterns for the approach above.

One motivation for improving the bound is to use \Cref{thm:boundveertet} to solve the minimal dilatation problem by computation. Starting with an upper bound, we can use \Cref{thm:boundveertet} to reduce the search to a finite list. However, it has been conjectured that one cannot do better than $P = \left( \frac{3+\sqrt{5}}{2} \right)^2$ (see \cite{Hir10}, where examples that asymptotically attain this bound are also shown). Plugging this into our \Cref{thm:boundveertet}, we have to look at all 3-manifolds triangulated by at most $9.7 \times 10^{15}$ veering tetrahedra. In comparison, the veering triangulation census only covers triangulations up to 16 veering tetrahedra for now, so huge advancements would have to be made before this approach is computationally feasible.

Finally, there are plenty of natural questions one could ask about our construction of the pseudo-Anosov flow in \Cref{sec:pAflow}.

We have been using the unstable branched surface $B$ throughout this paper, but there is also a stable branched surface associated to a veering triangulation, which can be defined as the unstable branched surface associated to the veering triangulation with reversed coorientation on the faces. It is natural to ask whether the stable lamination of our pseudo-Anosov flow is carried by the stable branched surface. 

Also, one can repeat our entire construction with the stable branched surface replacing the unstable branched surface. Then one can ask whether the pseudo-Anosov flows obtained using the stable and unstable branched surfaces are $C^0$-orbit equivalent. A related development is that Parlak (\cite{Par21}) has found veering triangulations whose veering polynomial is different from that of the veering triangulation with reversed coorientation on the faces.

We mentioned in the introduction that the construction of a pseudo-Anosov flow without perfect fits from a veering triangulation was first done by Schleimer and Segerman. It is an interesting question to ask whether the flow constructed by our approach ends up being equivalent to the one constructed by Schleimer and Segerman.

One could also ask if our construction is inverse, in either direction, to the construction by Agol and Gu\'eritaud, which produces a veering triangulation out of a pseudo-Anosov flow without perfect fits. More precisely, if one starts with a pseudo-Anosov flow without perfect fits relative to orbits $\{c_i\}$ on a closed orientable 3-manifold $N$ and constructs a veering triangulation on $N$ with $\{c_i\}$ drilled out, then constructs a pseudo-Anosov flow by filling in slopes that recover $N$, is the final pseudo-Anosov flow $C^0$-orbit equivalent to the initial pseudo-Anosov flow? Conversely, if one starts with a veering triangulation and constructs a pseudo-Anosov flow by filling along slopes $s$ with $| \langle s,l \rangle | \geq 2$, then constructs a veering triangulation with the core orbits drilled out, is the final veering triangulation isomorphic to the initial veering triangulation? If one could show that our construction is equivalent to that of Schleimer and Segerman, then the answer to both questions would be `yes', since as mentioned in the introduction, this is a feature of their construction.

\bibliographystyle{alpha}
\bibliography{bib.bib}

\end{document}